\documentclass[11pt]{amsart}

\usepackage{hyperref}
\usepackage{amsmath} 
\usepackage{amsthm}
\usepackage{amssymb}
\usepackage{color}
\usepackage{tikz-cd}
\usepackage{adjustbox}
\usepackage[margin=1.25in]{geometry}
\usepackage{todonotes}
\usepackage{caption}
\usepackage{enumitem}
\usepackage{appendix}
\usepackage{float}
\usepackage{array}
\usepackage{longtable}
\usepackage{tocvsec2}

\DeclareUnicodeCharacter{00F7}{$\div$}

\newtheorem{theorem}{Theorem}[section]
\newtheorem{lemma}[theorem]{Lemma}
\newtheorem{proposition}[theorem]{Proposition}
\newtheorem{algorithm}[theorem]{Algorithm}
\newtheorem{corollary}[theorem]{Corollary}

\theoremstyle{definition}

\newtheorem{remark}[theorem]{Remark}

\DeclareMathOperator{\C}{\mathcal{C}}
\DeclareMathOperator{\GC}{\mathcal{G}(\mathcal{C})}
\DeclareMathOperator{\sizeGC}{|\mathcal{G}(\mathcal{C})|}
\DeclareMathOperator{\Cpt}{\mathcal{C}_\mathrm{pt}}
\DeclareMathOperator{\Cad}{\mathcal{C}_\mathrm{ad}}
\DeclareMathOperator{\Cadpt}{(\mathcal{C}_\mathrm{ad})_\mathrm{pt}}
\DeclareMathOperator{\FPdim}{FPdim}

\DeclareMathOperator{\Hom}{Hom}
\DeclareMathOperator{\End}{End}

\DeclareMathOperator{\Rep}{Rep}
\DeclareMathOperator{\rank}{rank}
\DeclareMathOperator{\di}{d}
\DeclareMathOperator{\Vect}{Vec}
\newcommand\VecGw{\operatorname{\textbf{Vec}_G^{\omega}}}
\newcommand{\deeqZ}[1]{(\Cad)_{\Z_#1}}

\DeclareMathOperator{\GCad}{\mathcal{G}(\Cad)}
\DeclareMathOperator{\sizeGCad}{|\mathcal{G}(\Cad)|}
\DeclareMathOperator{\Z}{\mathbb Z}

\title{Classification of low-rank odd-dimensional modular categories}

\begin{document}

\author[A. Czenky]{Agustina Czenky}
\address{Department of Mathematics, University of Oregon, Eugene, OR 97403, USA}
\email{aczenky@uoregon.edu}

\author[W. Gvozdjak]{William Gvozdjak}
\address{Bellevue High School, Bellevue, WA 98004, USA}
\email{w.gvoz@outlook.com}

\author[J. Plavnik]{Julia Plavnik}
\address{\parbox{\linewidth}{Department of Mathematics, Indiana University, Bloomington, IN 47405, USA\\ 
Fachbereich Mathematik, Universit\"at Hamburg, Hamburg 20146, Germany}}
\email{jplavnik@iu.edu}

\begin{abstract}
 We prove that any odd-dimensional modular category of rank at most $23$ is pointed. We also show that an odd-dimensional modular category of rank $25$ is either pointed, perfect, or equivalent to $\Rep(D^\omega(\mathbb Z_7\rtimes \mathbb Z_3))$.  Finally, we give partial classification results for modular categories of rank up to $73$.
\end{abstract}

\maketitle

\tableofcontents

\section{Introduction}
Modular tensor categories (MTCs) are fusion categories with additional braiding and ribbon structures  satisfying a non-degeneracy condition for the braiding, see~\cite{ENO1, T}. They are of interest for a variety of mathematical subjects, such as topological quantum field theory~\cite{T}, representation theory of quantum groups~\cite{BK}, von Neumann algebras~\cite{EK}, conformal field theory~\cite{MS} and vertex operator algebras~\cite{H}. Furthermore, MTCs appear in the study of topological phases of matter and topological quantum computation, as anyon systems are modeled by  unitary MTCs~\cite{R}.

The problem of classifying MTCs is an active area of research. In~\cite{BNRW1}, it was shown that there are finitely many MTCs of a fixed rank, up to equivalence. This result makes the classification of MTCs by rank a more feasible endeavor and many efforts have been made in this direction. A classification of unitary MTCs of rank at most 4 was presented in~\cite{RSW}. In~\cite{HR}, all MTCs of rank at most $5$ such that some object is not isomorphic to its dual (non-self-dual) were classified. Later on, all possible fusion rules for MTCs of rank 5 were determined in~\cite{BNRW2}, and were used to
describe their classification up to monoidal equivalence. Recently, a classification of rank 6 MTCs up to
modular data was given in~\cite{NRWW}.

Integral MTCs, that is, MTCs such that the Frobenius-Perron dimension of every simple object is an integer, are of particular relevance since they are in correspondence with the categories of representations of modular finite-dimensional semisimple quasi-Hopf algebras. Classification of integral MTCs has been approached from many directions, see for example~\cite{BR, BGHKNNPR, DLD, DN,DT}. Recently,~\cite{ABPP} reported the classification of all integral MTCs of rank at most $12$. 

An important class of integral MTCs are those of odd Frobenius-Perron dimension. As a consequence of~\cite[Corollary 8.2]{NS} and~\cite[Theorem 2.2]{HR}, this condition is equivalent to asking that no non-trivial object is self-dual (i.e., the category is maximally non-self-dual, or MNSD). It was proven in~\cite{BR} that odd-dimensional MTCs of rank at most $11$ are pointed, that is, all simple objects are invertible. At the same time, they exhibited an example of an odd-dimensional MTC of rank $25$ that is not pointed, given by the representation category $\Rep (D^\omega(\mathbb Z_7\rtimes \mathbb Z_3))$, and asked whether this is the smallest (in rank) example of a non-pointed odd-dimensional MTC. A partial answer to this question is given in~\cite{CP}, where it was shown that all odd-dimensional MTCs of rank $13$ or $15$ must be pointed, and all odd-dimensional MTCs of rank between $17$ and $23$ are either pointed or perfect, the latter meaning that the unit is the only invertible object.

In this work, we continue the study of low-rank MNSD MTCs, or equivalently, low-rank odd-dimensional MTCs. We show that odd-dimensional MTCs of ranks 17 to 23 are all pointed, which gives a positive answer to the question in~\cite{BR}. A proof of this statement can be found in Section~\ref{section: ranks 17 to 23}. Pointed MTCs are classified by pairs $(G, q)$, where $G$ is a finite abelian group, and $q:G \to \mathbf k^{\times}$ is a non-degenerate quadratic form on G, see~\cite[Example 8.13.5]{EGNO}. Thus our result completes the classification of odd-dimensional MTCs up to rank 23.

We also show that an odd-dimensional MTC of rank 25 is either pointed, perfect, or equivalent to the representation category $\Rep (D^\omega(\mathbb Z_7\rtimes \mathbb Z_3))$, which is the non-pointed example exhibited in~\cite{BR}. For higher ranks, we show partial results for classification of odd-dimensional MTCs of rank at most 73.

The following Theorem summarizes our main results.
\begin{theorem}\label{thm: summary}
    Let $\mathcal C$ be an odd-dimensional MTC.
    \begin{itemize}
        \item If $17\leq \rank(\mathcal C)\leq 23$, then $\mathcal C$ is pointed.
        
        \item If $\rank(\mathcal C)=25$, then $\mathcal C$ is pointed, perfect, or equivalent to $\Rep(D^\omega(\mathbb Z_7\rtimes \mathbb Z_3))$.
        
        \item If $\rank(\C)\in \{27, 29, 31,37,39,45,47,53,55,63\}$, then $\C$ is either pointed or perfect.
        
        \item If $\rank(\C)\in \{33, 61, 69, 71\}$, then $\C$ is either pointed, perfect, or has 3 invertible objects. 
        
        \item If $\rank(\C)=35$, then $\C$ is pointed, perfect, or  the modular subcategory of $\mathcal Z(\operatorname{Vec}_{H_3}^{\omega})$ with 9 invertible objects and 26 simple objects of dimension 3, where $H_3$ denotes the Heisenberg group of order $3^3$.
        
        \item If $\rank(\C)=41$, then $\C$ is pointed, perfect, or has 3 or 5 invertible objects. Moreover, if the latter case exists, then $\C$ must be equivalent to $\mathcal D^{\mathbb Z_5}$, where $\mathcal D$ is a categorification of the ring $R_{5,H}$ as defined in~\cite[Definition 1.3]{JL}, and $H$ is a finite abelian group of order $3^4$.
        
        \item If $\rank(\C)=43$, then $\C$ is pointed, perfect, or has $9$ invertible objects.
        
        \item If $\rank(\C)=49$, then $\C$ is pointed, perfect, or has $3$ or $5$ invertible objects. Additionally, if $\C$ has $5$ invertible objects and $7\nmid\FPdim(\C)$, then $\C\cong\Rep(D^\omega(\Z_{11}\rtimes\Z_5))$ with $\Z_{11}\rtimes\Z_5$ non-abelian and for some $3$-cocycle $\omega$.
        
        \item If $\rank(\C)\in \{51, 59, 67\}$, then $\C$ is either pointed, perfect or has 3 or 9 invertible objects.
        
        \item If $\rank(\C)\in \{57, 65\}$, then $\C$ is either pointed, perfect or has 3 or 5 invertible objects. In particular, the category $\Rep(D^\omega(\Z_{13}\rtimes\Z_3))$ is an example of an odd-dimensional MTC of rank $65$  with $3$ invertible objects. 
        
        \item If $\rank(\C)=73$, then $\C$ is either pointed, perfect or has 3, 5 or 7 invertible objects. 
    \end{itemize}
\end{theorem}

\begin{remark}
We note that if $\C$ is an odd-dimensional MTC that can be obtained as the Drinfeld center of another fusion category, then its dimension is a square and thus $\FPdim(\C)\equiv 1 \mod 8$. In particular, this is the case for MTCs of the form $\Rep(D^{\omega}G)$ for some finite group $G$. 

We have included in Theorem \ref{thm: summary} all examples of the form $\Rep(D^{\omega}G)$  (that we know of) that appear up to rank 73. There could be more examples in ranks 33, 41, 49, 57, 65 or 73, as it is difficult to know in general the rank of $\Rep(D^{\omega}G)$ in relation to $|G|$, see for example \cite{E}. Possibilities for such ranks have been computed up to $|G|=47$ in \cite[Figure 3]{GS}.
\end{remark}

Regarding the methods used in this paper, we started by proving several general results useful for classification of low rank MTCs, which can be found in Section~\ref{sec: general results}. We also developed an algorithm based on techniques from~\cite{CP}, which we improved with results from Section~\ref{sec: general results}, that allowed us to inspect potential arrays for the Frobenius-Perron dimensions of simple objects in an odd-dimensional MTC of fixed rank. Another important tool used in this work was the de-equivariantization of the adjoint subcategory by a Tannakian subcategory, which permitted in certain cases to work with an odd-dimensional MTC of lower dimension (and sometimes lower rank) than the original one.

There is a major conjecture in fusion categories that states that every weakly integral fusion category is weakly group-theoretical~\cite[Question 2]{ENO2}. As stated in~\cite{CP}, a perfect odd-dimensional MTC would yield a counter-example for said conjecture, and is thus conjectured that such a category cannot exist~\cite[Conjecture 1.1]{CP}. Consequently, if 
weakly integral MTCs are proven to be weakly group-theoretical, our results above would give a complete classification of odd-dimensional MTCs $\C$ such that 
\begin{itemize}
    \item $\rank(\C)\leq 47$ and $\rank(\C)\ne 33, 41, 43$, or
    \item $\rank(\C)=53, 55$ or 63.
\end{itemize}

The rest of this paper is organized as follows. A brief introduction to modular tensor categories and related constructions is given in Section~\ref{section: preliminaries}. The algorithm used for computing potential arrays of Frobenius-Perron dimensions is explained in Section~\ref{section: algorithm}, as well as general results useful for the classification of odd-dimensional MTCs by rank. Section~\ref{section: ranks 17 to 23} contains the classification of odd-dimensional MTCs of rank 17 to 23. Lastly, Sections~\ref{sec: rank 25},~\ref{sec: rank 27-31},~\ref{sec: ranks 33-49} and~\ref{sec: ranks 51-73} are dedicated to advancing the classification of MTCs of ranks 25, 27 to 31, 33 to 49, and 51 to 73, respectively.  

	\settocdepth{part}
\section*{Acknowledgements}
	\settocdepth{subsection}
We thank the MIT PRIMES program and its organizers Prof.~Pavel Etingof, Dr.~Slava Gerovitch, and Dr.~Tanya Khovanova for making this possible. The research of J.P. was partially supported by NSF grant DMS-2146392 and by Simons Foundation Award 889000 as part of the Simons Collaboration on Global Categorical Symmetries. J.P. would like to thank the hospitality
and excellent working conditions at the Department of Mathematics at Universit\"at Hamburg, where she has carried out part of this research as an Experienced Fellow of the Alexander von Humboldt Foundation.

\section{Preliminaries}\label{section: preliminaries}
We work over an algebraically closed field \textbf k of characteristic zero. We refer the reader to~\cite{ENO1, ENO2, EGNO} for the basic theory of fusion categories and braided fusion categories, and for terminology used throughout this paper. 

We denote by $\operatorname{Vec}$ the category of finite dimensional vector spaces over $\textbf k$, and by $\Rep(G)$ the category of finite dimensional representations of a finite group $G$.

\subsection{Fusion categories}
A \emph{fusion category} over \textbf k is a semisimple rigid tensor category over \textbf k with finitely many isomorphism classes of simple objects.

For the rest of this section, let $\mathcal C$ be a fusion category. We denote by \textbf 1 its identity object, and assume $\End_{\mathcal C}(\textbf 1)\cong \mathbf k$.

For an object $X\in \mathcal C$, we denote its dual by $X^*\in \mathcal C$. We say an object is $\emph{self-dual}$ if $X\cong X^*$, and \emph{non-self-dual} if $X\not\cong X^*$. The unit object \textbf 1 is always self-dual, and so we say that $\mathcal C$ is $\emph{maximally-non-self-dual}$ (MNSD) if the only self-dual simple object in $\mathcal C$ is $\mathbf 1$.

 We shall denote by $\mathcal{O(C)}$ the set of isomorphism classes of simple objects of $\mathcal C$. For $X, Y $ and $Z$ in $\mathcal O(\mathcal C)$, we denote the \emph{fusion coefficients} by $$N^X_{Y,Z} := \dim \Hom(Y\otimes Z, X).$$

A \emph{braiding} on  $\mathcal C$ is a natural isomorphism $$c_{X,Y}: X\otimes Y \to Y \otimes X$$ satisfying the so-called hexagonal diagrams, see~\cite[Definition 8.1.1]{EGNO}. We will say that $\mathcal C$ is \emph{braided} if it is equipped with a braiding.

\subsubsection{Frobenius-Perron dimension} Let $\mathcal C$ be a fusion category. We will denote by $\mathcal{K(C)}$ its Grothendieck ring, see e.g.~\cite[4.5]{EGNO}. For an object $X$ in $\mathcal C$ we will use the same notation for its class $X$  in $\mathcal{K(C)}$. We recall that there is a unique ring homomorphism $\FPdim : \mathcal{K(C)} \to \mathbb R$ called \emph{Frobenius-Perron dimension} such that $\FPdim(X) \geq 1$ for
any object $X \ne 0$, see~\cite[Proposition 3.3.4]{EGNO}. The \emph{Frobenius-Perron dimension} $\FPdim(\mathcal C)$ of $\mathcal C$ is defined as
\begin{align*}
	\FPdim(\mathcal C) = \sum\limits_{X\in \mathcal{O(C)} } \FPdim(X)^2.
\end{align*}

	We say that $\mathcal C$ is \emph{weakly integral} if $\FPdim(\mathcal C)$ is an integer, and \emph{integral} if $\FPdim(X)$ is an integer for all simple objects $X$.

If the non-invertible simple objects of $\C$ are $X_1, X_1^*, \dots, X_k, X_k^*$ and their Frobenius-Perron dimensions are $\di_1, \di_1, \dots, \di_k, \di_k$, then we say that the \emph{dimension array} of $\C$ is the array $[\di_1, \di_2, \dots, \di_k]$.

\subsubsection{Pointed fusion categories}
Let $\mathcal C$ be a fusion category. An object $X$ in $\mathcal C$ is said to be \emph{invertible} if its evaluation $X^* \otimes X \to \mathbf 1$ and coevaluation $\mathbf 1 \to X \otimes X^*$ maps are isomorphisms, see~\cite[Definition 2.10.1]{EGNO}. Equivalently, an object $X$ is invertible if $\FPdim(X)=1$. 

The unit  object $\mathbf 1$ is invertible. Generalizing the notion of a perfect group, which is a finite group with only one representation of dimension 1 (the trivial one), a fusion category $\mathcal C$ is said to be \emph{perfect} if the only invertible object is the unit. 

	A fusion category $\mathcal C$ is \emph{pointed} if all simple objects are invertible. Pointed fusion categories are classified by finite groups in the following way. Any pointed fusion category $\mathcal C$ is equivalent to the category of finite dimensional $G$-graded vector
spaces $\VecGw$, where $G$ is a finite group and $\omega$ is a 3-cocycle on $G$ with coefficients in $\textbf{k}^{\times}$ codifying the associativity constraint. 

We denote the group of isomorphism classes of invertible objects of $\mathcal C$  by $\mathcal{G(C)}$. 
The largest pointed subcategory of $\mathcal C$ will be denoted by $ \mathcal{C}_\mathrm{pt}$, that is the fusion subcategory of $\mathcal{C}$ generated by $\mathcal{G(C)}$. We will often identify the elements in $\mathcal{G}(\mathcal{C})$ with the invertible objects in $\mathcal{C}_\mathrm{pt}$.

A fusion category is said to be \emph{group-theoretical} if it is Morita equivalent to a pointed fusion category, see~\cite[Section 9.7]{EGNO}.

\subsubsection{The universal grading}\label{section: universal grading}

Let $G$ be a finite group. A \emph{$G$-grading} on a fusion category $\mathcal C$ is a decomposition $$\mathcal C = \displaystyle\bigoplus_{g\in G} \mathcal C_g,$$ 
where $\mathcal C_g$ is an abelian subcategory of $\mathcal C$ for all $g\in G$, 
such that the tensor product satifies $\otimes: \mathcal C_g\times \mathcal C_h \to \mathcal C_{gh}$, the unit object $\textbf 1$ is in $\mathcal C_{e}$, and the dualizing functor maps $\mathcal C_g$ to $\mathcal{C}_{g^{-1}}$, see~\cite[Section 4.14]{EGNO}. 

Such grading is said to be \emph{faithful} if $\mathcal C_g\ne 0$ for all $g \in G$. It was shown in~\cite[Proposition 8.20]{ENO2} that for a faithful grading all the components $\mathcal C_g$ have the same Frobenius-Perron dimension. Hence, $$\FPdim(\mathcal C)=|G|\FPdim(\mathcal C_e).$$

 By~\cite{GN}, any fusion category $\mathcal{C}$ admits a canonical faithful grading $$\mathcal{C}= \bigoplus\limits_{g \in U(\mathcal{C})} \mathcal{C}_g,$$  called the \emph{universal grading}. Its trivial component coincides with the \emph{adjoint subcategory}  $\mathcal{C}_{\mathrm{ad}}$ of $\mathcal C$, defined as the fusion subcategory generated by $X\otimes X^*$ for all $X \in \mathcal{O(C)}$.

 If $\mathcal{C}$ is equipped with a braiding, then $\mathcal{U(C)}$ is abelian. Moreover, if $\mathcal C$ is modular then $\mathcal{U(C)}$  is isomorphic to the group of (isomorphism classes of) invertibles $\mathcal{G(C)}$~\cite[Theorem 6.3]{GN}. We will use this fact repeatedly throughout this work.

 We denote the ranks of the components of the universal grading of a fusion category $\C$ using an array of ordered pairs $(n, r)$, indicating that there are $n$ components of rank $r$ each.

\subsubsection{Nilpotent, solvable, and weakly group-theoretical fusion categories}

 The \emph{upper central series} of  a fusion category $\mathcal{C}$ is the sequence defined recursively by
$$\mathcal{C}^{(0)}=\mathcal{C}   \text { and }\mathcal{C}^{(n)}=\big(\mathcal{C}^{(n-1)}\big)_{ad} 
\ \text{ for all } n\geq 1.$$
When the upper central series converges to $\operatorname{Vec}$, the category $\mathcal{C}$ is said to be \emph{nilpotent}, see~\cite{GN, ENO1}. Alternatively, a fusion category is nilpotent if there is a sequence of fusion subcategories $\mathcal C_0=\operatorname{Vec} \subset \dots \subset \mathcal C_n=\mathcal C$ and finite groups $G_1, \dots, G_n$ such that $\mathcal C_i$ is a $G_i$-extension of $\mathcal C_{i-1}$ for all $i$. If, moreover, the groups $G_i$ are cyclic the category is said to be \emph{cyclically nilpotent}.

As is the case for finite groups, fusion categories of Frobenius-Perron dimension a  prime power are known to be nilpotent, see~\cite[Theorem 8.28]{ENO1}. Also, when $\mathcal C$ is a nilpotent fusion category then $\FPdim(X)^2$ divides $\FPdim(\mathcal C_{\mathrm{ad}})$ for all $X\in \mathcal{O(C)}$~\cite{GN}. 

A fusion category $\mathcal{C}$ is \emph{weakly group-theoretical}, respectively \emph{solvable}, if it is Morita equivalent to a nilpotent fusion category, respectively to a cyclically nilpotent fusion category~\cite{ENO1}. Both of these Morita classes are closed under Deligne tensor product, Drinfeld centers, and fusion subcategories, see~\cite[Propositions 4.1 and 4.5]{ENO1}. The Morita class of weakly group-theoretical fusion categories is also closed under group extensions and equivariantizations. On the other hand, the class of solvable categories is closed under extensions and equivariantizations by solvable groups, and by taking component categories of quotient categories.

It was shown in ~\cite[Proposition 4.5]{ENO2} that if  $\mathcal C$ is a solvable braided fusion category, then $\mathcal C \cong \operatorname{Vec}$ or $\mathcal C$ has a non-trivial invertible object. It is also known that braided nilpotent fusion categories are solvable~\cite[Proposition 4.5]{ENO2}.

\subsection{Modular tensor categories}

Let $\mathcal{C}$ be a braided fusion category. A \emph{pivotal structure} on $\mathcal C$ is a natural isomorphism $\psi: \text{Id} \xrightarrow{\sim} (-)^{**}$, i.e., an isomorphism between the double dual and identity functors,  see~\cite{BW, EGNO}.  Associated to a pivotal structure we can define the left and right trace of a morphism $X\to X$, see e.g.~\cite[4.7]{EGNO}. The pivotal structure is called \emph{spherical} if for any such morphism its right trace equals its left trace.

Associated to a pivotal structure $\psi$ we get the notion of \emph{quantum dimension} $\dim(X)$ of an object $X\in \mathcal C$, as the scalar given by 
	\begin{equation}
	\begin{tikzcd}
\dim(X):=\operatorname{Tr}_X(\psi_X)\in \End(1)\cong \textbf k,
	\end{tikzcd}
	\end{equation}
 where $\operatorname{Tr}_X(\psi_X)$ denotes the left trace of $\psi_X$.
	The \emph{quantum dimension} of $\mathcal C$ is given by $\dim(\mathcal C)=\sum\limits_{X\in \mathcal{O(C)}} \di_X \di_{X^*}.$ When $\mathcal C$ is spherical, it satisfies that $\dim(X)=\dim(X^*)$ for all simple objects $X\in \mathcal C$. Note that in this case, $\dim(\mathcal C)= \sum\limits_{X\in \mathcal{O(C)}} \di_X^2$.

 A \emph{pre-modular} tensor category is a braided fusion category equipped with a spherical structure.

Equivalently, a pre-modular category is a braided fusion category endowed with a compatible ribbon structure. A \emph{ribbon structure} on $\mathcal{C}$ is a  natural isomorphism $\theta_X: X\xrightarrow{\sim} X$ for all $X\in \mathcal C$, satisfying
\begin{align}
	\theta_{X\otimes Y}=(\theta_X \otimes \theta_Y) \circ c_{Y,X} \circ c_{X,Y},
\end{align}
and $(\theta_X)^*=\theta_{X^*}$ for all $X ,Y \in \mathcal{C}$.

Let $\mathcal C$ be a premodular tensor category, with braiding $c_{X,Y}:X\otimes Y\xrightarrow{\sim}Y\otimes X$. The \emph{S-matrix} $S$ of $\mathcal C$ is defined by $S:= \left(s_{X,Y}\right)_{X,Y \in \mathcal{O(\mathcal{C})}}$, where $s_{X,Y}$ is the trace of $c_{Y,X}c_{X,Y}:X\otimes Y \to X\otimes Y$. A premodular tensor category $\mathcal C$ is said to be \emph{modular} if its $S$-matrix is non-degenerate.

We can obtain the entries of the $S$-matrix in terms of the twists, fusion rules, and quantum dimensions via the so-called \emph{balancing equation}
	\begin{align}\label{balancing}
	s_{X,Y}=\theta_X^{-1}\theta_Y^{-1} \sum\limits_{Z \in \mathcal{O(\mathcal{C})}} N_{XY}^Z \theta_Z \di_Z,
	\end{align}
for all $X,Y \in \mathcal{O(\mathcal{C})}$~\cite[Proposition 8.13.8]{EGNO}.

Note that, in this work, the terms ``premodular tensor category'' and ``modular tensor category'' imply semisimplicity, as we require them to be fusion categories. This is a slight abuse of terminology, which we adopt to be consistent with prior papers.

\subsubsection{Centralizers}
Let $\mathcal{C}$ be a braided fusion category and let $\mathcal{K}$ be a fusion subcategory of $\mathcal{C}$.  The \emph{M\"{u}ger centralizer} of $\mathcal{K}$ is the fusion subcategory $\mathcal K'$ of $\mathcal C$ consisting of all objects $Y$ in $\mathcal{C}$ such that
\begin{align}\label{centralizador}
	c_{Y,X}c_{X,Y}=\operatorname{id}_{X\otimes Y}, \ \ \text{for all}\ X \in \mathcal{K},
\end{align} 
see~\cite{M}.

We say $\mathcal C$ is \emph{symmetric} if $\mathcal C'=\mathcal C$. A symmetric fusion category is called \emph{Tannakian} if it is equivalent as a braided fusion category to the category $\Rep(G)$ for some finite group $G$, with braiding given by the usual flip of vector spaces. It is known that a symmetric fusion category of odd Frobenius-Perron dimension is Tannakian~\cite[Corollary 2.50]{DGNO2}. We will use this fact repeatedly throughout this work.

A necessary and sufficient condition for $\mathcal{K}$ to be modular is $\mathcal{K} \cap \mathcal{K}' = \ \text{Vec}$. This tells us that $\mathcal C$ is modular if and only if $\mathcal C'=\text{Vec}.$

For a modular tensor category $\mathcal{C}$, any fusion subcategory $\mathcal K$ satisfies $\mathcal{K}=\mathcal{K}''$ and $\dim(\mathcal{K})\dim(\mathcal{K}')=\dim(\mathcal{C})$~\cite[Theorem 3.2]{M}; moreover, $\mathcal{C}_\mathrm{pt}=(\mathcal{C}_{\mathrm{ad}})'$ and  $(\mathcal{C}_\mathrm{pt})'=\mathcal{C}_{\mathrm{ad}}$~\cite[Corollary 6.9]{GN}.

\begin{remark}\cite[Remark 2.2]{CP}
	If $\mathcal{C}$ is a braided fusion category, then $(\mathcal{C}_{\mathrm{ad}})_{\mathrm{pt}}$  is symmetric. 
\end{remark}

\subsubsection{Equivariantization and de-equivariantization}\label{section: equivariantization}
Let $\C$ be a fusion category with an action of a finite group $G$, see~\cite[Definition 4.15.1]{EGNO}.
Following~\cite[Section 2.7]{EGNO} one can construct a fusion category $\mathcal C^G$ of $G$-equivariant objects in $\mathcal C$. We call $\mathcal C^G$ the \emph{equivariantization} of $\mathcal C$ by $G$. If  the action of $G$ on $\mathcal C$ is  braided, see~\cite[Definition 8.23.8]{EGNO}, then $\mathcal C^G$ is also braided, with braiding induced from $\mathcal C$.  Moreover, $\Rep(G)$ is identified with a Tannakian subcategory of $\mathcal C^G$ via the canonical embedding $\Rep(G) \hookrightarrow \mathcal C^G$.

Conversely, suppose that $\mathcal C$ is a braided fusion category with a Tannkian subcategory $\Rep(G)$. Then we can consider the \emph{de-equivariantization} $\mathcal C_G$ of $\mathcal C$ with
respect to $\Rep(G)$, see~\cite[Theorem 8.23.3]{EGNO} for the construction. The category $\mathcal C_G$ is a braided fusion category, and there is an equivalence of braided fusion categories $\mathcal C\cong (\mathcal C_G)^G$.

Note that dimensions are well-behaved under equivariantization and de-equivariantization, see~\cite[Proposition 4.26]{DGNO2}. In fact,  we have 
\begin{align*}
   && \FPdim(\mathcal C^G)=|G|\cdot \FPdim(\mathcal C) &&\text{and} &&\FPdim(\mathcal C_G)=\frac{1}{|G|}\cdot \FPdim(\mathcal C).
\end{align*}

\begin{remark}\cite[Remark 2.3]{ENO2}
	Let $\mathcal C$ be a braided fusion category and let $\mathcal E \subseteq \mathcal C'$ be
	a Tannakian subcategory. Then the de-equivariantization of $\mathcal C$ by $\mathcal E$ is non-degenerate if and only if $\mathcal E = \mathcal C'$.
\end{remark}

Let $\mathcal C$ be an odd-dimensional modular tensor category. By~\cite[Remark 2.2]{CP}, we have that  $(\mathcal{C}_{\mathrm{ad}})_{\mathrm{pt}}$ is symmetric and thus Tannakian~\cite[Corollary 2.50]{DGNO2}. Let $G$ be a finite group such that $(\mathcal{C}_{\mathrm{ad}})_{\mathrm{pt}}\cong \Rep(G)$, and recall that the M\"uger center of $\mathcal C_{\mathrm{ad}}$ is $\mathcal C_{\mathrm{pt}}\cap \mathcal C_{\mathrm{ad}}=(\mathcal{C}_{\mathrm{ad}})_{\mathrm{pt}}$. Then by the remark above, the de-equivariantization $(\mathcal C_{\mathrm{ad}})_G$ is also modular. 
We will use these facts repeatedly throughout this work.

\section{Algorithm and general results}\label{section: algorithm}

\subsection{Basic algorithm}\label{section: basic algorithm}

We begin by presenting a recursive algorithm to generate potential lists of dimensions of the simple objects of an odd-dimensional integral MTC, given the rank and number of invertible simple objects of that MTC. The code for the implementation of this algorithm can be found in Appendix~\ref{section: algorithm code} in the function \texttt{basic\_algorithm}.

Consider an odd-dimensional MTC $\C$ of odd rank $n=2k+s$ with $s$ invertible objects. Then let the non-invertible objects of $\C$ be $X_1, X_1^*, \dots, X_k, X_k^*$ and their Frobenius-Perron dimensions be $\di_1, \di_1, \dots, \di_k, \di_k$. Without loss of generality, suppose that $\di_1\geq \di_2\geq\dots\geq \di_k>1$.

Note that by~\cite[Theorem 2.11]{ENO2}, for each $i$, there exists a positive integer $m_i$ such that $\FPdim(\mathcal{C})=m_i\di_i^2$. Also, note that each $m_i$ must be odd since $\FPdim(\mathcal{C})$ is odd.

We prove the following lemma, which allows us to initiate the recursion of the algorithm.

\begin{lemma}\label{lemma: recursion initiation}
    The value $m_1$ must satisfy the following two conditions:
    \begin{itemize}
        \item $m_1\equiv \rank(\C)\pmod{8}$, and
        \item $s\leq m_1\leq 2k+\frac{s}{9}$.
    \end{itemize}
\end{lemma}
\begin{proof}
    By~\cite[Lemma 5.4]{CP}, $\FPdim(\mathcal{C})\equiv\rank(\mathcal{C})\pmod{8}$, so $m_1\di_1^2\equiv \rank(\C)\pmod{8}$. Hence, as $\di_1$ is odd, $m_1\equiv \rank(\C)\pmod{8}$, showing the first condition.

    Additionally, let $\C_g$ be the component in the universal grading of $\C$ that contains the simple object of dimension $\di_1$. We have $m_1=\frac{\FPdim(\C)}{\di_1^2}\geq\frac{\FPdim(\C)}{\FPdim(\C_g)}=\sizeGC=s$.
    
    Finally, we know that
    \[m_1\di_1^2=\FPdim(\C)=\sizeGC+2\di_1^2+\dots+2\di_k^2\leq s+2k\di_1^2,\]
    so $m_1\leq \frac{s}{\di_1^2}+2k$. As $\di_1\geq 3$, this implies that $m_1\leq 2k+\frac{s}{9}$.
\end{proof}

The next lemma will allow us to recursively generate lists of dimensions.

\begin{lemma}\label{lemma: algorithm recursion}
    Suppose that
    \begin{equation}\label{eq:dimeq}
        c_i\di_i^2=s+2\di_{i+1}^2+2\di_{i+2}^2+\dots+2\di_k^2
    \end{equation}
    for some $i$ with $c_i$ a positive rational number. Then $m_{i+1}$ satisfies
    \begin{itemize}
        \item $m_{i+1}\in\left[m_i, \frac{sm_i}{tc_i}+\frac{2(k-i)m_i}{c_i}\right]$,
        \item $\left(\frac{c_im_{i+1}}{m_i}-2\right)\di_{i+1}^2=s+2\di_{i+2}^2+\dots+2\di_k^2$, and
        \item $\frac{c_im_{i+1}}{m_i}-2>0$,
    \end{itemize}
    where $t=225$ if $s=1$ and $t=9$ otherwise.
\end{lemma}

\begin{proof}
     We know that the right-hand side of equation~\eqref{eq:dimeq} is at most $s+2(k-i)\di_{i+1}^2$, which implies that
    \begin{equation}\label{eq:dimineq}
        c_i\di_i^2\leq s+2(k-i)\di_{i+1}^2.
    \end{equation}
    
    Since $m_{i+1}\di_{i+1}^2=\FPdim(\mathcal{C})=m_i\di_i^2$, we can substitute $\di_i^2=\frac{m_{i+1}}{m_i}\di_{i+1}^2$ into inequality~\eqref{eq:dimineq}, giving us
    \begin{align*}
        \frac{c_im_{i+1}}{m_i}\di_{i+1}^2&\leq s+2(k-i)\di_{i+1}^2 \\
        m_{i+1}&\leq \frac{sm_i}{c_i\di_{i+1}^2}+\frac{2(k-i)m_i}{c_i}.
    \end{align*}
    
    Additionally, if $s=1$, by~\cite[Corollary 7.2]{ENO2}, as $\mathcal{C}$ is perfect, $\mathcal{C}$ cannot have a simple object whose dimension is a power of a prime. As a result, we must have $\di_{i+1}\geq 15$, so $\di_{i+1}^2\geq 225$. If instead $s\neq 1$, we must have $\di_{i+1}\geq 3$, so $\di_{i+1}^2\geq 9$. Clearly, $m_{i+1}\geq m_i$ (as $\di_{i+1}\leq\di_i$). We therefore can bound the value of $m_{i+1}$:
    \begin{equation}\label{eq:newmrange}
        m_{i+1}\in\left[m_i, \frac{sm_i}{tc_i}+\frac{2(k-i)m_i}{c_i}\right],
    \end{equation}
    where $t=225$ if $s=1$ and $t=9$ otherwise.
    
    Also, by substituting $\di_i^2=\frac{m_{i+1}}{m_i}\di_{i+1}^2$ into equation~\eqref{eq:dimeq}, we know
    \begin{align}
        \frac{c_im_{i+1}}{m_i}\di_{i+1}^2&=s+2\di_{i+1}^2+2\di_{i+2}^2+\dots+2\di_k^2 \nonumber \\
        \left(\frac{c_im_{i+1}}{m_i}-2\right)\di_{i+1}^2&=s+2\di_{i+2}^2+\dots+2\di_k^2 \label{eq:newdimeq}.
    \end{align}

    Clearly, the right-hand side is positive, so we must have $\frac{c_im_{i+1}}{m_i}-2>0$.
\end{proof}

Now, for each $i$, let $u_i^2$ be the largest perfect square that divides $m_i$. The following corollary reduces the number of cases that we must check each time we recurse in our algorithm.

\begin{corollary}\label{corollary: u_i recursion}
    Suppose that
    \begin{equation}\label{eq: dimeq u_i}
        c_i\di_i^2=s+2\di_{i+1}^2+2\di_{i+2}^2+\dots+2\di_k^2
    \end{equation}
    for some $i$ with $c_i$ a positive rational number. Then $u_{i+1}$ satisfies
    \begin{itemize}
        \item $u_{i+1}^2\in\left[u_i^2, \frac{su_i^2}{tc_i}+\frac{2(k-i)u_i^2}{c_i}\right]$,
        \item $\left(\frac{c_iu_{i+1}^2}{u_i^2}-2\right)\di_{i+1}^2=s+2\di_{i+2}^2+\dots+2\di_k^2$, and
        \item $\frac{c_iu_{i+1}^2}{u_i^2}-2>0$,
    \end{itemize}
    where $t=225$ if $s=1$ and $t=9$ otherwise.
\end{corollary}
\begin{proof}
    Let $w=\frac{m_i}{u_i^2}$ (hence, $m_i=u_i^2w$). Since $w$ is squarefree, the equation $m_{i+1}\di_{i+1}^2=m_i\di_i^2$ implies that $m_{i+1}=u_{i+1}^2w$ for an odd integer $u_{i+1}$ ($u_{i+1}$ must be odd as $\FPdim(\mathcal{C})=u_{i+1}^2w\di_{i+1}^2$ is odd). As a result, checking all the values of $m_{i+1}$ in the range~\eqref{eq:newmrange} is equivalent to checking odd values of $u_{i+1}$ such that $u_{i+1}^2$ is in the range
    \[u_{i+1}^2\in\left[u_i^2, \frac{su_i^2}{tc_i}+\frac{2(k-i)u_i^2}{c_i}\right]\]
    for whether they satisfy the equation
    \[\left(\frac{c_iu_{i+1}^2}{u_i^2}-2\right)\di_{i+1}^2=s+2\di_{i+2}^2+\dots+2\di_k^2.\]
\end{proof}

\begin{algorithm}\label{alg: dim generation}
    Given odd integers $n$ and $s$, in order to determine whether an odd-dimensional MTC $\C$ can exist with $\rank(\C)=n$ and $\sizeGC=s$, perform the following steps:
    \begin{enumerate}
        \item Initialize the variable $k$ to $\frac{n-s}{2}$, and $t$ to $225$ if $s=1$ and to $9$ otherwise. 
        \item For every integer value $m_1$ satisfying the conditions in Lemma~\ref{lemma: recursion initiation}, set $i=1$ and $c_1=m_1-2$. Then recursively perform the following steps.
        \item If $i<k$, for every odd integer value $u_{i+1}$ that satisfies the conditions of Corollary~\ref{corollary: u_i recursion}, set $c_{i+1}$ to $\frac{c_iu_{i+1}^2}{u_i^2}-2$, increment $i$ and repeat this step. 
        \item If $i=k$ (i.e., the recursion is complete), check whether $\frac{s}{c_k}$ is a perfect square. 
        \begin{itemize}
            \item If yes, we know $\di_k^2=\frac{s}{c_k}$. Calculate the rest of the dimensions of the simple objects of $\C$, using the previously computed values of $u_1, \dots, u_k$, and verify whether this is a valid solution.
            \item If no, backtrack to the prior value of $i$.
        \end{itemize}
    \end{enumerate}
\end{algorithm}

\begin{proof}
    We know that
    \[\FPdim(\mathcal{C})=m_1\di_1^2=s+2\di_1^2+\dots+2\di_k^2.\]
    
    Suppose that we consider a specific $m_1$ satisfying the conditions of Lemma~\ref{lemma: recursion initiation}. The value of $m_1$ must also satisfy the equation
    \begin{equation}\label{eq:firstdimeq}
        (m_1-2)\di_1^2=s+2\di_2^2+\dots+2\di_k^2.
    \end{equation}

    Now, we can recursively apply Corollary~\ref{corollary: u_i recursion} in order to determine whether there exist valid simple object dimensions for $\C$: for each value of $u_{i+1}$ satisfying the first and third conditions in the corollary, we recursively check whether the second condition has a solution. We initiate this recursion with equation~\eqref{eq:firstdimeq}.

    Note that each time we recurse, we reduce the number of $\di_j$s on the right-hand side by $1$. As a result, after we recurse a sufficient number of times, we will reach an equation of the form
    \begin{equation}\label{eq:finaleq}
        c_k\di_k^2=s.
    \end{equation}
    
    Clearly, in order for there to exist a solution, $\frac{s}{c_k}$ must be a perfect square. Hence, if a case results in either:
    \begin{itemize}
        \item Equation~\eqref{eq: dimeq u_i} with $c_i\leq 0$, or
        \item Equation~\eqref{eq:finaleq} with a $c_k$ such that $\frac{s}{c_k}$ is not a perfect square, 
    \end{itemize}
    that case cannot represent an odd-dimensional MTC of rank $n$ with $s$ invertible objects.
\end{proof}

\subsection{Adjoint algorithm}\label{section: adjoint algorithm}

The algorithm from Section~\ref{section: basic algorithm} allows us to generate candidate Frobenius-Perron dimensions for all simple objects in an odd-dimensional MTC. However, this algorithm becomes slow with a large number of simple objects. Therefore, in this section, we present a second algorithm that generates the candidate dimensions of only simple objects in the adjoint subcategory $\Cad$, rather than all simple objects. The code for the implementation of this algorithm can be found in Appendix~\ref{section: algorithm code} in the function \texttt{adjoint\_algorithm}.

Let the non-invertible simple objects in $\Cad$ be $X_1, X_1^*, \dots, X_k, X_k^*$, and let the dimensions of $X_1, X_2, \dots, X_k$ be $\di_1\geq\di_2\geq\dots\geq\di_k$. Also let $m_i=\frac{\FPdim(\C)}{\di_i^2}$ for every $i$. The algorithm is based on the fact that $\sizeGC\FPdim(\Cad)=\FPdim(\C)$, by~\cite[Proposition 8.20]{ENO1}. Specifically, if we let $s=\sizeGCad$, we know that
\[\FPdim(\Cad)=\frac{\FPdim(\C)}{\sizeGC}=s+2\di_1^2+\dots+2\di_k^2.\]
Thus, as $\FPdim(\C)=m_1\di_1^2$, in place of equation~\eqref{eq:firstdimeq}, we have
\[\left(\frac{m_1}{\sizeGC}-2\right)\di_1^2=s+2\di_2^2+\dots+2\di_k^2.\]
Note that Lemma~\ref{lemma: algorithm recursion} and Corollary~\ref{corollary: u_i recursion} still hold. Using these, we can recurse as we did in Algorithm~\ref{alg: dim generation} to generate potential solutions.

However, the initiation from Lemma~\ref{lemma: recursion initiation} needs the following adjustment.

\begin{lemma}\label{lemma: Cad recursion initiation}
    The value $m_1$ must satisfy the following two conditions:
    \begin{itemize}
        \item $m_1\equiv \rank(\C)\pmod{8}$, and
        \item $\sizeGC\leq m_1\leq 2k\cdot\sizeGC+\frac{s\cdot\sizeGC}{9}$.
    \end{itemize}
\end{lemma}
\begin{proof}
    By~\cite[Lemma 5.4]{CP}, $\FPdim(\mathcal{C})\equiv\rank(\mathcal{C})\pmod{8}$, so $m_1\di_1^2\equiv \rank(\C)\pmod{8}$. Hence, as $\di_1$ is odd, $m_1\equiv \rank(\C)\pmod{8}$, showing the first condition.

    Additionally, we have $m_1=\frac{\FPdim(\C)}{\di_1^2}\geq\frac{\FPdim(\C)}{\FPdim(\Cad)}=\sizeGC$.
    
    Finally, we know that
    \[m_1\di_1^2=\FPdim(\C)=\sizeGC\cdot(s+2\di_1^2+\dots+2\di_k^2)\leq \sizeGC\cdot(s+2k\di_1^2),\]
    so $m_1\leq\frac{s\cdot\sizeGC}{\di_1^2}+2k\cdot\sizeGC$. As $\di_1\geq 3$, this implies that $m_1\leq 2k\cdot\sizeGC+\frac{s\cdot\sizeGC}{9}$.
\end{proof}

As a result, we have the following algorithm.

\begin{algorithm}\label{alg: Cad dim generation}
    In order to determine whether an odd-dimensional MTC $\C$ can exist with $\rank(\Cad)=n$ and $\sizeGCad=s$ for odd integers $n$ and $s$, perform the following steps, given the values of $\rank(\C)$ and $\sizeGC$:
    \begin{enumerate}
        \item Initialize the variable $k$ to $\frac{n-s}{2}$, and $t$ to $225$ if $\sizeGC=1$ and to $9$ otherwise. 
        \item For every integer value $m_1$ satisfying the conditions in Lemma~\ref{lemma: Cad recursion initiation}, set $i=1$ and $c_1=\frac{m_1}{\sizeGC}-2$. Then recursively perform the following steps.
        \item If $i<k$, for every odd integer value $u_{i+1}$ that satisfies the conditions of Corollary~\ref{corollary: u_i recursion}, set $c_{i+1}$ to $\frac{c_iu_{i+1}^2}{u_i^2}-2$, increment $i$ and repeat this step. 
        \item If $i=k$ (i.e., the recursion is complete), check whether $\frac{s}{c_k}$ is a perfect square. 
        \begin{itemize}
            \item If yes, we know $\di_k^2=\frac{s}{c_k}$. Calculate the rest of the dimensions of the simple objects of $\C$, using the previously computed values of $u_1, \dots, u_k$, and verify whether this is a valid solution.
            \item If no, backtrack to the prior value of $i$.
        \end{itemize}
    \end{enumerate}
\end{algorithm}

\subsection{General results for classification}\label{sec: general results}

In this subsection, we include results that are true for general ranks, which we will apply for specific cases in later sections to classify odd-dimensional MTCs of given ranks.

We begin by proving the following Proposition, which is a generalization of~\cite[Proposition 5.6]{CP}.

\begin{proposition}\label{lemma: fixed simple}
Let $\mathcal C$ be an MTC such that $\mathcal{C}_\mathrm{pt}\subseteq \mathcal C_{\mathrm{ad}}$ and $\Cpt$ is odd-dimensional. Consider the action of $\mathcal{G(C)}$ on the set of simple objects of $\mathcal C$. If a simple object $X$ is fixed by this action, then $X\in \mathcal C_{\mathrm{ad}}.$
\end{proposition}
\begin{proof}
Suppose there exists a simple object $X$ in $\mathcal C$ that is fixed by the action. Since $\mathcal{C}_\mathrm{pt}\subseteq \mathcal C_{\mathrm{ad}}$ and $(\mathcal{C}_\mathrm{pt})'=\mathcal C_{\mathrm{ad}}$, it follows that $\mathcal{C}_\mathrm{pt}$ is an odd-dimensional symmetric fusion category. Then by~\cite[Corollary 2.7]{DGNO1} we know that $\theta_h=1$ for all $h\in \mathcal{C}_\mathrm{pt}.$ Using the balancing equation \eqref{balancing}, we obtain
\begin{align*}
s_{h,X}=\theta_h^{-1}\theta_X^{-1}\theta_X \dim(X)=\dim(X)=\dim(h)\dim(X), \text{ for all } h\in \mathcal{C}_\mathrm{pt}.
\end{align*}
Hence by~\cite[Proposition 2.5]{M} we conclude that $X\in (\mathcal{C}_\mathrm{pt})'=\mathcal C_{\mathrm{ad}}$, as desired.
\end{proof}

\begin{corollary}\label{corollary: same dim outside Cad}
    Let $\C$ be an MTC such that $\sizeGC$ is an odd prime $p$, $\Cadpt$ is non-trivial, and the ranks of all non-adjoint components of the universal grading are $p$. Then, for each non-adjoint component of the universal grading, the simple objects in that component form an orbit under the action of $\GC$ on the simple objects of $\C$. Additionally, all simple objects outside of $\Cad$ have the same Frobenius-Perron dimension.
\end{corollary}
\begin{proof}
    As $\mathcal{G}(\Cad)$ is a non-trivial subgroup of $\GC$ and $\sizeGC$ is prime, $\Cpt\subseteq\Cad$. Thus, by Proposition~\ref{lemma: fixed simple}, a simple object $X\not\in\Cad$ is not fixed by the action of $\GC$ on the set of simple objects of $\C$. Therefore, the size of the orbit of $X$ is $p$. Since $\Cpt\subseteq\Cad$, the action preserves the grading, and hence all objects of the orbit are in the same component, so the size of the orbit of $X$ is equal to the rank of the component of $X$ in the universal grading. Hence, all simple objects in that component have the same Frobenius-Perron dimension.

    Additionally, by~\cite[Proposition 8.20]{ENO1}, all components of the universal grading have the same Frobenius-Perron dimension. Hence, as all non-adjoint components have the same rank, all simple objects in the non-adjoint components must have the same Frobenius-Perron dimension.
\end{proof}

In the subsequent proofs, we will frequently use the following remark, as we did in the proof of Corollary~\ref{corollary: same dim outside Cad}.

\begin{remark}\label{remark: Cpt subset Cad}
    Let $\C$ be an MTC such that $\sizeGC$ is prime and $\Cadpt$ is non-trivial. Then as $\GCad$ is a non-trivial subgroup of $\GC$, we know that $\Cpt\subseteq\Cad$.
\end{remark}

\begin{lemma}\label{thm:fixeddims}
    Let $\mathcal{C}$ be an odd-dimensional MTC such that $\sizeGC$ is a prime $p$, $\Cadpt$ is non-trivial, and the ranks of all non-adjoint components of the universal grading are $p$. Then in the multiset of the Frobenius-Perron dimensions of the non-invertible simple objects of $\C$, each element must either appear exactly $2pk$ times for an integer $k$ or its Frobenius-Perron dimension must be divisible by $p$.
\end{lemma}
\begin{proof}
    First, by Corollary~\ref{corollary: same dim outside Cad}, the Frobenius-Perron dimensions of the simple objects outside of $\Cad$ are equal. Specifically, this dimension appears $p(p-1)$ times in the multiset, which is of the form $2pk$ as $p-1$ is even. Therefore, we exclusively focus on the simple objects inside of $\Cad$ from now on.

    By Remark~\ref{remark: Cpt subset Cad}, consider the action of $\GC\cong\GCad$ on the simple objects of $\mathcal{C}_\mathrm{ad}$. The sizes of its orbits divide $\sizeGC=p$, i.e., they are either $1$ or $p$. Let $X\in\Cad$ be a non-invertible simple object that is not fixed by this action, so the size of its orbit must be $p$. 
    
    Letting $\GC$ be generated by $g$, denote the elements of this orbit by $g^i\otimes X$ for integers $0\leq i\leq p-1$. Then $\FPdim(g^i\otimes X)=\FPdim(X)$ for all $i$. Since the size of the orbit of $X$ is odd, the orbits of $X$ and $X^*$ are disjoint. We also know that $\FPdim(X)=\FPdim(X^*)$. Hence, we have exactly $2p$ simple objects of the same Frobenius-Perron dimension for every $X$ that is not fixed by the action.
    
    Now, suppose that $X$ is fixed by the action. Let $Z_1, \dots, Z_p$ be the simple objects of $\C_g$. By Proposition~\ref{lemma: fixed simple}, we know that none of these elements are fixed by the action of $\GC$. Without loss of generality, suppose that $Z_i=g^{i-1}\otimes Z_1$.
    
    Since $X\in\Cad$, $X\otimes Z_1\in \C_g$. Therefore,
    \[X\otimes Z_1=N_1Z_1\oplus N_2Z_2\oplus\dots\oplus N_pZ_p\]
    for some integers $N_1, \dots, N_p$. We also have
    \[g\otimes X\otimes Z_1=N_1Z_2\oplus N_2Z_3\oplus\dots N_pZ_1.\]
    But $X\otimes Z_1=g\otimes X\otimes Z_1$, so $N_1=N_2=\dots=N_p$. Consequently,
    \[X\otimes Z_1=N_1(Z_1\oplus Z_2\oplus\dots\oplus Z_p).\]
    Taking the Frobenius-Perron dimension of both sides and applying Corollary~\ref{corollary: same dim outside Cad} gives us
    \begin{align*}
        \FPdim(X)\FPdim(Z_1)&=N_1\cdot p\cdot\FPdim(Z_1) \\
        \FPdim(X)&=N_1\cdot p.
    \end{align*}
    Hence, $\FPdim(X)$ is divisible by $p$.
\end{proof}

\begin{remark}\label{remark: fixed fpdim divisible by p}
    From the proof of Lemma~\ref{thm:fixeddims}, we can see that if a simple object in $\C$ is fixed by the action of $\GC$, then its Frobenius-Perron dimension must be divisible by $p$. Also, the number of non-fixed objects in $\C$ must be a multiple of $2p$.
\end{remark}

\begin{remark}
    Considering the action of a group of odd order on the simple objects of $\C$, the proof of Lemma~\ref{thm:fixeddims} also shows that the orbits of a simple object $X\in\C$ and its dual $X^*$ are disjoint.
\end{remark}

We also prove a result that allows us to easily discard certain possible values of $\sizeGC$ for a given rank.

\begin{lemma}\label{lemma: rank linear combination}
    If $\mathcal{C}$ is an odd-dimensional MTC, then $\rank(\mathcal{C})$ is expressible as a non-negative integral linear combination of $\sizeGC$ and $8$.
\end{lemma}
\begin{proof}
    Consider the universal grading $\mathcal{C}=\bigoplus_{g\in\mathcal{G}(\mathcal{C})}\mathcal{C}_g$ of $\mathcal{C}$. By~\cite[Proposition 8.20]{ENO1}, all such $\mathcal{C}_g$'s have the same Frobenius-Perron dimension. But $\rank(\C_g)\equiv\FPdim(\C_g)\pmod{8}$ for all $g\in\GC$~\cite[Lemma 5.4]{CP}, so all components of the universal grading must have ranks that are congruent mod $8$. Letting the minimum value of $\rank(\C_g)$ across all $g$'s be $m$, we therefore have $\rank(\C)=\sizeGC\cdot m+8k$ for some non-negative integer $k$, as desired.
\end{proof}

We proceed with three lemmas with necessary conditions for odd-dimensional MTCs with non-trivial $\Cadpt$, which we will use in future sections. Note that non-trivial $\Cadpt$ implies that $\C$ is neither pointed (as $\rank(\Cad)\geq \rank(\Cadpt)>1$) nor perfect (as $\rank(\Cpt)\geq\rank(\Cadpt)>1$). 

\begin{lemma}\label{lemma: rank more than 1}
    Let $\mathcal{C}$ be an odd-dimensional MTC with non-trivial $(\mathcal{C}_\mathrm{ad})_\mathrm{pt}$ and $p$ be a prime divisor of $\sizeGCad$. Then the universal grading of $\mathcal{C}$ must have at least three components of rank at least $p$. 
\end{lemma}
\begin{proof}
    Note that $p$ must be odd because $p$ is a divisor of $\sizeGCad$, which divides $\sizeGC$ and $\FPdim(\C)$. By~\cite[Lemma 5.1 (a)]{CP}, there must exist a non-trivial component $\C_g$ of the universal grading that has at least $p$ non-invertible simple objects. Because of the duality of universal grading components, the dual of $\C_g$ must also have at least $p$ simple objects. $\Cad$ must clearly have at least $p$ simple objects, completing the proof.
\end{proof}

\begin{lemma}\label{lemma: rank not divisible by p}
    Let $\mathcal{C}$ be an odd-dimensional MTC with non-trivial $(\mathcal{C}_\mathrm{ad})_\mathrm{pt}$. If $|\mathcal{G}(\mathcal{C})|$ is a prime number $p$, then the universal grading of $\mathcal{C}$ has at most one component with rank not divisible by $p$. Additionally, if such a component exists, it must be $\mathcal{C}_\mathrm{ad}$. 
\end{lemma}
\begin{proof}
    By Remark~\ref{remark: Cpt subset Cad}, $\Cpt\subseteq\Cad$. We also have $\FPdim(\Cpt)=\sizeGC=p$. As $\Cadpt$ is non-trivial, $\C$ is not pointed, so by~\cite[Proposition 5.6]{CP}, all non-trivial components in the universal grading of $\C$ must have rank divisible by $p$. The only component whose rank may not be divisible by $p$ is $\Cad$.
\end{proof}

\begin{lemma}\label{lemma: odd number of components of rank}
    Let $\mathcal{C}$ be an odd-dimensional MTC with non-trivial $(\mathcal{C}_\mathrm{ad})_\mathrm{pt}$. In the universal grading of $\mathcal{C}$, there is exactly one value of rank such that an odd number of components of the universal grading of $\mathcal{C}$ have this rank. Additionally, the value of such rank is equal to $\rank(\Cad)$ and is greater than $1$. 
\end{lemma}
\begin{proof}
    Given that the number of components in the universal grading must be odd (since $|\mathcal{G}(\mathcal{C})|$ is odd) and dual components of the universal grading other than $\mathcal{C}_\mathrm{ad}$ have the same rank, there must be an even number of non-adjoint components whose rank matches $\rank(\mathcal{C}_\mathrm{ad})$. Including $\Cad$ gives us an odd total number of components of that rank. Since $(\mathcal{C}_\mathrm{ad})_\mathrm{pt}$ is non-trivial, this rank must be greater than $1$. All other ranks must occur an even number of times in the universal grading of $\mathcal{C}$ because of the duality of universal grading components.
\end{proof}

The next two results will be used to discard multiple cases in later sections. Note that the following proposition holds for all MTCs (not just odd-dimensional ones). 

\begin{proposition}\label{proposition: invertible implies equal rank}
    Let $\C$ be an MTC. Then exactly $\frac{\sizeGC}{\sizeGCad}$ components of the universal grading of $\C$ contain an invertible object. Additionally, for each component that contains an invertible object, consider the multiset of the Frobenius-Perron dimensions of all its simple objects. Then, all of those multisets are identical.
\end{proposition}
\begin{proof}
    Suppose that a component $\C_g$ of the universal grading contains an invertible object $h$. Also, let the simple objects in $\C_g$ be $X_1, \dots, X_k$. We prove that all simple objects in $\Cad$ are $h^*\otimes X_1, \dots, h^*\otimes X_k$. As $h^*\in\C_{g^{-1}}$ is invertible, all of those objects are non-isomorphic simple objects in $\Cad$. Hence, as $\FPdim(h^*\otimes X_i)=\FPdim(X_i)$,
    \begin{align}
        \FPdim(\Cad)&\geq\FPdim(h^*\otimes X_1)^2+\dots+\FPdim(h^*\otimes X_k)^2 \label{eq: Cad dimension inequality} \\
        &=\FPdim(X_1)^2+\dots+\FPdim(X_k)^2 \nonumber,
    \end{align}
    where the inequality is strictly greater if there are simple objects other than $h^*\otimes X_1, \dots, h^*\otimes X_k$ in $\Cad$ and equal otherwise. Also, we have
    \[\FPdim(\C_g)=\FPdim(X_1)^2+\dots+\FPdim(X_k)^2.\]

    By~\cite[Proposition 8.20]{ENO1}, $\FPdim(\Cad)=\FPdim(\C_g)$, so the equality case in Inequality~\eqref{eq: Cad dimension inequality} holds. As a result, the only simple objects in $\Cad$ are $h^*\otimes X_1, \dots, h^*\otimes X_k$. As $h^*$ is invertible, these have the same Frobenius-Perron dimensions as $X_1, \dots, X_k$, so the multiset of dimensions of the simple objects in $\C_g$ is the same as that for $\Cad$. This shows the second statement of the proposition. This also implies that every component that contains an invertible object actually contains $\sizeGCad$ invertible objects. Hence exactly $\frac{\sizeGC}{\sizeGCad}$ components contain an invertible object, as desired.
\end{proof}

\begin{corollary}\label{corollary: three components rank Cad}
    Let $\C$ be an odd-dimensional MTC with $\GCad\lneq\GC$. Then there must exist at least three components of the universal grading of $\C$ that have the same rank as $\Cad$.
\end{corollary}
\begin{proof}
    Clearly, one such component is $\Cad$ itself. As $\GCad\lneq\GC$, there must exist a non-trivial component $\C_g$ of the universal grading that contains an invertible object. By Proposition~\ref{proposition: invertible implies equal rank}, $\rank(\C_g)=\rank(\Cad)$. Since $\rank(\C_{g^{-1}})=\rank(\C_g)$, we have the desired three components: $\Cad$, $\C_g$, and $\C_{g^{-1}}$.    
\end{proof}

Next, we will show the following proposition, whose proof follows the proof of~\cite[Proposition 6]{DKP}.

\begin{proposition}\label{proposition: semidirect product}
    Let $\C$ be an odd-dimensional MTC of Frobenius-Perron dimension $p^rq^2$, where $p$ and $q$ are distinct odd prime numbers such that $p\nmid (q+1)$ and $\sizeGC=p^{r-1}$. Then
    \begin{enumerate}[label=(\alph*)]
        \item Either $p\mid q-1$ or $q\mid p-1$. 
        \item Additionally, if $r=2$, then $\mathcal{C}\cong\Rep(D^\omega(H))$ with $H\cong\mathbb{Z}_q\rtimes\mathbb{Z}_p$ non-abelian for some $3$-cocycle $\omega$.
    \end{enumerate}
\end{proposition}
\begin{proof}
    (a) Since $\sizeGC=p^{r-1}$, we know that $\FPdim(\Cad)=pq^2$. Therefore, by~\cite[Theorem 1.6]{ENO2}, $\Cad$ is solvable, so by~\cite[Proposition 4.5 (iv)]{ENO2}, it must contain a non-trivial invertible object. Since $p$ is a prime, $\GCad$ is a subgroup of $\GC$, and $\sizeGCad$ divides $\FPdim(\Cad)$, we know that $\sizeGCad=p$, and $\GCad\cong\Z_p$. 
    
   Consider the de-equivariantization $(\Cad)_{\Z_p}$ of $\Cad$ by $\GCad\cong\Z_p$ (see Section~\ref{section: equivariantization}). We know that $\FPdim((\Cad)_{\Z_p})=\frac{\FPdim(\Cad)}{p}=q^2$. Since $(\Cad)_{\Z_p}$ is modular, we know that $(\Cad)_{\Z_p}\cong\Vect_G^\chi$, where $G$ is an abelian group of order $q^2$ and $\chi$ is a nondegenerate quadratic form on $G$, see~\cite[Lemma 4.11]{DN}.

Note that $G\cong \Z_{q^2}$ or $\Z_q\times \Z_q$, and so it has either $1$ or $q+1$ subgroups of order $q$, respectively. Since by assumption $p\nmid q+1$, then there exists at least one subgroup of order $q$ that is invariant under the action of $\Z_p$. 
    Now, take the subcategory $\mathcal{D}$ of $(\Cad)_{\Z_p}$ associated with said subgroup, which has Frobenius-Perron dimension $q$, and consider the equivariantization $\mathcal{D}^{\Z_p}$ of $\mathcal{D}$ by $\Z_p$. We know that $\mathcal{D}^{\Z_p}\subseteq\Cad$, $\FPdim(\mathcal{D}^{\Z_p})=pq$, and $\mathcal{D}^{\Z_p}$ is integral. Since $|\mathcal{G}(\mathcal{D}^{\Z_p})|\leq\sizeGCad=p<pq=\FPdim(\mathcal{D}^{\Z_p})$, we also know that $\mathcal{D}^{\Z_p}$ is not pointed. As a result, since both $p$ and $q$ are odd primes, by~\cite[Theorem 6.3]{EGO}, we know that $\mathcal{D}^{\Z_p}\cong\Rep(K)$ where either $K\cong\Z_q\rtimes\Z_p$ or $K\cong\Z_p\rtimes\Z_q$ is non-abelian. Therefore, either $p\mid q-1$ or $q\mid p-1$, as desired.

    (b) From the proof of part (a), we know that $\mathcal{D}^{\Z_p}\cong\Rep(K)\subseteq\C$ with either $K\cong\Z_q\rtimes\Z_p$ or $K\cong\Z_p\rtimes\Z_q$, so there exists the canonical \'etale algebra $k^{\Z_q\rtimes\Z_p}$ or $k^{\Z_p\rtimes\Z_q}$ of Frobenius-Perron dimension $pq$ in $\C$. 

    Hence, as $\FPdim(\C)=p^2q^2$, by~\cite[Corollary 4.1 (i)]{DMNO}, $\C\cong\mathcal{Z}(\mathcal{B})$ as braided tensor categories, where $\mathcal{B}$ is a fusion category of Frobenius-Perron dimension $pq$ and $\mathcal{Z}(B)$ denotes its Drinfeld center (see~\cite[Definition 7.13.1]{EGNO}). By~\cite{EGO}, either $\mathcal{B}\cong\Vect_H^\omega$ or $\mathcal{B}\cong\Rep(H)$ for some group $H$ of order $pq$ and some $3$-cocycle $\omega$ of $H$. Then $\C\cong\mathcal{Z}(\Vect_H^\omega)$.
    
    Given that $p$ and $q$ are primes, we know that one of the following is true: (1) $H\cong\Z_{pq}$, (2) $p\mid q-1$ and $H\cong\Z_q\rtimes\Z_p$ is non-abelian, or (3) $q\mid p-1$ and $H\cong\Z_p\rtimes\Z_q$ is non-abelian. By casework, we will show that only option (2) is possible.
    
  First, suppose for contradiction that $H\cong\Z_{pq}$. As all Sylow subgroups of $\Z_{pq}$ are cyclic, its Schur multiplier is trivial, and hence $\mathcal{Z}(\Vect_{\Z_{pq}}^\omega)$ is pointed. But $\FPdim(\C)=p^2q^2$, while $\sizeGC=p$, so $\C$ cannot be pointed, a contradiction. 
    
    Next, assume that $p\mid q-1$ and $H\cong\Z_q\rtimes\Z_p$ is non-abelian. As a result, $\C\cong\mathcal{Z}(\Vect_H^\omega)\cong\Rep(D^\omega(H))$ for some $3$-cocycle $\omega$, as desired. 

   Finally, assume that $q\mid p-1$ and $H\cong\Z_p\rtimes\Z_q$ is non-abelian. Then $\C\cong\mathcal{Z}(\Vect_{H}^\omega)$ for some $3$-cocycle $\omega$. However, $\sizeGC=|\mathcal{G}(\mathcal{Z}(\Vect_H^\omega))|=q$ by~\cite[Proposition 4]{DKP}. This is a contradiction with the assumption $\sizeGC=p$.
\end{proof}

\section{Odd-dimensional MTCs of rank 17--23}\label{section: ranks 17 to 23}

In this section, we complete the classification of odd-dimensional MTCs of rank $17$ to $23$, and we use that classification to show two general results about odd-dimensional MTCs that we use in future sections. 

The classification of odd-dimensional MTCs of rank $17$ to $23$ was begun in~\cite{CP}, which showed that such MTCs must be either pointed or perfect. The following result follows directly from an application of Algorithm~\ref{alg: dim generation}.

\begin{theorem}\label{thm: pointed 17 to 23}
    All odd-dimensional MTCs of rank $17$, $19$, $21$, and $23$ must be pointed.
\end{theorem}

\begin{proof}
    Based on~\cite[Theorem 6.3 (b)]{CP}, all odd-dimensional MTCs of rank $17$, $19$, $21$, and $23$ must be either pointed or perfect. Running Algorithm~\ref{alg: dim generation} for each of $\rank(\C)=17, 19, 21, 23$ and $\sizeGC=1$ produces no potential solutions, and therefore the perfect case is impossible. The desired result follows.
\end{proof}

Next, we use the above result to show two general results that are used in future sections to classify higher-rank odd-dimensional MTCs. We first extend~\cite[Lemma 6.1]{CP} to odd-dimensional MTCs of higher rank.

\begin{lemma}\label{thm:trivial}
    Let $\mathcal{C}$ be a non-pointed odd-dimensional MTC of rank at most $73$. Then $(\mathcal{C}_\mathrm{ad})_\mathrm{pt}$ is trivial if and only if $\mathcal{C}_\mathrm{pt}$ is trivial.
\end{lemma}
\begin{proof}
    Clearly, if $\mathcal{C}_\mathrm{pt}$ is trivial, then $(\mathcal{C}_\mathrm{ad})_\mathrm{pt}$ is trivial. We now show the opposite direction.

    Suppose that $(\mathcal{C}_\mathrm{ad})_\mathrm{pt}$ is trivial. We know $\rank(\mathcal{C}_\mathrm{ad})\rank(\mathcal{C}_\mathrm{pt})=\rank(\mathcal{C})$~\cite[Lemma 5.2]{CP}. Therefore, since $\mathcal{C}$ is non-pointed, $\rank(\mathcal{C}_\mathrm{pt})<\rank(\C)$, so $\rank(\Cad)>1$ and $\Cad$ is not pointed. As a result, by~\cite[Theorem 4.5]{BR},~\cite[Theorem 6.3]{CP}, and Theorem~\ref{thm: pointed 17 to 23}, $\rank(\Cad)$ must be at least $25$. We also know that $\rank(\C)\leq 73$ and is odd, so $\rank(\Cad)$ cannot be a proper divisor of $\rank(\C)$. Therefore, $\rank(\Cad)=\rank(\C)$, so $\C$ is perfect and $\Cpt$ is trivial, as desired.
\end{proof}

We also use the following lemma to discard many cases where $\Cad$ has $9$ simple objects in later sections.

\begin{lemma}\label{lemma: rank Cad 9}
    Let $\C$ be an odd-dimensional MTC such that $\rank(\Cad)=9$. Then the ranks of all components of the universal grading of $\C$ are greater than $1$. 
\end{lemma}
\begin{proof}
    We know that $\sizeGCad$ is odd as it divides $\FPdim(\C)$.  We perform casework on its value. 

    If $\sizeGCad=1$, by Proposition~\ref{proposition: invertible implies equal rank}, all components of the universal grading of $\C$ have rank equal to $\rank(\Cad)=9$. 

    If $\sizeGCad=3$, consider the de-equivariantization $(\Cad)_{\Z_3}$ of $\Cad$ by $\GCad\cong \Z_3$. The three invertible objects of $\Cad$ create one simple object in $(\Cad)_{\Z_3}$ (see Section~\ref{section: equivariantization}). The six non-invertible simple objects in $\Cad$ are either all fixed or all non-fixed under the action of $\GCad$ on the simple objects of $\Cad$. 
    
    In the first case, each non-invertible simple object creates three simple objects in $(\Cad)_{\Z_3}$, so $\rank((\Cad)_{\Z_3})=1+3\cdot 6 = 19$, and $(\Cad)_{\Z_3}$ is pointed by Theorem~\ref{thm: pointed 17 to 23}. This means that $\FPdim((\Cad)_{\Z_3})=19$ and  $\FPdim(\Cad)=3\cdot 19 = 57$. Since $57$ is not a perfect square, no component of the universal grading of $\C$ can have rank $1$. 

    In the second case, each set of $3$ simple objects in $\Cad$ creates one simple object in $(\Cad)_{\Z_3}$, so $\rank((\Cad)_{\Z_3})=1+\frac{6}{3} = 3$, and $(\Cad)_{\Z_3}$ is again pointed by~\cite[Theorem 2.6]{HR}. Then $\FPdim(\Cad) = 3\cdot 3 = 9$. However, $\Cad$ contains $3$ invertible objects and $6$ non-invertible objects, so $\FPdim(\Cad)\geq 3+6\cdot 3^2>9$, which is a contradiction.

    If $\sizeGCad=5$, consider the de-equivariantization $\deeqZ{5}$ of $\Cad$ by $\GCad\cong \Z_5$. The five invertible objects of $\Cad$ create one simple object in the de-equivariantization $\deeqZ{5}$. The four non-invertible simple objects in $\Cad$ must be all fixed under the action of $\GCad$ on the simple objects of $\Cad$, and each of them creates five simple objects in $\deeqZ{5}$. As a result, $\rank(\deeqZ{5})=1+5\cdot 4=21$, and by Theorem~\ref{thm: pointed 17 to 23}, $\deeqZ{5}$ is pointed. This means that $\FPdim(\deeqZ{5})=21$ and $\FPdim(\Cad)=5\cdot 21=105$, which is not a perfect square, and no component of the universal grading of $\C$ can have rank $1$. 

    If $\sizeGCad=7$, a similar argument holds. In the the de-equivariantization $\deeqZ{7}$ of $\Cad$ by $\GCad\cong \Z_7$, the seven invertible objects of $\Cad$ create one simple object in $\deeqZ{7}$, and the two non-invertible simple objects in $\Cad$ each create seven simple objects in $\deeqZ{7}$. Hence, $\rank(\deeqZ{7})=1+7\cdot 2=15$, so $\deeqZ{7}$ is pointed~\cite[Theorem 6.3 (a)]{CP}. This means that $\FPdim(\deeqZ{7})=15$ and $\FPdim(\Cad)=7\cdot 15=105$, which is not a perfect square, so no component of the universal grading of $\C$ can have rank $1$.

    Finally, consider $\sizeGCad=9$. In this case $\FPdim(\Cad)=9$, so the Frobenius-Perron dimension of every component of the universal grading must be $9$. As a result, each component either contains exactly $9$ invertible objects or exactly $1$ object of dimension $3$, and hence no component contains at least $3$ non-invertible objects. But this contradicts~\cite[Lemma 5.1 (a)]{CP}, as $3$ divides $\sizeGCad=9$, so this case is impossible.
\end{proof}

\section{Odd-dimensional MTCs of rank 25}\label{sec: rank 25}

In this section, we prove that all odd-dimensional MTCs of rank $25$ must be pointed, perfect, or equivalent to $\Rep(D^\omega(\mathbb{Z}_7\rtimes\mathbb{Z}_3))$. Let $\C$ be an odd-dimensional MTC of rank $25$. We begin by restricting the possible values of $\sizeGC$, as a corollary of Lemma~\ref{lemma: rank linear combination}.

\begin{corollary}\label{corollary: count invertibles 25}
    If $\mathcal{C}$ is an odd-dimensional MTC of rank $25$, then we must have $\sizeGC=1, 3, 5, 9, 17, 25$.
\end{corollary}

We continue by further reducing the number of possible values of $\sizeGC$.

\begin{proposition}\label{prop: 1 3 25 invertibles rank 25}
    Let $\mathcal{C}$ be an odd-dimensional MTC of rank $25$. Then $\mathcal{C}$ is either pointed, perfect, or has $3$ invertible objects. If $\C$ has $3$ invertible objects, then the ranks of the components of the universal grading must be $19$, $3$, and $3$.
\end{proposition}

\begin{proof}
    We perform casework based on the possible values of $\sizeGC$, from Corollary~\ref{corollary: count invertibles 25}. Then, in each subcase, by~\cite[Proposition 8.20]{ENO1} and~\cite[Lemma 5.4]{CP}, we know that $\rank(\mathcal{C}_g)\equiv\rank(\mathcal{C}_h)\pmod{8}$ for all $g, h\in\mathcal{G}(\mathcal{C})$. This combined with $\sum_{g\in\mathcal{G}(\mathcal{C})}\rank(\mathcal{C}_g)=25$ allows us to enumerate all possible multisets of the ranks of the components of the universal grading for each $\sizeGC$.
    
    For every $\sizeGC>1$, we know that $\Cpt$ is non-trivial, and therefore, by Lemma~\ref{thm:trivial}, $\Cadpt$ is also non-trivial. As a result, we can use Lemmas~\ref{lemma: rank more than 1},~\ref{lemma: rank not divisible by p}, and~\ref{lemma: odd number of components of rank} to discard cases.

    \textbf{Case 1: $\sizeGC=17$.} The ranks of the components of the universal grading must be $[(1, 9), (16, 1)]$ (see Section~\ref{section: universal grading} for a description of the notation. In this case, there is $1$ component of rank $9$, and $16$ components of rank $1$). This case is therefore discarded by Lemma~\ref{lemma: rank more than 1}.

    \textbf{Case 2: $\sizeGC=9$.} The ranks of the components of the universal grading must be either $[(1, 17), (8, 1)]$ or $[(2, 9), (7, 1)]$. We discard both possibilities by Lemma~\ref{lemma: rank more than 1}. 

    \textbf{Case 3: $\sizeGC=5$.} The ranks of the components of the universal grading must be $[(5, 5)]$. We know that $\mathcal{G}(\mathcal{C})\cong\mathbb{Z}_5$. As $\mathcal{G}(\mathcal{C}_\mathrm{ad})$ is a subgroup of $\mathcal{G}(\mathcal{C})$ and $|\mathcal{G}(\mathcal{C}_\mathrm{ad})|\neq 1$ by Lemma~\ref{thm:trivial}, we know that $|\mathcal{G}(\mathcal{C}_\mathrm{ad})|=5$. But $\rank(\mathcal{C}_\mathrm{ad})=5$, so all simple objects of $\mathcal{C}_\mathrm{ad}$ must be invertible, i.e., their Frobenius-Perron dimensions are all $1$. As a result, $\FPdim(\mathcal{C}_\mathrm{ad})=5$, so by~\cite[Proposition 8.20]{ENO1}, the Frobenius-Perron dimension of all the components of the universal grading must be $5$. This is a contradiction, as the square of the Frobenius-Perron dimension of a non-invertible object of $\mathcal{C}$ must be at least $9$.

    \textbf{Case 4: $|\mathcal{G}(\mathcal{C})|=3$.} The ranks of the components of the universal grading must be either $[(1, 19), (2, 3)]$ or $[(2, 11), (1, 3)]$. The latter case is discarded by Lemma~\ref{lemma: rank not divisible by p}.

    This leaves $\sizeGC=1, 3, 25$ as possible, with the desired conditions for the case with $\sizeGC=3$, completing the proof.
\end{proof}

Next, we show that an MTC of rank $25$ with $\sizeGC=3$ must be equivalent to the category $\Rep(D^\omega(\mathbb{Z}_7\rtimes\mathbb{Z}_3))$ for some $3$-cocycle $\omega$. Let $\C$ be an MTC of rank $25$ with $\sizeGC=3$. Let the elements of $\mathcal{G}(\mathcal{C})$ be $1$, $g$, and $g^2$. Then, we know that $\rank(\mathcal{C}_\mathrm{ad})=19$, $\rank(\mathcal{C}_g)=3$, and $\rank(\mathcal{C}_{g^2})=3$. 

\begin{lemma}\label{thm:deequiv}
    Let $\mathcal{C}$ be an odd-dimensional MTC of rank $25$ with $3$ invertible objects. Suppose that there are $t$ simple objects in $\mathcal{C}$ of Frobenius-Perron dimension $3$, and that $r$ of them are not fixed by $\mathcal{G}(\mathcal{C})$. Then $t$ is a multiple of $6$, and $3t-3r+1$ divides $\frac{\FPdim(\mathcal{C})}{9}$.
\end{lemma}
\begin{proof}
    Clearly, all the simple objects of Frobenius-Perron dimension $3$ are in $\Cad$ (otherwise, $\FPdim(\Cad)=\FPdim(\C_g)=27$, while $\Cad$ has $16$ non-invertible simple objects, which is impossible). Consider the de-equivariantization $(\Cad)_{\mathbb{Z}_3}$ of $\Cad$ by $\mathbb{Z}_3$, see Section~\ref{section: equivariantization}. We know $\FPdim((\Cad)_{\mathbb{Z}_3})=\frac{\FPdim(\C)}{9}$. We count the number of invertible objects in $(\Cad)_{\mathbb{Z}_3}$, which we know must be a divisor of $\FPdim((\Cad)_{\mathbb{Z}_3})$.

    Consider the simple objects in $\Cad$. In order for an invertible object in $(\Cad)_{\mathbb{Z}_3}$ to be created by a simple object $X\in\Cad$, either $X$ is not fixed and $\FPdim(X)=1$, or $X$ is fixed and $\FPdim(X)=3$. For the former case, the only invertible object is created by $1, g, g^2\in\Cad$, resulting in one invertible object. For the latter case, we know there are $t$ simple objects of Frobenius-Perron dimension $3$ in $\Cad$. As a result, as the non-fixed objects come in sets of $6$ (an orbit of size $3$ of the action by $\GC$ and the dual of the orbit), $r$ is a multiple of $6$, and there are $t-r$ fixed simple objects of Frobenius-Perron dimension $3$ in $\Cad$. But each of these objects gives rise to $3$ simple objects of Frobenius-Perron dimension $1$ in $(\Cad)_{\mathbb{Z}_3}$, resulting in an additional $3t-3r$ invertible objects.

    Hence, there are a total of $3t-3r+1$ invertible objects in $(\Cad)_{\mathbb{Z}_3}$. As $(\Cad)_{\mathbb{Z}_3}$ is modular, this must divide $\FPdim((\Cad)_{\mathbb{Z}_3})=\frac{\FPdim(\C)}{9}$, as desired.
\end{proof}

Let $\C$ be a fusion category, and let its non-invertible simple objects be $X_1, X_1^*, \dots, X_k, X_k^*$ and their Frobenius-Perron dimensions be $\di_1, \di_1, \dots, \di_k, \di_k$. Recall that we call the array $[\di_1, \di_2, \dots, \di_k]$ the \emph{dimension array} of $\C$.

We now generate all potential dimension arrays of $\C$ using Algorithm~\ref{alg: dim generation} for $\rank(\C)=25$ and $\GC=3$. This produces the cases listed in Table~\ref{table: rank 25 dim arrays}, where the columns numbered $1$ through $11$ represent the elements of the array, i.e., the dimensions of the non-invertible simple objects, excluding duplicates due to duals.

We discard all of these potential arrays on a case-by-case basis, except for the one in row 34, which is realizable as $\Rep(D^\omega(\mathbb Z_7\rtimes \mathbb Z_3))$.

\begin{table}
    \centering
    \scalebox{0.8}{
    \begin{tabular}{ c | c | c | c | c | c | c | c | c | c | c | c | c }
        \# & $\FPdim(\C)$ & 1 & 2 & 3 & 4 & 5 & 6 & 7 & 8 & 9 & 10 & 11 \\
        \hline
        1 & 18275625 & 1425 & 1425 & 1425 & 1425 & 855 & 475 & 225 & 75 & 45 & 19 & 5 \\
        2 & 95355225 & 3255 & 3255 & 3255 & 3255 & 1953 & 1085 & 465 & 279 & 105 & 31 & 5 \\
        3 & 300155625 & 5775 & 5775 & 5775 & 5775 & 3465 & 1925 & 825 & 385 & 275 & 231 & 75 \\
        4 & 99225 & 105 & 105 & 105 & 105 & 63 & 35 & 15 & 7 & 5 & 3 & 3 \\
        5 & 300155625 & 5775 & 5775 & 5775 & 5775 & 3465 & 1925 & 825 & 385 & 231 & 225 & 175 \\
        6 & 52490025 & 2415 & 2415 & 2415 & 2415 & 1449 & 805 & 315 & 207 & 161 & 21 & 7 \\
        7 & 6125625 & 825 & 825 & 825 & 825 & 495 & 275 & 99 & 99 & 5 & 5 & 3 \\
        8 & 6125625 & 825 & 825 & 825 & 825 & 495 & 275 & 99 & 75 & 55 & 33 & 11 \\
        9 & 164025 & 135 & 135 & 135 & 135 & 81 & 45 & 15 & 15 & 5 & 5 & 5 \\
        10 & 108056025 & 3465 & 3465 & 3465 & 3465 & 2079 & 1155 & 385 & 315 & 315 & 11 & 7 \\
        11 & 99225 & 105 & 105 & 105 & 105 & 63 & 35 & 9 & 9 & 9 & 7 & 5 \\
        12 & 108056025 & 3465 & 3465 & 3465 & 3465 & 2079 & 945 & 693 & 385 & 385 & 105 & 11 \\
        13 & 1334025 & 385 & 385 & 385 & 385 & 231 & 105 & 77 & 35 & 35 & 35 & 11 \\
        14 & 7868025 & 935 & 935 & 935 & 935 & 561 & 255 & 165 & 165 & 51 & 17 & 5 \\
        15 & 1334025 & 385 & 385 & 385 & 385 & 231 & 105 & 55 & 55 & 35 & 35 & 35 \\
        16 & 38025 & 65 & 65 & 65 & 65 & 39 & 15 & 13 & 13 & 3 & 3 & 3 \\
        17 & 27225 & 55 & 55 & 55 & 55 & 33 & 11 & 11 & 11 & 5 & 5 & 3 \\
        18 & 1625625 & 425 & 425 & 425 & 425 & 255 & 85 & 85 & 75 & 51 & 51 & 3 \\
        19 & 5625 & 25 & 25 & 25 & 25 & 15 & 5 & 5 & 3 & 3 & 3 & 3 \\
        20 & 31585464729 & 59241 & 59241 & 59241 & 59241 & 25389 & 25389 & 19747 & 8463 & 1953 & 403 & 49 \\
        21 & 95355225 & 3255 & 3255 & 3255 & 3255 & 1395 & 1395 & 1085 & 465 & 105 & 31 & 5 \\
        22 & 23532201 & 1617 & 1617 & 1617 & 1617 & 693 & 693 & 539 & 231 & 49 & 21 & 11 \\
        23 & 1432809 & 399 & 399 & 399 & 399 & 171 & 171 & 133 & 57 & 9 & 7 & 7 \\
        24 & 52490025 & 2415 & 2415 & 2415 & 2415 & 1035 & 1035 & 805 & 315 & 161 & 21 & 7 \\
        25 & 99225 & 105 & 105 & 105 & 105 & 45 & 45 & 35 & 9 & 9 & 7 & 5 \\
        26 & 9801 & 33 & 33 & 33 & 33 & 11 & 11 & 11 & 9 & 9 & 3 & 3 \\
        27 & 13689 & 39 & 39 & 39 & 39 & 13 & 13 & 13 & 9 & 9 & 9 & 3 \\
        28 & 2025 & 15 & 15 & 15 & 15 & 5 & 5 & 5 & 3 & 3 & 3 & 3 \\
        29 & 38025 & 65 & 65 & 65 & 39 & 39 & 39 & 39 & 15 & 3 & 3 & 3 \\
        30 & 27225 & 55 & 55 & 55 & 33 & 33 & 33 & 33 & 11 & 5 & 5 & 3 \\
        31 & 1625625 & 425 & 425 & 425 & 255 & 255 & 255 & 255 & 75 & 51 & 51 & 3 \\
        32 & 5625 & 25 & 25 & 25 & 15 & 15 & 15 & 15 & 3 & 3 & 3 & 3 \\
        33 & 2025 & 15 & 15 & 15 & 9 & 9 & 9 & 5 & 5 & 5 & 3 & 3 \\
        34 & 441 & 7 & 7 & 7 & 3 & 3 & 3 & 3 & 3 & 3 & 3 & 3 \\
        35 & 2025 & 15 & 15 & 9 & 9 & 9 & 9 & 9 & 9 & 5 & 5 & 5
    \end{tabular}}
    \caption{Potential dimension arrays for MTC $\C$ of rank $25$ with $3$ invertible objects. All discarded except row 34.}
    \label{table: rank 25 dim arrays}
\end{table}

\begin{lemma}\label{lemma: FPdim array}
    Let $\mathcal{C}$ be an odd-dimensional MTC of rank $25$ with $3$ invertible objects. Then the dimension array of $\mathcal{C}$ is $[7, 7, 7, 3, 3, 3, 3, 3, 3, 3, 3]$, with the Frobenius-Perron dimensions of all the non-invertible simple objects in $\Cad$ equal to $3$ and the Frobenius-Perron dimensions of all simple objects outside of $\Cad$ equal to $7$.
\end{lemma}
\begin{proof}
     Consider the potential dimension arrays listed by Table~\ref{table: rank 25 dim arrays}. By Corollary~\ref{corollary: same dim outside Cad}, we can eliminate case $35$, as we know that $\FPdim(\mathcal{C}_g)=\frac{1}{3}\FPdim(\mathcal{C})=675$, and hence the Frobenius-Perron dimension of all the simple objects outside of $\Cad$ must be $\sqrt{\frac{675}{3}}=15$, but there are only $4$ simple objects with Frobenius-Perron dimension $15$.

     Then, using Lemma~\ref{thm:fixeddims}, we can eliminate all remaining cases except for cases $9$, $26$, $27$, $28$, $29$, $31$, $32$, $33$, and $34$. Using Lemma~\ref{thm:deequiv}, we can eliminate cases $26$, $27$, $31$, and $33$, leaving cases $9$, $28$, $29$, $32$, and $34$.

     We now individually rule out cases $9$, $28$, $29$, and $32$. For these cases, we can easily determine the Frobenius-Perron dimensions of any simple object $Z$ outside of $\Cad$, again using $\FPdim(Z)^2=\frac{1}{3}\FPdim(\mathcal{C}_g)=\frac{1}{9}\FPdim(\mathcal{C}_\mathrm{ad})$ and Corollary~\ref{corollary: same dim outside Cad}.

     \textbf{Case $9$: $[135, 135, 135, 135, 81, 45, 15, 15, 5, 5, 5]$}. Using the same reasoning as for case $35$, the Frobenius-Perron dimensions of all of the simple objects outside of $\Cad$ must be $135$. As a result, the Frobenius-Perron dimensions from the above multiset that belong to $\Cad$ are $135, 81, 45, 15, 15, 5, 5, 5$. Notice that the objects with Frobenius-Perron dimension $5$ cannot be fixed by $g$ per Remark~\ref{remark: fixed fpdim divisible by p}, as they are not divisible by $3$. Label the simple objects $X_1, X_2, \dots, X_8$ from right-to-left in the array (hence, $\FPdim(X_1)=5, \FPdim(X_2)=5, \dots, \FPdim(X_7)=81, \FPdim(X_8)=135$). Consider the decomposition of $X_1\otimes X_1^*$, where we know that the coefficients of $X_i$ and $X_i^*$ are equal as $\mathcal{C}$ is a braided category:
     \[X_1\otimes X_1^*=1\oplus\bigoplus_{i=1}^8 N_i(X_i\oplus X_i^*).\]
     Taking the Frobenius-Perron dimension of both sides gives
     \[25=1+\sum_{i=1}^8 2N_i\FPdim(X_i).\]
     Clearly, $N_6$, $N_7$, and $N_8$ will be $0$, as $\FPdim(X_8)>\FPdim(X_7)>\FPdim(X_6)>25$. But as all of $X_1, \dots, X_5$ are multiples of $5$, taking the above equation mod $5$ gives us $0\equiv 1\pmod{5}$, a contradiction. As a result, we rule out case $9$.

     \textbf{Case $28$: $[15, 15, 15, 15, 5, 5, 5, 3, 3, 3, 3]$}. The Frobenius-Perron dimensions of all of the simple objects outside of $\Cad$ must be $15$. Label the simple objects of $\mathcal{C}_\mathrm{ad}$ as $X_1, X_2, \dots, X_8$ from right-to-left. By Lemma~\ref{thm:deequiv}, all of $X_1$, $X_2$, $X_3$, and $X_4$ must be fixed, creating $12$ invertible objects in $(\mathcal{C}_\mathrm{ad})_{\mathbb{Z}_3}$. Then, as $5$ is not divisible by $3$, $X_5$, $X_6$, and $X_7$ are not fixed, resulting in one simple object of Frobenius-Perron dimension $5$ in $(\mathcal{C}_\mathrm{ad})_{\mathbb{Z}_3}$. Finally, $X_8$ clearly has to be fixed, resulting in three simple objects of Frobenius-Perron dimension $5$ in $(\mathcal{C}_\mathrm{ad})_{\mathbb{Z}_3}$.
     
     As a result, we have $\rank((\mathcal{C}_\mathrm{ad})_{\mathbb{Z}_3})=33$ and $|\mathcal{G}((\mathcal{C}_\mathrm{ad})_{\mathbb{Z}_3})|=25$ (note that we must account for the objects created by the duals of the simple objects in $\mathcal{C}$). Consider the universal grading of $(\mathcal{C}_\mathrm{ad})_{\mathbb{Z}_3}$. By Lemma~\ref{thm:trivial}, $\rank((((\mathcal{C}_\mathrm{ad})_{\mathbb{Z}_3})_\mathrm{ad})_\mathrm{pt})>1$, so we must have $\rank(((\mathcal{C}_\mathrm{ad})_{\mathbb{Z}_3})_\mathrm{ad})=9$, forcing the rank of all of the remaining components to be $1$. But this is impossible by Lemma~\ref{lemma: rank more than 1}, so we discard this case.

     \textbf{Case $29$: $[65, 65, 65, 39, 39, 39, 39, 15, 3, 3, 3]$}. The Frobenius-Perron dimensions of the simple objects outside of $\Cad$ must be $65$. Label the simple objects of $\mathcal{C}_\mathrm{ad}$ as $X_1, X_2, \dots, X_8$ from right-to-left. By Lemma~\ref{thm:deequiv}, all of $X_1$, $X_2$, and $X_3$ are not fixed, resulting in a simple object of Frobenius-Perron dimension $3$ in $(\mathcal{C}_\mathrm{ad})_{\mathbb{Z}_3}$. But $3$ does not divide $\FPdim((\mathcal{C}_\mathrm{ad})_{\mathbb{Z}_3})=\frac{\FPdim(\mathcal{C})}{9}=4225$, a contradiction by~\cite[Corollary 7.2]{ENO2}. Therefore, we discard case $29$.

     \textbf{Case $32$: $[25, 25, 25, 15, 15, 15, 15, 3, 3, 3, 3]$}. The Frobenius-Perron dimensions of the simple objects outside of $\Cad$ must be $25$. Label the simple objcets of $\mathcal{C}_\mathrm{ad}$ as $X_1, X_2, \dots, X_8$. By Lemma~\ref{thm:deequiv}, all of $X_1$, $X_2$, $X_3$, and $X_4$ must be fixed by $g$. We also know that either all of $X_5, X_6, X_7, X_8$ are fixed by $g$, or one of them is fixed by $g$ and three of them are not. First suppose that three of them are not. This results in a $(\mathcal{C}_\mathrm{ad})_{\mathbb{Z}_3}$ of rank $33$ and with $25$ invertible objects (once again recalling that we must also account for the duals of $X_1, \dots, X_8$). As a result, just as we had for case $28$, we must have $\rank(\mathcal{C}_\mathrm{ad})=9$. But this contradicts~\cite[Lemma 5.1 (c)]{CP}: $33\geq 9+25+2\cdot p-3\geq 37$ (as $p\geq 3$).

     Thus, all of $X_5, X_6, X_7, X_8$ must be fixed. As a result, in $(\mathcal{C}_\mathrm{ad})_{\mathbb{Z}_3}$, we have $25$ simple objects of Frobenius-Perron dimension $1$ and $24$ simple objects of Frobenius-Perron dimension $5$, resulting in $25$ invertibles, an overall rank of $49$, and an overall Frobenius-Perron dimension of $625$. Consider the universal grading of $(\mathcal{C}_\mathrm{ad})_{\mathbb{Z}_3}$, with $|\mathcal{G}((\mathcal{C}_\mathrm{ad})_{\mathbb{Z}_3})|=25$ components and a total of $49$ simple objects in those components. As a result, as the Frobenius-Perron dimension of all the components are congruent mod $8$, the number of simple objects in each component must be one of the following three cases:
     \begin{itemize}
         \item {[(1, 25), (24, 1)]},
         \item {[(1, 17), (1, 9), (23, 1)]},
         \item {[(3, 9), (22, 1)]}.
     \end{itemize}
     
     In the first two cases, by Lemma~\ref{thm:trivial}, $\rank((((\Cad)_{\mathbb{Z}_3})_\mathrm{ad})_\mathrm{pt})>1$. Therefore, we can discard these cases by Lemma~\ref{lemma: rank more than 1}.

     Finally, consider the third case. By~\cite[Proposition 8.20]{ENO1}, all of the components of the universal grading must have the same Frobenius-Perron dimension. Specifically, as $\FPdim((\mathcal{C}_\mathrm{ad})_{\mathbb{Z}_3})=625$, the Frobenius-Perron dimension of each component must be $25$. This means that the simple object in every component that has exactly $1$ simple object must have Frobenius-Perron dimension $5$, resulting in $22$ simple objects with Frobenius-Perron dimension $5$ being used and $24-22=2$ remaining. These $2$ remaining simple objects must be in the components with $9$ simple objects, resulting in at least one component having Frobenius-Perron dimension greater than $25$. But the Frobenius-Perron dimension of this components must also be $25$, a contradiction.

     Hence, we have discarded case $32$. This discards all cases except for case $34$, completing the proof.
\end{proof}

Finally, using the information from Lemma~\ref{lemma: FPdim array}, we prove that $\C$ must be equivalent to $\Rep(D^\omega(\mathbb{Z}_7\rtimes\mathbb{Z}_3))$.

\begin{proposition}
    Let $\mathcal{C}$ be an odd-dimensional MTC of rank $25$ with $3$ invertible objects. Then $\mathcal{C}\cong\Rep(D^\omega(\mathbb{Z}_7\rtimes\mathbb{Z}_3))$ for some $3$-cocycle $\omega$.
\end{proposition}
\begin{proof}
Following Lemma~\ref{lemma: FPdim array}, we know that the only possible dimension array for $\C$ is given in row 34 in Table~\ref{table: rank 25 dim arrays}.  Then it follows from Proposition~\ref{proposition: semidirect product} that $\mathcal C \cong\Rep(D^\omega(\mathbb{Z}_7\rtimes\mathbb{Z}_3))$.
\end{proof}

Therefore, in summary, we have shown the following theorem.

\begin{theorem}
    Let $\C$ be an odd-dimensional MTC of rank $25$. Then $\C$ is either pointed, perfect, or equivalent to $\Rep(D^\omega(\mathbb{Z}_7\rtimes\mathbb{Z}_3))$ for some $3$-cocycle $\omega$.
\end{theorem}

\section{Odd-dimensional MTCs of rank 27--31} \label{sec: rank 27-31}

In this section, we prove that all odd-dimensional MTCs of rank between $27$ and $31$ must be either pointed or perfect. We start with a corollary to Lemma~\ref{lemma: rank linear combination}, applied to ranks $27$, $29$, and $31$.
\
\begin{corollary}\label{corollary: count invertibles 27-33}
    If $\C$ is an odd-dimensional MTC, then the possible values of $\sizeGC$ by rank are the following:
    \begin{itemize}
        \item $\rank(\mathcal{C})=27$: $|\mathcal{G}(\mathcal{C})|=1, 3, 9, 11, 19, 27$,
        \item $\rank(\mathcal{C})=29$: $|\mathcal{G}(\mathcal{C})|=1, 3, 5, 7, 13, 21, 29$,
        \item $\rank(\mathcal{C})=31$: $|\mathcal{G}(\mathcal{C})|=1, 3, 5, 7, 15, 23, 31$.
    \end{itemize}
\end{corollary}

We use this corollary to show the following result using casework.

\begin{theorem}\label{thm: rank 27 to 31}
    Let $\C$ be an odd-dimensional MTC of rank $27$, $29$, or $31$. Then $\C$ must be either pointed or perfect.
\end{theorem}
\begin{proof}
    We mirror the proof of Proposition~\ref{prop: 1 3 25 invertibles rank 25} for this proof. Specifically, we again perform casework based on the possible values of $\sizeGC$, listed in Corollary~\ref{corollary: count invertibles 27-33}. We also can again find the possible multisets of the ranks of the components of the universal grading for each $\sizeGC$ using the facts that $\rank(\C_g)\equiv\rank(\C_h)\pmod{8}$ for all $g, h\in\mathcal{G}(\mathcal{C})$ and $\sum_{g\in\mathcal{G}(\C)}\rank(\C_g)$ is $27$, $29$, or $31$.
    
    We can still use Lemmas~\ref{lemma: rank more than 1},~\ref{lemma: rank not divisible by p}, and~\ref{lemma: odd number of components of rank} to discard cases, as for every $\sizeGC>1$, we know that $\Cpt$ is non-trivial, and therefore, by Lemma~\ref{thm:trivial}, $\Cadpt$ is also non-trivial.

    First, consider $\rank(\C)=27$. Using Lemma~\ref{lemma: rank more than 1}, we can discard the following cases, where we list each case as a multiset of the ranks of the components of the universal grading:
    \begin{itemize}
        \item $\sizeGC=19$: [(1, 9), (18, 1)],
        \item $\sizeGC=11$: [(1, 17), (10, 1)], [(2, 9), (9, 1)],
        \item $\sizeGC=3$: [(1, 25), (2, 1)], [(1, 17), (1, 9), (1, 1)].
    \end{itemize}

    This leaves the cases [(9, 3)] with $\sizeGC=9$, and [(3, 9)] with $\sizeGC=3$.

    For the former remaining case, [(9, 3)], we know that $\GCad$ is a subgroup of $\GC$. Hence, as $\sizeGC=9$, $\sizeGCad\neq 1$ by Lemma~\ref{thm:trivial}, and $\rank(\Cad)=3$, we must have $\sizeGCad=3$. But this means that $\FPdim(\C_g)=\FPdim(\Cad)=1^2+1^2+1^2=3$ for all $g\in\GC$, i.e., $\C$ is pointed, a contradiction.
    
    For the latter remaining case, [(3, 9)], as in Section~\ref{section: algorithm}, let the non-invertible simple objects of such a category $\C$ be $X_1, X_1^*, \dots, X_{12}, X_{12}^*$ and their dimensions be $\di_1, \di_1, \dots, \di_{12}, \di_{12}$. Define $m_1=\frac{\FPdim(\C)}{\di_1^2}$ (the same definition as we had in Section~\ref{section: algorithm}), which must be an odd positive integer by~\cite[Theorem 2.11]{ENO2}. If $m_1=3$, then $\FPdim(\C)=3\di_1^2$, so $\FPdim(\C_g)=\di_1^2$ by~\cite[Proposition 8.20]{ENO1}. But the component containing the simple object $X_1$ clearly has Frobenius-Perron dimension greater than $\di_1^2$ (as there are $9$ simple objects in that component), a contradiction. As a result, we assume $m_1\neq 3$.

    In order to discard the cases for $m_1\neq 3$, we run Algorithm~\ref{alg: dim generation}, considering only the cases where $m_1\geq 5$. Doing so produces one solution set of Frobenius-Perron dimensions of the simple objects. Specifically, this solution has $\FPdim(\C)=2475$ and dimensions of non-invertible simple objects are
    \[15, 15, 15, 15, 15, 5, 5, 5, 3, 3, 3, 3.\]

    As a result, there are $10$ simple objects in $\mathcal{C}$ with Frobenius-Perron dimension $15$ (the five listed above and their duals). Hence, as there are three components in the universal grading, there must exist one such component with at least $4$ simple objects with Frobenius-Perron dimension $15$. This component therefore must have dimension at least $4\cdot 15^2=900$. But the dimension of each component is $\frac{2475}{3}=825$, a contradiction, so we cannot have [(3, 9)].

    Next, consider $\rank(\C)=29$. By Lemma~\ref{lemma: rank more than 1}, we discard the cases listed below:
    \begin{itemize}
        \item $\sizeGC=21$: [(1, 9), (20, 1)],
        \item $\sizeGC=13$: [(1, 17), (12, 1)], [(2, 9), (11, 1)].
    \end{itemize}

    Additionally, by Lemma~\ref{lemma: rank not divisible by p}, we discard the cases:
    \begin{itemize}
        \item $\sizeGC=7$: [(1, 11), (6, 3)],
        \item $\sizeGC=5$: [(1, 25), (4, 1)], [(1, 17), (1, 9), (3, 1)], [(3, 9), (2, 1)],
        \item $\sizeGC=3$: [(1, 15), (2, 7)].
    \end{itemize}
    
    As this discards all non-pointed and non-perfect cases, if $\rank(\C)=29$, $\C$ must be pointed or perfect.

    Finally, consider $\rank(\C)=31$. We discard the following cases by Lemma~\ref{lemma: rank more than 1}:
    \begin{itemize}
        \item $\sizeGC=23$: [(1, 9), (22, 1)],
        \item $\sizeGC=15$: [(1, 17), (14, 1)], [(2, 9), (13, 1)].
    \end{itemize}

    We also can discard the following cases by Lemma~\ref{lemma: rank not divisible by p}:
    \begin{itemize}
        \item $\sizeGC=7$: [(1, 25), (6, 1)], [(1, 17), (1, 9), (5, 1)], [(3, 9), (4, 1)],
        \item $\sizeGC=5$: [(1, 19), (4, 3)], [(2, 11), (3, 3)],
        \item $\sizeGC=3$: [(1, 21), (2, 5)], [(2, 13), (1, 5)].
    \end{itemize}

    This leaves just the pointed and perfect cases, as desired.
\end{proof}

\section{Odd-dimensional MTCs of rank 33--49}\label{sec: ranks 33-49}

In this section, we reduce the possible cases of odd-dimensional MTCs of rank between $33$ and $49$. We begin each case by determining all possible values of $\sizeGC$, which we find using Lemma~\ref{lemma: rank linear combination}. Then, for each of those values, we enumerate all possible multisets that contain the ranks of the components of the universal grading, and use this information to discard possibilities.

Additionally, we can often directly apply lemmas previously shown to quickly discard many cases. Specifically, the following are the methods, which are also listed in Appendix~\ref{section: case list}.

\begin{enumerate}[label=(\alph*)]
    \item Lemma~\ref{lemma: rank more than 1}\label{enumitem: 33-49 rank more than 1},
    \item Lemma~\ref{lemma: odd number of components of rank}\label{enumitem: 33-49 odd number of components of rank},
    \item Lemma~\ref{lemma: rank Cad 9} in conjunction with Lemma~\ref{lemma: odd number of components of rank},
    \item \cite[Proposition 5.6]{CP}, Proposition~\ref{lemma: fixed simple}, and Proposition~\ref{proposition: invertible implies equal rank}, all of which are in conjunction with Lemma~\ref{lemma: odd number of components of rank}.
\end{enumerate}

Note that we can use options~\ref{enumitem: 33-49 rank more than 1} and~\ref{enumitem: 33-49 odd number of components of rank} whenever $\sizeGC$ is larger than 1 (i.e., whenever the multiset representing $\C$ has length greater than $1$), since then $\Cpt$ is non-trivial and $\Cadpt$ is also non-trivial by Lemma~\ref{thm:trivial}.

All of the cases for each rank that can be discarded using the above options can be found in Appendix~\ref{section: case list}, and we present methods to manually discard some of the remaining cases in the following propositions.

\begin{proposition}\label{prop: rank 33}
    Let $\C$ be an odd-dimensional MTC of rank $33$. Then $\C$ must be pointed, perfect, or have $3$ invertible objects. Additionally, if $\sizeGC=3$, then the ranks of the components of the universal grading of $\C$ must be $[(1, 27), (2, 3)]$.
\end{proposition}
\begin{proof}
    We can discard all values of $\sizeGC$ except for $33$, $25$, $17$, $11$, $9$, $5$, $3$, and $1$ using Lemma~\ref{lemma: rank linear combination}. Then, we discard most cases directly using previously shown lemmas, which are displayed in Appendix~\ref{section: case list}. The following non-pointed and non-perfect cases remain:
    \begin{itemize}
        \item {[(1, 13), (4, 5)]},
        \item {[(1, 27), (2, 3)]}.
    \end{itemize}

    \textbf{Case [(1, 13), (4, 5)]}: by Lemma~\ref{lemma: odd number of components of rank}, we know that $\rank(\Cad)=13$. Since $\sizeGC=5$, by Corollary~\ref{corollary: same dim outside Cad}, we also know that the simple objects outside of $\Cad$ have Frobenius-Perron dimension $\di$ satisfying $\di^2=\frac{\FPdim(\C)}{25}$. Let $\di_1$ be the largest Frobenius-Perron dimension of any simple object in $\C$, corresponding to simple object $X_1$, and let $m_1=\frac{\FPdim(\C)}{\di_1^2}$, as in Algorithm~\ref{alg: dim generation}. If $X_1$ is not is $\Cad$, then $m_1=25$. If $X_1$ is in $\Cad$, then so is its dual, and $2\di_1^2<\FPdim(\Cad)=\frac{\FPdim(\C)}{5}$, so $m_1=\frac{\FPdim(\C)}{\di_1^2}>10$. By~\cite[Lemma 5.4]{CP}, we know that $m_1\equiv\rank(\C)\equiv1\pmod{8}$, so $m_1\geq 17$, and by $\FPdim(\C)=25\di^2$, $m_1$ must be a perfect square, so $m_1\geq 25$. Running the Algorithm~\ref{alg: dim generation} for rank $33$ with $5$ invertible objects only for cases with $m_1\geq 25$ produces no solutions, which discards this case. 

    This means that $\C$ must be pointed, perfect, or have $3$ invertible objects, with the last case represented by [(1, 27), (2, 3)], as desired.
\end{proof}

\begin{proposition}
    Let $\C$ be an odd-dimensional MTC of rank $35$. Then $\C$ must be
    \begin{enumerate}
        \item pointed,
        \item perfect, or
        \item the modular subcategory of $\mathcal Z(\operatorname{Vec}_{H_3}^{\omega})$ with 9 invertible objects and 26 simple objects of dimension 3, where $H_3$ denotes the Heisenberg group of order $3^3$.
    \end{enumerate}
\end{proposition}

\begin{proof}
    We can discard all values of $\sizeGC$ except for $35$, $27$, $19$, $11$, $9$, $7$, $5$, $3$, and $1$ using Lemma~\ref{lemma: rank linear combination}. Then, we discard most cases directly using previously shown lemmas, which are displayed in Appendix~\ref{section: case list}. The following non-pointed and non-perfect cases remain:
    \begin{itemize}
        \item {[(1, 11), (8, 3)]},
        \item {[(1, 17), (2, 9)]}.
    \end{itemize}

    \textbf{Case [(1, 11), (8, 3)]}: 
   
    since $\Cad$ has rank 11 and we know that $\sizeGCad=9$ by Corollary~\ref{corollary: three components rank Cad}, then there are only two noninvertible simple objects $X$ and $X^*$ in $\Cad$. Moreover, both $X$ and $X^*$ are fixed by the action of $\mathcal{G}(\Cad)$. Consider the decomposition  $$X\otimes X^* = \bigoplus_{g\in \mathcal{G}(\Cad)} g \oplus m (X\oplus X^*),$$ for some $m\geq 0$. Taking Frobenius-Perron dimension on the equation above, we get $\FPdim(X)^2 = 9+ 2m\FPdim(X)$. Solving the quadratic gives two potential dimensions for $X$: $\FPdim(X) =3$ or $\FPdim(X)=9$. But if $\FPdim(X)=9$ then $\FPdim(\Cad) = 2\cdot 9^2+9 = 3^2\cdot 19$, which is not possible. In fact, let $\mathcal C_g$ be any other component. Then $\C_g$ has 3 simple objects with the same dimension $d$, see~\cite[Lemma 5.1]{CP}, and so $3^2\cdot 19=\FPdim(\C_g)=3d^2$, a contradiction. Now, if $\FPdim(X) =3$ then $\FPdim(\Cad) = 2\cdot 3^2+9 = 3^2\cdot 3$ and  $\FPdim(\mathcal C) = 3^5$. Hence $\C$ is the modular subcategory of $\mathcal Z(\operatorname{Vec}_{H_3}^{\omega})$ with 9 invertible objects and 26 simple objects of dimension 3, where $H_3$ denotes the Heisenberg group of order $3^3$, see~\cite[Example 4]{CPS} and~\cite[Theorem 3.5]{CP}.

\textbf{Case [(1, 17), (2, 9)]}: note that by Lemma~\ref{thm:trivial}, $|\mathcal{G}(\mathcal C)|=|\mathcal{G}(\mathcal C_{\mathrm{ad}})|=3$.

We use Algorithm~\ref{alg: Cad dim generation} to compute the potential Frobenius-Perron dimensions of non-invertible simple objects in $\mathcal C_{\mathrm{ad}}$. This produces the cases listed in Table~\ref{table: rank 35 dim arrays}, which exhibits these dimensions in columns 1 through 7, excluding duplicates due to duals.

We will discard all potential arrays on a case-by-case basis.

\begin{table}[h]
	\begin{center}
		\begin{tabular}{ c | c | c | c | c | c | c | c | c }
			\# & $\FPdim(\C)$ & 1 & 2 & 3 & 4 & 5 & 6 & 7 \\
			\hline
			1 & $3^2 \cdot 5^2\cdot 19$	&	15&	15&	15&	3&	3	&3&	3 \\
			2 & $3^3 \cdot 5^2$	&5	&5	&5	&3	&3	&3	&3\\
			3 &$3^2 \cdot 43$&		3	&3&	3&	3&	3&	3&	3 \\
		\end{tabular}
		\caption{Potential dimension arrays of $\Cad$ for $\C$ an MTC of rank $35$ with $3$ invertible objects, $\rank(\Cad)=17$, and $\sizeGCad=3$. All discarded.}
		\label{table: rank 35 dim arrays}
	\end{center}
\end{table}

Recall that MTCs of Frobenius-Perron dimension $mp^2$, for $m$ a square free integer, are pointed. 
In case 1, we would have that $(\mathcal C_{\mathrm{ad}})_{\mathbb Z_3}$ is an MTC of Frobenius-Perron dimension $5^2\cdot 19$, see Section~\ref{section: equivariantization}. It follows that $(\mathcal C_{\mathrm{ad}})_{\mathbb Z_3}$ is pointed, which is not possible. In fact, the simple objects of Frobenius-Perron dimension 15 in $\mathcal C_{\mathrm{ad}}$ would generate simple objects of dimensions either 15 or 5 in $(\mathcal C_{\mathrm{ad}})_{\mathbb Z_3}$. Case 2 is discarded in the same way, since $(\mathcal C_{\mathrm{ad}})_{\mathbb Z_3}$ is an MTC of Frobenius-Perron dimension $5^2\cdot 3$, hence pointed, but simple objects of dimension 5 in $\mathcal C_{\mathrm{ad}}$ generate objects of dimension 5 in $(\mathcal C_{\mathrm{ad}})_{\mathbb Z_3}$.  Lastly, in case 3, $\FPdim(\mathcal C)=3^2\cdot 43$ and so $\mathcal C$ should be pointed, a contradiction.
\end{proof}

\begin{proposition}
    Let $\C$ be an odd-dimensional MTC of rank $37$. Then $\C$ must be either pointed or perfect.
\end{proposition}
\begin{proof}
    We can discard all values of $\sizeGC$ except for $37$, $29$, $21$, $13$, $7$, $5$, $3$, and $1$ using Lemma~\ref{lemma: rank linear combination}. Then, we discard most cases directly using previously shown lemmas, which are displayed in Appendix~\ref{section: case list}. The following non-pointed and non-perfect case remains:
    \begin{itemize}
        \item {[(2, 15), (1, 7)]}.
    \end{itemize}

    \textbf{Case [(2, 15), (1, 7)]}: assume for contradiction that an MTC corresponding to this case exists. By Lemma~\ref{lemma: rank not divisible by p}, we know that $\rank(\Cad)=7$. We also know that $|\mathcal{G}(\Cad)|=3$. Consider the de-equivariantization of $\Cad$ by $\mathbb Z_3$. Since there are only $4$ non-invertible objects in $\Cad$ and the number of non-fixed objects must be a multiple of $6$ by Remark~\ref{remark: fixed fpdim divisible by p}, all must be fixed by the action of $\GC$ on the simple objects of $\Cad$. Therefore, each non-invertible simple object in $\Cad$ creates three simple objects in $(\Cad)_{\mathbb Z_3}$, so the rank of $(\Cad)_{\mathbb Z_3}$ must be $1+4\cdot 3=13$.
    
    Since $(\Cad)_{\mathbb Z_3}$ is modular (see Section~\ref{section: equivariantization}), it must be pointed by~\cite[Theorem 6.3 (a)]{CP}. This means that $\FPdim((\Cad)_{\mathbb Z_3})=13$ and $\FPdim(\Cad)=39$. This must be equal to the Frobenius-Perron dimension of any of the two non-adjoint components $\C_g$ of the universal grading of $\C$. The Frobenius-Perron dimension of all simple objects outside of $\Cad$ must be at least $3$ since all invertible objects are in $\Cad$. Since the rank of the non-adjoint components is $15$, the Frobenius-Perron dimension of a non-adjoint component $\C_g$ is at least $15\cdot 9 = 135 > 39$, a contradiction.

    This means that $\C$ must be pointed or perfect, as desired.
\end{proof}

\begin{proposition} \label{proposition: rank 39}
    Let $\C$ be an odd-dimensional MTC of rank $39$. Then $\C$ must be either pointed or perfect.
\end{proposition}
\begin{proof}
    We can discard all values of $\sizeGC$ except for $39$, $31$, $23$, $15$, $13$, $7$, $5$, $3$, and $1$ using Lemma~\ref{lemma: rank linear combination}. Then, we discard all cases directly using previously shown lemmas, which are displayed in Appendix~\ref{section: case list}. Hence, as all non-pointed and non-perfect cases are discarded, $\C$ must be pointed or perfect, as desired.

\end{proof}

\begin{proposition}
  Let $\C$ be an odd-dimensional MTC of rank $41$. Then $\C$ is either pointed, perfect, or
\begin{itemize}
    \item  $\mathcal C$ has 5 invertible objects, and Frobenius-Perron dimension array as in row 15 of Table~\ref{table: rank 41 dim arrays}, or
    \item $\mathcal C$ has 3 invertible objects, and the ranks of the components of the universal grading of $\C$ are $[(1, 35), (2, 3)]$.
\end{itemize}
\end{proposition}
\begin{proof}
    We can discard all values of $\sizeGC$ except for $41$, $33$, $25$, $17$, $11$, $9$, $5$, $3$, and $1$ using Lemma~\ref{lemma: rank linear combination}. Then, we discard most cases directly using previously shown lemmas, which are displayed in Appendix~\ref{section: case list}. The following non-pointed and non-perfect cases remain:
    \begin{itemize}
        \item {[(1, 21), (4, 5)]},
        \item {[(1, 35), (2, 3)]}.
    \end{itemize}

\textbf{Case [(1, 21), (4, 5)]}: note that by Lemma~\ref{thm:trivial}, $|\mathcal{G}(\mathcal C)|=|\mathcal{G}(\mathcal C_{\mathrm{ad}})|=5$. 

Using Algorithm~\ref{alg: dim generation}, we compute all potential cases for the Frobenius-Perron dimensions of simple objects in the category. Similarly to the case [(1, 13), (4, 5)] for rank $33$ in Proposition~\ref{prop: rank 33}, $m_1$ must be a perfect square that is at least $25$. Running Algorithm~\ref{alg: dim generation} for rank $41$ with $5$ invertible objects only for cases with $m_1\geq 25$ produces the arrays listed in Table~\ref{table: rank 41 dim arrays}. The first column holds the Frobenius-Perron dimension of the 20 objects outside $\mathcal C_{\mathrm{ad}}$ for each case. The rest of the columns show the Frobenius-Perron dimensions of the remaining non-invertible simple objects, excluding duplicates due to duals.

We will discard all potential arrays on a case-by-case basis.

\begin{table}[h]
	\begin{center}
		\begin{tabular}{ c | c | c | c | c | c | c | c | c | c  | c}
			\# & $\FPdim(\C)$ & 1 & 2 & 3 & 4 & 5 & 6 & 7 & 8 &9 \\
			\hline
			1 & $3^6 \cdot 5^2 \cdot 13^2$ & 351& 351 &351 & 195 & 135 & 65 & 27 & 15 & 13 \\
			2 & $3^4 \cdot 5^2 \cdot 7^2 \cdot 11^2 \cdot 13^2$& 9009&9009&9009& 5005& 3465 & 1365& 1287 & 63 & 35 \\
			3 &$3^4 \cdot 5^4 \cdot 7^2 \cdot 13^2$& 4095 &4095& 4095 & 2275 & 1365 & 819 & 819 & 63 & 13  \\
			4 & $3^6 \cdot 5^4$ & 135 &135& 135 & 75 & 45 & 27 & 25 & 9 & 5  \\
			5 & $3^2 \cdot 5^2 \cdot 7^2 \cdot 11^2$& 231 & 231&231 & 105 & 105 & 35 & 35 & 33 & 33 \\
			6 & $3^6 \cdot 5^2 \cdot 11^2$ & 297 & 297&297 & 135 & 99 & 99 & 55 & 55 & 15 \\
			7 & $3^2 \cdot 5^2\cdot 11^2$ &33 & 33&33 & 15 & 11 & 11 & 5 & 5 & 5  \\
			8 & $3^2 \cdot 5^2 \cdot 7^2 \cdot 13^2$ & 273&273& 273 & 105 & 91 & 91 & 91 & 35 & 13 \\
			9 & $3^2 \cdot 5^2 \cdot 7^2 \cdot 13^2$ & 273& 273&273 & 105 & 91 & 91 & 65 & 65 & 35\\
			10 & $3^2 \cdot 5^2 \cdot 17^2 \cdot 37^2$ & 1887 & 1887&1887 & 629 & 629 & 555 & 555 & 555 & 255 \\
			11 & $3^2 \cdot 5^2 \cdot 19^2$& 57	& 57&51&	19	&19	&15	&15	&15&	15 \\
			12 & $3^4 \cdot 5^2 \cdot 7^2 \cdot 11^2\cdot 13^2$& 9009&	9009&6435	&6435	&5005&	3465&	1365	&63	&35 \\
			13 & $3^2 \cdot 5^2 \cdot 7^2 \cdot 11^2$ &231	&231&165	&165	&105	&105&	35	&35	&33 \\
			14 & $3^2 \cdot 5^2 \cdot 7^2 \cdot 11^2$ & 231&165&165&	165	&165	&105	&105&	35	&35\\
			15 & $3^4 \cdot 5^2$ &9&5	&5&	5&	5&	5	&5	&5&	5
		\end{tabular}
		\caption{Potential dimension arrays for MTC $\C$ of rank $41$ with $5$ invertible objects, such that $m_1\geq 25$. All but line 15 are discarded.}
		\label{table: rank 41 dim arrays}
	\end{center}
\end{table}
Note that dimensions in columns 2--9 that appear less than five times in a row must correspond with simple objects that are fixed by the action of $\mathcal G(\mathcal C_{\mathrm{ad}})$. Hence cases 1--13 are discarded following Remark~\ref{remark: fixed fpdim divisible by p}. Case 15 is treated in Proposition~\ref{proposition: rank 41 invertibles 5}. 

It remains to look at case 14. All simple objects in $\mathcal C_{\mathrm{ad}}$ must be fixed, as no dimension in columns 2--9 is repeated 5 times. Consider the de-equivariantization $(\mathcal C_{\mathrm{ad}})_{\mathbb Z_5}$, which is again an MTC, see Section~\ref{section: equivariantization}. Non-invertible  simple objects in $(\mathcal C_{\mathrm{ad}})_{\mathbb Z_5}$ will have Frobenius-Perron dimensions 33, 21 and 7. Moreover, there is a unique invertible object in $(\mathcal C_{\mathrm{ad}})_{\mathbb Z_5}$. Let $X_1$ be one of the simple objects of Frobenius-Perron dimension $7$, and consider the decomposition 
\begin{align*}
	X_1\otimes X_1^{*}=1\oplus \bigoplus\limits_{i=1}^k N_{X_1,X_1^*}^{X_i} (X_i\oplus X_{i^*}),
\end{align*}
where $ 1,X_1, \dots, X_k, X_1^*, \dots, X_k^*$ denote all simple objects in $(\mathcal C_{\mathrm{ad}})_{\mathbb Z_5}$, and we are using that  $N_{X_1,X_1^*}^{X_i}=N_{X_1,X_1^*}^{X_i^*}$ for all $i=1,\dots, k$. Taking Frobenius-Perron dimension on both sides of the equation above, we get
\begin{align*}
	49=1+2\sum\limits_{i=1}^k N_{X_1,X_1^*}^{X_i} \FPdim(X_i).
\end{align*}
Hence $N_{X_1,X_1^*}^{X_i}=0$ for all $X_i$ with  $\FPdim(X_i)=33$ or $21$. But then 
\begin{align*}
	49=1+2m\cdot 7
\end{align*}
for some $m\geq 1$, a contradiction. Hence case 14 is also discarded.  

This means that $\C$ must be pointed, perfect, have $5$ invertible objects with Frobenius-Perron dimensions as in row 15 of Table~\ref{table: rank 41 dim arrays}, or have $3$ invertible objects, with the last case represented by [(1, 35), (2, 3)], as desired.
\end{proof}

\begin{proposition}\label{proposition: rank 41 invertibles 5}
    Suppose there exists an MTC $\C$ of rank 41 with 5 invertible objects and Frobenius-Perron dimension array as in row 15 in Table~\ref{table: rank 41 dim arrays}. Then $\C\cong \mathcal D^{\mathbb Z_5}$, where $\mathcal D$ is a categorification of the ring $R_{5,H}$ as defined in~\cite[Definition 1.3]{JL}, and $H$ is a finite abelian group of order $3^4$.
\end{proposition}

\begin{proof}
	Let $\mathcal C$ be as in the statement, and consider the Tannakian subcategory $(\mathcal C_{\mathrm{ad}})_{\mathrm{pt}}\cong \Rep(\mathbb Z_5)$. 
 
Note that all non-invertible simple objects in $\Cad$ are fixed by the action of $\GC\cong\Z_5$: otherwise, the rank of the de-equivariantization $(\Cad)_{\Z_5}$ (which is modular by Section~\ref{section: equivariantization}) would be of rank $33$ with $31$ invertible objects, which is impossible by Proposition~\ref{prop: rank 33}. Hence, $\mathcal C_{\mathbb Z_5}$ has 81 invertible objects. On the other hand, each non-trivial component $\mathcal C_g$ contains 5 simple objects which are not fixed by the action, see Proposition~\ref{lemma: fixed simple}. Hence the simple objects of each non-trivial component induce an object of dimension 9 in $\mathcal C_{\Z_5}$. That is, $\mathcal C_{\Z_5}$ is $\Z_5$-graded of rank 85, with 81 invertible objects in the trivial component, and 4 simple objects of dimension $9$. Let $(\mathcal C_{\Z_5})_{\mathrm{pt}}\cong \operatorname{Vec}_H^{\omega}$ for some abelian group $H$ of order 81 and 3-cocycle $\omega$. Then computing the fusion rules for $\mathcal C_{\Z_5}$ shows that $\mathcal K(\mathcal C_{\Z_5})\cong R_{5,H}$, where  $\mathcal K(\mathcal C_{\Z_5})$  denotes the Grothendieck ring of $\mathcal C_{\Z_5}$ and $R_{5,H}$ is the fusion ring  defined on~\cite[Definition 1.3]{JL}. 
 
We have shown that, if such a category $\mathcal C$ exists, its de-equivariantization $\mathcal C_{\mathbb Z_5}$ must be equivalent to  some categorification of $R_{5,H}$, and the statement follows. 
\end{proof}

\begin{remark}
Categorifications of $R_{5,H}$ were parametrized in~\cite[Proposition 3.1]{JL}. Let $\mathcal D$ be such a categorification. To obtain a category $\mathcal C$ as in row 15 in Table~\ref{table: rank 41 dim arrays}, by the lemma above we would need to have a (non-trivial) action of $\mathbb Z_5$  on $\mathcal D$. Since this induces an action of $\mathbb Z_5$ on $\mathcal D_{\mathrm{pt}}\cong \operatorname{Vec}_H$, we must have that $5||\operatorname{Aut}(H)|$. Since $H$ is a finite abelian group of order 81, it follows that we should have $H\cong \mathbb Z_3\times \mathbb Z_3\times \mathbb Z_3\times \mathbb Z_3$.
\end{remark}

\begin{proposition}
    Let $\C$ be an odd-dimensional MTC of rank $43$. Then $\C$ must be pointed, perfect, or have $9$ invertible objects and Frobenius-Perron dimension array of $\Cad$ as in one of the rows 5--8 or 11 of Table~\ref{table: rank 43 dim arrays 2}.
\end{proposition}
\begin{proof}
    We can discard all values of $\sizeGC$ except for $43$, $35$, $27$, $19$, $11$, $9$, $7$, $5$, $3$, and $1$ using Lemma~\ref{lemma: rank linear combination}. Then, we discard most cases directly using previously shown lemmas, which are displayed in Appendix~\ref{section: case list}. The following non-pointed and non-perfect cases remain:
    \begin{itemize}
        \item {[(1, 19), (8, 3)]},
        \item {[(2, 11), (7, 3)]},
        \item {[(1, 25), (2, 9)]}.
    \end{itemize}

\textbf{Case [(1, 19), (8, 3)]}: by Lemma~\ref{thm:trivial}, $|\mathcal{G}(\mathcal C_{\mathrm{ad}})|=3$ or 9. 

For the case $\sizeGCad=3$, since $\sizeGC=9=3\cdot\sizeGCad$ and $\rank(\Cad)=19$, this case is discarded by Corollary~\ref{corollary: three components rank Cad}.

 For the case $|\mathcal{G}(\mathcal C_{\mathrm{ad}})|=9$, we use Algorithm~\ref{alg: Cad dim generation} to compute the potential Frobenius-Perron dimensions of non-invertible simple objects in $\mathcal C_{\mathrm{ad}}$. This produces the cases listed in Table~\ref{table: rank 43 dim arrays 2}, which exhibits these dimensions in columns 1 through 5, excluding duplicates due to duals. We will discard potential arrays on a case-by-case basis.

\begin{table}[h]
	\begin{center}
		\begin{tabular}{ c | c | c | c | c | c | c }
			\# & $\FPdim(\C)$ & 1 & 2 & 3 & 4 & 5   \\
			\hline
			1& $3^6 \cdot 5^4 \cdot 19$		&675	&135&	75&	27&	27\\
			2&$3^4 \cdot 5^2 \cdot  19$	&	45&	9&	3&	3&	3\\
			3&$3^6 \cdot 7^2 \cdot 19$	&	189	&27	&27	&21	&9\\
			4&$3^6 \cdot 19$	&	27	&3&	3&	3	&3\\
			5&$3^7 \cdot 5^4$	&225&	135&	75&	27	&27\\
			6&$3^5 \cdot 5^2$		&15	&9&	3	&3&	3\\
			7&$3^7 \cdot 7^2$	&	63&	27	&27	&21	&9\\
			8&$3^7$	&9	&	3&	3&	3&	3\\
			9&$3^4 \cdot 43$	&	9&	9	&3&	3	&3\\
			10&$3^4 \cdot  59$	&	9&	9&	9&	3&	3\\
			11&$3^5 \cdot 5^2$	&	9&	9&	9&	9&	3\\
			12&$3^4 \cdot 7 \cdot 13$		&9	&9	&9&	9&	9\\
			13&$3^4 \cdot 11$	&3	&3&	3&	3&	3\\
		\end{tabular}
		\caption{Potential dimension arrays of $\Cad$ for $\C$ an MTC of rank $43$ with $9$ invertible objects, $\rank(\Cad)=19$, and $\sizeGCad=9$. All but lines 5--8 and 11 are discarded.}
		\label{table: rank 43 dim arrays 2}
	\end{center}
\end{table}

In each case, by Proposition~\ref{lemma: fixed simple}, the Frobenius-Perron dimensions of all simple objects in the components of rank $3$ are equal. As a result, we know that $\FPdim(\C)$ is of the form $27\di^2$, so $\frac{\FPdim(\C)}{27}$ must be a perfect square. Hence, we can discard cases 1--4, 9, 10, 12, and 13. This leaves rows 5--8 and 11, as desired.

\textbf{Case [(2, 11), (7, 3)]}: in this case $\sizeGC=9$. By Lemma~\ref{lemma: odd number of components of rank}, we know that $\rank(\Cad)=3$. Since $\Cadpt$ is non-trivial by Lemma~\ref{thm:trivial} and $\sizeGC=9$, we also have $|\mathcal{G}(\Cad)|=3$, so $\Cad$ is pointed and $\FPdim(\Cad)=3$. This is a contradiction since the Frobenius-Perron dimension of the components of rank $11$ must be greater than $3$.

\textbf{Case [(1, 25), (2, 9)]}: note that by Lemma~\ref{thm:trivial}, $|\mathcal{G}(\mathcal C)|=|\mathcal{G}(\mathcal C_{\mathrm{ad}})|=3$.

We use Algorithm~\ref{alg: Cad dim generation} to compute the potential Frobenius-Perron dimensions of non-invertible simple objects in $\mathcal C_{\mathrm{ad}}$. After removing the cases that violate Remark~\ref{remark: fixed fpdim divisible by p}, this produces the cases listed in Table~\ref{table: rank 43 dim arrays 3}, which exhibits these dimensions in columns 1 through 11, excluding duplicates due to duals. We will discard all potential arrays on a case-by-case basis.

\begin{table}[h]
	\begin{center}
		\begin{tabular}{ c | c | c | c | c | c | c | c |c | c| c | c | c| }
			\# & $\FPdim(\C)$ & 1 & 2 & 3 & 4 & 5  &6 &7 &8&9&10&11 \\
			\hline
			1&$3^2 \cdot 5^2 \cdot 11$	&	15	&5	&5	&5	&5&	5&	5&	3&	3&	3&	3\\
			2&$3^8\cdot 5^2 \cdot 19$		&405	&405	&405	&135	&81	&45	&15	&15	&5	&5	&5\\
			3&$3^4 \cdot 11^2 \cdot 19$&	99&	99&	99&	33	&11&	11&	11&	9&	9&	3&	3\\
			4&$3^4 \cdot 13^2 \cdot 19$	&	117	 &117	&117	&39	&13	&13&	13&	9&	9&	9	&3\\
			5&$3^4 \cdot 5^2 \cdot 19$	&	45	&45	&45	&15&	5	&5	&5	&3	&3	&3&	3\\
			6&$3^2 \cdot 5^2 \cdot 13^2 \cdot 19$	&195	&195	&195&	39&	39&	39&	39	&15	&3	&3	&3\\
			7&$3^2 \cdot 5^4 \cdot 17^2 \cdot 19$	&	1275	&1275	&1275&	255&	255&	255&	255&	75&	51&	51&	3\\
			8&$3^2 \cdot 5^4 \cdot 19$	&	75	&75&	75&	15&	15&	15&	15&	3&	3	&3	&3\\
			9&$3^4 \cdot 5^2 \cdot 19$	&45&	45&	45&	9&	9&	9&	5&	5&	5&	3&	3\\
			10&$3^2 \cdot 7^2 \cdot 19$	&	21&	21&	21&	3&	3	&3	&3	&3	&3	&3	&3\\
			11&$3^9 \cdot 5^2$		&135&	135&	135&	135&	81&	45&	15&	15&	5&	5&	5\\
			12&$3^5 \cdot 11^2$	&	33&	33&	33&	33&	11&	11&	11&	9&	9&	3&	3\\
			13&$3^5 \cdot  13^2$	&39&	39&	39&	39&	13&	13&	13&	9&	9&	9&	3\\
			14&$3^5 \cdot 5^2$		&15	&15&	15&	15&	5	&5	&5	&3	&3	&3&	3\\
			15&$3^3 \cdot 5^2 \cdot 13^2$	&	65&	65&	65&	39&	39&	39&	39&	15&	3&	3&	3\\
			16&$3^3 \cdot 5^4 \cdot 17^2$	&	425&	425&	425&	255&	255&	255&	255	&75	&51	&51	&3\\
			17&$3^3 \cdot 5^4$		&25&	25&	25&	15&	15&	15&	15&	3&	3&	3&	3\\
			18&$3^5 \cdot 5^2$		&15&	15&	15&	9&	9	&9	&5&	5&	5&	3&	3\\
			19&$3^3 \cdot 7^2$	&	7	&7&	7&	3&	3&	3&	3&	3&	3	&3	&3\\
			20&$3^5 \cdot 5^2$	&	15&	15&	9&	9	&9	&9	&9	&9	&5&	5&	5\\
			21&$3^2 \cdot 5^2 \cdot 43$	&	15&	15&	15&	15&	15&	15&	15	&3	&3	&3&	3\\
			22&$3^2 \cdot 67$	&	3	&3	&3	&3	&3	&3	&3	&3	&3	&3	&3
		\end{tabular}
		\caption{Potential dimension arrays of $\Cad$ for $\C$ an MTC of rank $43$ with $3$ invertible objects, $\rank(\Cad)=25$, and $\sizeGCad=3$. All discarded.}
		\label{table: rank 43 dim arrays 3}
	\end{center}
\end{table}
Dimensions in columns 1--11 that appear a number of times not divisible by 3 must correspond with simple objects that are fixed by the action of $\mathcal G(\mathcal C_{\mathrm{ad}})\cong \mathbb Z_3$. In case 3, this implies the simple objects of Frobenius-Perron dimension 3 must be fixed by the action. Hence for each one of these, we get 3 invertible objects in $(\mathcal C_{\mathrm{ad}})_{\mathbb Z_3}$, and so there are a total of 13 invertible objects in $(\mathcal C_{\mathrm{ad}})_{\mathbb Z_3}$. This is a contradiction, since 13 does not divide $\FPdim((\mathcal C_{\mathrm{ad}})_{\mathbb Z_3})=3^2\cdot 11^2\cdot 19.$ Cases 4, 7, 9, 12--13, 16 and 18 are discarded in the same way.

In cases 1, 8, 10, 17, 19 and 21,  $\FPdim((\mathcal C_{\mathrm{ad}})_{\mathbb Z_3})=mp^k,$ where $p$ is an odd prime, $k\leq 4$ and $m$ is an odd square-free integer. Since $(\mathcal C_{\mathrm{ad}})_{\mathbb Z_3}$ is modular, this implies it should be pointed, see~\cite{DN}. Hence the Frobenius-Perron dimension of all simple objects in $\mathcal C_{\mathrm{ad}}$ should be 3, which is not true in any of these cases.

In cases 2, 11 and 20, $\FPdim((\mathcal C_{\mathrm{ad}})_{\mathbb Z_3})=mp^kq^l,$ where $p$ and $q$ are odd primes, $k,l\geq 1$, and $m$ is an odd square-free integer. It follows that $(\mathcal C_{\mathrm{ad}})_{\mathbb Z_3}$ is weakly-group-theoretical and thus solvable, see~\cite{N1, NP}. Hence by~\cite[Proposition 4.5]{ENO2} $(\mathcal C_{\mathrm{ad}})_{\mathbb Z_3}$ should have a non-invertible simple object, which contradicts the fact that in all these cases, $\mathcal C_{\mathrm{ad}}$ has no simple object of Frobenius-Perron dimension 3.

In case 15, $\FPdim((\mathcal C_{\mathrm{ad}})_{\mathbb Z_3})=3\cdot 5^2 \cdot 13^2,$ and so again $(\mathcal C_{\mathrm{ad}})_{\mathbb Z_3}$ should be solvable. There are 6 simple objects of Frobenius-Perron dimension 3 in $\mathcal C_{\mathrm{ad}}$. By Remark~\ref{remark: fixed fpdim divisible by p}, either all of them are fixed, or none of them are. So $\mathcal G((\mathcal C_{\mathrm{ad}})_{\mathbb Z_3})$ has either size 1 or 19 , respectively. The former is not possible by~\cite[Proposition 4.5]{ENO2}. The latter is also not possible, since 19 does not divide $\FPdim((\mathcal C_{\mathrm{ad}})_{\mathbb Z_3})=3\cdot 5^2 \cdot 13^2$.

In case 5, $\FPdim((\mathcal C_{\mathrm{ad}})_{\mathbb Z_3})=3^2\cdot 5^2 \cdot 19$. Note that there are 8 simple objects of Frobenius-Perron dimension 3 in $\mathcal C_{\mathrm{ad}}$. By Remark~\ref{remark: fixed fpdim divisible by p}, either exactly 2 of them are fixed, or all of them are. So $(\mathcal C_{\mathrm{ad}})_{\mathbb Z_3}$ has either 25 or 7  invertible objects, respectively. The latter is not possible, since 7 does not divide $\FPdim((\mathcal C_{\mathrm{ad}})_{\mathbb Z_3})=3^2\cdot 5^2 \cdot 19$. If the former is true, then because 25 does not divide $\FPdim(((\mathcal C_{\mathrm{ad}})_{\mathbb Z_3})_{\mathrm{ad}})=3^2\cdot 19$, it must be the case that $((\mathcal C_{\mathrm{ad}})_{\mathbb Z_3})_{\mathrm{ad}}$ is perfect. But from~\cite[Theorem 1.6]{ENO2} we get that $((\mathcal C_{\mathrm{ad}})_{\mathbb Z_3})_{\mathrm{ad}}$  is solvable and so it should contain a non-trivial invertible object by~\cite[Proposition 4.5]{ENO2}, which is a contradiction. Cases 6 and 14 are discarded in the same way. 

Lastly, in case 22 we have that $\FPdim(\C)=3^2\cdot 67$, and so it should be pointed, a contradiction. 
\end{proof}

\begin{proposition}\label{prop: rank 45}
    Let $\C$ be an odd-dimensional MTC of rank $45$. Then $\C$ must be either pointed or perfect.
\end{proposition}
\begin{proof}
    We can discard all values of $\sizeGC$ except for $45$, $37$, $29$, $21$, $15$, $13$, $9$, $7$, $5$, $3$, and $1$ using Lemma~\ref{lemma: rank linear combination}. Then, we discard most cases directly using previously shown lemmas, which are displayed in Appendix~\ref{section: case list}. The following non-pointed and non-perfect cases remain:
    \begin{itemize}
        \item {[(15, 3)]},
        \item {[(9, 5)]},
        \item {[(3, 15)]}.
    \end{itemize}

    \textbf{Case [(15, 3)]}: since $\Cadpt$ is non-trivial, we know $\sizeGCad=3=\rank(\Cad)$, and $\FPdim(\Cad)=3$. That means that $\FPdim(\C_g)=3$ for all components $\C_g$ of the universal grading of $\C$. But that means that all simple objects of $\C$ are invertible, and $\C$ is pointed, a contradiction.

    \textbf{Case [(9, 5)]}: in this case $\sizeGC=9$, $\sizeGCad=3$, and $\GCad\cong\Z_3$. Consider the de-equivariantization $(\Cad)_{\Z_3}$ of $\Cad$ by $\GCad\cong \Z_3$. The three invertible objects of $\Cad$ create one simple object in $(\Cad)_{\Z_3}$. Each of the two non-invertible simple objects in $\Cad$ must be fixed under the action of $\GCad$ on the simple objects of $\Cad$, and as such, each of them creates three simple objects in $(\Cad)_{\Z_3}$. Therefore, $\rank((\Cad)_{\Z_3})=1+3\cdot 2 = 7$, and $(\Cad)_{\Z_3}$ is pointed by~\cite[Theorem 4.5]{BR}. This means that $\FPdim((\Cad)_{\Z_3})=7$ and $\FPdim(\Cad)=3\cdot 7 = 21$, which implies that $\FPdim(\C_g)=21$ for every component $\C_g$ of the universal grading of $\C$. However, there are only $6$ invertible objects in the $8$ non-adjoint components. That means that at least one of the components must contain no invertible objects, and the Frobenius-Perron dimensions of that component must be at least $5\cdot 9 > 21$, leading to a contradiction, and thus discarding the case. 
    
\textbf{Case [(3, 15)]}: note that by Lemma~\ref{thm:trivial}, $|\mathcal{G}(\mathcal C)|=|\mathcal{G}(\mathcal C_{\mathrm{ad}})|=3$.

We use Algorithm~\ref{alg: Cad dim generation} to compute the potential Frobenius-Perron dimensions of non-invertible simple objects in $\mathcal C_{\mathrm{ad}}$. This produces the cases listed in Table~\ref{table: rank 45 dim arrays}, which exhibits these dimensions in columns 1 through 6, excluding duplicates due to duals.

We will discard all potential arrays on a case-by-case basis.

\begin{table}[h]
	\begin{center}
		\begin{tabular}{ c | c | c | c | c | c | c | c }
			\# & $\FPdim(\C)$ & 1 & 2 & 3 & 4 & 5 & 6  \\
			\hline
			1 &$ 3^2 \cdot 5^2 \cdot 13$	&15	&15	&3	&3	&3&	3 \\
			2 & $3^2 \cdot 37$	&	3&	3	&3	&3	&3	&3\\
		\end{tabular}
		\caption{Potential dimension arrays of $\Cad$ for $\C$ an MTC of rank $45$ with $3$ invertible objects, $\rank(\Cad)=15$, and $\sizeGCad=3$. All discarded.}
		\label{table: rank 45 dim arrays}
	\end{center}
\end{table}

In case 1, we would have that $(\mathcal C_{\mathrm{ad}})_{\mathbb Z_3}$ is an MTC of Frobenius-Perron dimension $5^2\cdot 13$, see Section~\ref{section: equivariantization}. It follows that $(\mathcal C_{\mathrm{ad}})_{\mathbb Z_3}$ is pointed, which is not possible. In fact, the simple objects of Frobenius-Perron dimension 15 in $\mathcal C_{\mathrm{ad}}$ would generate simple objects of dimensions either 15 or 5 in $(\mathcal C_{\mathrm{ad}})_{\mathbb Z_3}$. Lastly, in case 2, $\FPdim(\mathcal C)=3^2\cdot 37$ and so $\mathcal C$ should be pointed, a contradiction.

This means that $\C$ must be pointed or perfect, as desired.
\end{proof}

\begin{proposition}
    Let $\C$ be an odd-dimensional MTC of rank $47$. Then $\C$ must be either pointed or perfect.
\end{proposition}
\begin{proof}
    We can discard all values of $\sizeGC$ except for $47$, $39$, $31$, $23$, $15$, $13$, $7$, $5$, $3$, and $1$ using Lemma~\ref{lemma: rank linear combination}. Then, we discard most cases directly using previously shown lemmas, which are displayed in Appendix~\ref{section: case list}. The following non-pointed and non-perfect case remains:
    \begin{itemize}
        \item {[(2, 21), (1, 5)]}.
    \end{itemize}

    \textbf{Case [(2, 21), (1, 5)]}: similarly to the case [(9, 5)] in the proof of Proposition~\ref{prop: rank 45} for rank $45$, considering the de-equivariantization of $\Cad$ by $\GCad\cong\Z_3$, we know that $\FPdim(\Cad)=21$. However, since all invertible objects are in $\Cad$, the non-adjoint components of the universal grading of $\C$ must have Frobenius-Perron dimension of at least $21\cdot 9 > 21$, leading to a contradiction, and discarding this case. 

    This means that $\C$ must be pointed or perfect, as desired.
\end{proof}

\begin{proposition}\label{proposition: rank 49}
    Let $\C$ be an odd-dimensional MTC of rank $49$. Then $\C$ must be pointed, perfect, or have $3$ or $5$ invertible objects. Additionally, if $\sizeGC=3$, then the ranks of the components of the universal grading of $\C$ must be [(1, 43), (2, 3)], and if $\sizeGC=5$, then the ranks must be [(1, 29), (4, 5)].
\end{proposition}
\begin{proof}
    We can discard all values of $\sizeGC$ except for $49$, $41$, $33$, $25$, $17$, $11$, $9$, $7$, $5$, $3$, and $1$ using Lemma~\ref{lemma: rank linear combination}. Then, we discard most cases directly using previously shown lemmas, which are displayed in Appendix~\ref{section: case list}. The following non-pointed and non-perfect cases remain:
    \begin{itemize}
       
        \item {[(7, 7)]},
        \item {[(1, 29), (4, 5)]},
        \item {[(1, 43),(2, 3)]}.
    \end{itemize}

    \textbf{Case [(7, 7)]}: since $\Cadpt$ is non-trivial and $\sizeGC=7$, we must have $\sizeGCad=7=\rank(\Cad)$, which means $\Cad$ is pointed. Then the Frobenius-Perron dimension of all components must be $7$, and $\C$ is pointed, which is a contradiction.

    This means that $\C$ must be pointed, perfect, or have $3$ or $5$ invertible objects, with the last cases represented by [(1, 43), (2, 3)] and [(1, 29), (4, 5)], respectively, as desired.
\end{proof}

The next proposition provides restrictions on an odd-dimensional MTC $\C$ of rank $49$ with $\sizeGC=5$.

\begin{proposition}
    Let $\C$ be an odd-dimensional MTC of rank $49$ with $5$ invertible objects. If $7\nmid\FPdim(\C)$, then $\mathcal C\cong\Rep(D^\omega(\Z_{11}\rtimes\Z_5))$, with $\Z_{11}\rtimes\Z_5$ non-abelian and $\omega$ a $3$-cocycle.
\end{proposition}
\begin{proof}    
    By Proposition~\ref{proposition: rank 49}, the multiset of the ranks of the universal grading components of $\C$ is represented by [(1, 29), (4, 5)]. We perform casework based on whether there exists a simple object in $\Cad$ that is not fixed by the action of $\GCad$ on the simple objects of $\Cad$. 

    First, consider the case in which all non-invertible simple objects in $\Cad$ are fixed by that action. The Frobenius-Perron dimensions of all non-invertible simple objects in $\Cad$ must be divisible by $5$. Define $\di_i$ and $m_i=\frac{\FPdim(\C)}{\di_i^2}$ for each $i$, as in Section~\ref{section: adjoint algorithm}. We show that none of the values of $m_i$ can be divisible by $5$.

    Suppose for the sake of contradiction that $5\mid m_i$ for some $i$. As $5\mid\di_i$, we know that $\FPdim(\C)=m_i\di_i^2$ is a multiple of $125$, so $\FPdim(\Cad)=\frac{\FPdim(\C)}{5}$ is a multiple of $25$. But we also know that $\FPdim(\Cad)=5+2\di_1^2+\dots+2\di_{12}^2\equiv 5\pmod{25}$, as each $\di_i$ is divisible by $5$, a contradiction.

    Additionally, as $\FPdim(\C)$ is not divisible by $7$, we know that $m_1\neq 49$. Similarly to the case [(1, 13), (4, 5)] for rank $33$ in Proposition~\ref{prop: rank 33}, $m_1$ must be a perfect square that is at least $25$. We use Algorithm~\ref{alg: Cad dim generation} modified to consider only the cases where all values of $m_i$ are not divisible by $5$, $m_1\neq 49$, and $m_1>25$ to compute the potential Frobenius-Perron dimensions of non-invertible simple objects in $\Cad$ (note that we skip $m_1=25$ as $m_1$ cannot be divisible by $5$). This produces the cases listed in Table~\ref{table: rank 49 dim arrays all divisible by 5}, which exhibits these dimensions in columns $1$ through $12$, excluding duplicates due to duals.

    \begin{table}[h]
	\begin{center}
            \scalebox{0.8}{
		\begin{tabular}{ c | c | c | c | c | c | c | c | c | c | c | c | c | c }
			\# & $\FPdim(\C)$ & 1 & 2 & 3 & 4 & 5 & 6 & 7 & 8 & 9 & 10 & 11 & 12 \\
			\hline
			1 & $3^4\cdot 5^2\cdot 29^2$ & 145 & 145 & 145 & 145 & 145 & 145 & 145 & 145 & 45 & 5 & 5 & 5 \\
                2 & $3^8\cdot 5^2\cdot 11^2$ & 495 & 495 & 495 & 495 & 495 & 495 & 495 & 495 & 135 & 55 & 55 & 15 \\
                3 & $3^4\cdot 5^2\cdot 11^2$ & 55 & 55 & 55 & 55 & 55 & 55 & 55 & 55 & 15 & 5 & 5 & 5 \\
                4 & $3^4\cdot 5^2\cdot 7^2\cdot 37^2$ & 1295 & 1295 & 1295 & 1295 & 1295 & 1295 & 1295 & 1295 & 315 & 185 & 185 & 5 \\
                5 & $3^4\cdot 5^2\cdot 7^2\cdot 13^2$ & 455 & 455 & 455 & 455 & 455 & 455 & 455 & 455 & 105 & 65 & 65 & 35 \\
                6 & $3^4\cdot 5^2\cdot 17^2\cdot 37^2$ & 3145 & 3145 & 3145 & 3145 & 3145 & 3145 & 3145 & 3145 & 555 & 555 & 555 & 255 \\
                7 & $3^4\cdot 5^2\cdot 19^2$ & 95 & 95 & 95 & 95 & 95 & 95 & 95 & 95 & 15 & 15 & 15 & 15 \\
                8 & $3^4\cdot 5^2\cdot 11^2$ & 55 & 55 & 55 & 55 & 55 & 55 & 45 & 45 & 45 & 15 & 5 & 5 \\
                9 & $3^4\cdot 5^2\cdot 11^2$ & 55 & 55 & 55 & 55 & 45 & 45 & 45 & 45 & 45 & 45 & 15 & 5 \\
                10 & $3^4\cdot 5^2\cdot 11^2$ & 55 & 55 & 45 & 45 & 45 & 45 & 45 & 45 & 45 & 45 & 45 & 15 \\
                11 & $5^2\cdot 11^2$ & 5 & 5 & 5 & 5 & 5 & 5 & 5 & 5 & 5 & 5 & 5 & 5
		\end{tabular}}
		\caption{Potential dimension arrays of $\Cad$ for $\C$ an MTC of rank $49$ with $5$ invertible objects, $\rank(\Cad)=29$, and $\sizeGCad=5$, where all values of $m_i$ are not divisible by $5$, $m_1\neq 49$, and $m_1>25$. All but line 11 discarded.}
		\label{table: rank 49 dim arrays all divisible by 5}
	\end{center}
    \end{table}

    First, note that we do not consider cases $4$ and $5$, as their Frobenius-Perron dimensions are divisible by $7$.

    Now, consider cases $1$, $3$, and $8$, and consider the de-equivariantizations $(\Cad)_{\Z_5}$, which are modular (see Section~\ref{section: equivariantization}). In these cases, there are $31$, $31$, and $21$ invertible objects in $(\Cad)_{\Z_5}$, respectively (as each non-invertible object of Frobenius-Perron dimension $5$ creates $5$ invertible objects in $(\Cad)_{\Z_5}$). However, $\FPdim((\Cad)_{\Z_5})=\frac{\FPdim(\C)}{25}$ is not divisible by this count in any of the cases, a contradiction. Thus, these cases are discarded.

    Next, we discard cases $2$, $7$, and $10$. In each case, $\FPdim((\Cad)_{\Z_5})=\frac{\FPdim(\C)}{25}$ is of the form $p^aq^b$ for primes $p$ and $q$, so $(\Cad)_{\Z_5}$ is solvable~\cite[Theorem 1.6]{ENO2}. Thus, by~\cite[Proposition 4.5 (iv)]{ENO2}, $(\Cad)_{\Z_5}$ must contain a non-trivial invertible object. But this is a contradiction, as there are no simple objects of Frobenius-Perron dimension $5$ in $\Cad$ for each case, discarding these cases.

    We proceed to case $6$. Non-invertible simple objects in $(\mathcal C_{\mathrm{ad}})_{\mathbb Z_5}$ will have Frobenius-Perron dimensions 629, 111, 51, and possibly 3145. Moreover, there is $1$ invertible object in $(\mathcal C_{\mathrm{ad}})_{\mathbb Z_5}$. Let $X_1$ be one of the simple objects of Frobenius-Perron dimension $51$. Then consider the decomposition 
    \begin{align*}
    	X_1\otimes X_1^{*}=1\oplus \bigoplus\limits_{i=1}^k N_{X_1,X_1^*}^{X_i} (X_i\oplus X_{i^*}),
    \end{align*}
    where $X_1, \dots, X_k, X_1^*, \dots, X_k^*$ denote all non-invertible simple objects in $(\mathcal C_{\mathrm{ad}})_{\mathbb Z_5}$, and we are using that  $N_{X_1,X_1^*}^{X_i}=N_{X_1,X_1^*}^{X_i^*}$ for all $i=1,\dots, k$. Taking Frobenius-Perron dimension on both sides of the equation above, we get
    \begin{align*}
    	51^2=1+2\sum\limits_{i=1}^k N_{X_1,X_1^*}^{X_i} \FPdim(X_i).
    \end{align*}
    But then 
    \begin{align*}
    	2601=1+2n_1\cdot 629+2n_2\cdot 111+2n_3\cdot 51
    \end{align*}
    for some $n_1, n_2, n_3\geq 0$. As there are no solutions to this equation, we have a contradiction.

    The last case we discard is case $9$. Non-invertible simple objects in $(\mathcal C_{\mathrm{ad}})_{\mathbb Z_5}$ will have Frobenius-Perron dimensions 11, 9, 3, and possibly 45. Moreover, there are $11$ invertible objects in $(\mathcal C_{\mathrm{ad}})_{\mathbb Z_5}$. Let $X_1$ be one of the simple objects of Frobenius-Perron dimension $3$, and suppose that it is fixed by only $1$ of the invertible objects. Then consider the decomposition 
    \begin{align*}
    	X_1\otimes X_1^{*}=1\oplus \bigoplus\limits_{i=1}^k N_{X_1,X_1^*}^{X_i} (X_i+X_{i^*}),
    \end{align*}
    where $X_1, \dots, X_k, X_1^*, \dots, X_k^*$ denote all non-invertible simple objects in $(\mathcal C_{\mathrm{ad}})_{\mathbb Z_5}$, and we are using that  $N_{X_1,X_1^*}^{X_i}=N_{X_1,X_1^*}^{X_i^*}$ for all $i=1,\dots, k$. Taking Frobenius-Perron dimension on both sides of the equation above, we get
    \begin{align*}
    	9=1+2\sum\limits_{i=1}^k N_{X_1,X_1^*}^{X_i} \FPdim(X_i).
    \end{align*}
    Hence $N_{X_1,X_1^*}^{X_i}=0$ for all $X_i$ with  $\FPdim(X_i)=11$ or $9$. But then 
    \begin{align*}
    	9=1+2m\cdot 3
    \end{align*}
    for some $m\geq 0$, a contradiction. 
    
    Next, suppose that $X_1$ is fixed by all $11$ of the invertible objects. Consider the decomposition 
    \begin{align*}
    	X_1\otimes X_1^{*}=\bigoplus\limits_{g\in G[X]} g\oplus \bigoplus\limits_{i=1}^k N_{X_1,X_1^*}^{X_i} (X_i\oplus X_{i^*}),
    \end{align*}
    where $X_1, \dots, X_k, X_1^*, \dots, X_k^*$ denote all non-invertible simple objects in $(\mathcal C_{\mathrm{ad}})_{\mathbb Z_5}$. Taking Frobenius-Perron dimension on both sides of the equation above, we get
    \begin{align*}
    	9=11+2\sum\limits_{i=1}^k N_{X_1,X_1^*}^{X_i} \FPdim(X_i),
    \end{align*}
    which is clearly a contradiction.

    Finally, consider case $11$.  It follows from Proposition~\ref{proposition: semidirect product} that $\mathcal C \cong\Rep(D^\omega(\mathbb{Z}_{11}\rtimes\mathbb{Z}_5))$.
   
    Now, we move to the case in which there exists a non-invertible simple object in $\Cad$ that is not fixed by the action of $\GCad$ on the simple objects of $\Cad$. Define $\di_i$ and $m_i$ similarly to the previous case for each $i$. Then, we know that at least $5$ of the values of $\di_i$ are equal.
    
    Similarly to the case [(1, 13), (4, 5)] for rank $33$ in Proposition~\ref{prop: rank 33}, $m_1$ must be a perfect square that is at least $25$. We use Algorithm~\ref{alg: Cad dim generation} modified to consider only the cases where at least five consecutive values $\di_i, \di_{i+1}, \dots, \di_{i+4}$ are equal and $m_1$ is a perfect square that is at least $25$ to compute the potential Frobenius-Perron dimensions of non-invertible simple objects in $\Cad$. This produces the cases listed in Table~\ref{table: rank 49 dim arrays at least 5 equal}, which exhibits these dimensions in columns $1$ through $12$, excluding duplicates due to duals.

    \begin{table}[h]
	\begin{center}
		\scalebox{0.8}{
            \begin{tabular}{ c | c | c | c | c | c | c | c | c | c | c | c | c | c }
			\# & $\FPdim(\C)$ & 1 & 2 & 3 & 4 & 5 & 6 & 7 & 8 & 9 & 10 & 11 & 12 \\
			\hline
			1 & $3^4\cdot 5^2\cdot 29^2$ & 145 & 145 & 145 & 145 & 145 & 145 & 145 & 145 & 45 & 5 & 5 & 5 \\
                2 & $3^8\cdot 5^2\cdot 11^2$ & 495 & 495 & 495 & 495 & 495 & 495 & 495 & 495 & 135 & 55 & 55 & 15 \\
                3 & $3^4\cdot 5^2\cdot 11^2$ & 55 & 55 & 55 & 55 & 55 & 55 & 55 & 55 & 15 & 5 & 5 & 5 \\
                4 & $3^4\cdot 5^2\cdot 7^2\cdot 37^2$ & 1295 & 1295 & 1295 & 1295 & 1295 & 1295 & 1295 & 1295 & 315 & 185 & 185 & 5 \\
                5 & $3^4\cdot 5^2\cdot 7^2\cdot 13^2$ & 455 & 455 & 455 & 455 & 455 & 455 & 455 & 455 & 105 & 65 & 65 & 35 \\
                6 & $3^4\cdot 5^2\cdot 17^2\cdot 37^2$ & 3145 & 3145 & 3145 & 3145 & 3145 & 3145 & 3145 & 3145 & 555 & 555 & 555 & 255 \\
                7 & $3^4\cdot 5^2\cdot 19^2$ & 95 & 95 & 95 & 95 & 95 & 95 & 95 & 95 & 15 & 15 & 15 & 15 \\
                8 & $3^4\cdot 5^2\cdot 11^2$ & 55 & 55 & 55 & 55 & 55 & 55 & 45 & 45 & 45 & 15 & 5 & 5 \\
                9 & $5^2\cdot 11^2$ & 5 & 5 & 5 & 5 & 5 & 5 & 5 & 5 & 5 & 5 & 5 & 5 \\
                10 & $3^4\cdot 5^2\cdot 7^2\cdot 37^2$ & 1665 & 1295 & 1295 & 1295 & 1295 & 1295 & 1295 & 555 & 555 & 315 & 185 & 5 \\
                11 & $3^4\cdot 5^2\cdot 7^2\cdot 13^2$ & 585 & 455 & 455 & 455 & 455 & 455 & 455 & 195 & 195 & 105 & 65 & 35 \\
                12 & $3^6\cdot 5^2\cdot 7^2$ & 135 & 105 & 105 & 105 & 105 & 105 & 105 & 45 & 35 & 35 & 15 & 15 \\
                13 & $3^6\cdot 5^2\cdot 7^2$ & 135 & 105 & 105 & 105 & 105 & 105 & 105 & 35 & 35 & 35 & 35 & 5 \\
                14 & $3^4\cdot 5^2\cdot 7^2\cdot 11^2\cdot 17^2$ & 8415 & 6545 & 6545 & 6545 & 6545 & 6545 & 5355 & 5355 & 1785 & 1155 & 315 & 105 \\
                15 & $5^2\cdot 7^2\cdot 13^2$ & 65 & 65 & 35 & 35 & 35 & 35 & 35 & 35 & 35 & 35 & 35 & 35 \\
                16 & $3^4\cdot 5^2\cdot 11^2$ & 55 & 55 & 45 & 45 & 45 & 45 & 45 & 45 & 45 & 45 & 45 & 15 \\
                17 & $3^2\cdot 5^2\cdot 7^2$ & 15 & 15 & 15 & 15 & 5 & 5 & 5 & 5 & 5 & 5 & 5 & 5 \\
                18 & $3^4\cdot 5^2\cdot 11^2$ & 55 & 55 & 55 & 55 & 45 & 45 & 45 & 45 & 45 & 45 & 15 & 5 \\
                19 & $3^6\cdot 5^2\cdot 7^2$ & 135 & 135 & 135 & 105 & 105 & 45 & 45 & 45 & 45 & 45 & 35 & 35 \\
                20 & $5^2\cdot 7^2\cdot 11^2$ & 55 & 55 & 55 & 55 & 35 & 35 & 7 & 7 & 7 & 7 & 7 & 5 \\
                21 & $5^2\cdot 7^2\cdot 13^2$ & 65 & 65 & 65 & 65 & 35 & 35 & 35 & 5 & 5 & 5 & 5 & 5
                
		\end{tabular}}
		\caption{Potential dimension arrays of $\Cad$ for $\C$ an MTC  of rank $49$ with $5$ invertible objects, $\rank(\Cad)=29$, and $\sizeGCad=5$, where at least $5$ values of $\di_i$ are equal and $m_1\geq 25$ is a perfect square. All discarded.}
		\label{table: rank 49 dim arrays at least 5 equal}
	\end{center}
    \end{table}

    We disregard cases 4, 5, 10--15, 17, and 19--21, as in those cases $\FPdim(\C)$ is divisible by $7$.

    In cases 1--3, 6--9, 16, and 18, the Frobenius-Perron dimensions of all non-invertible objects in $\Cad$ are divisible by $5$. We show that all these objects must be fixed under the action of $\GCad$ on the simple objects of $\Cad$. In all these cases, $\FPdim((\Cad)_{\Z_5})=\frac{\FPdim(\C)}{25}$ is not divisible by $5$. If a simple object $X$ in $\Cad$ were not fixed by the action, then there would be a simple object in $(\Cad)_{\Z_5}$ of dimension $\FPdim(X)$, which is divisible by $5$ and is hence not a divisor of $\FPdim((\Cad)_{\Z_5})$, resulting in a contradiction~\cite[Theorem 2.11]{ENO2}. Thus, all simple objects in $\Cad$ are fixed by the action.

    Now, note that all of the cases we are considering now (1--3, 6--9, 16, and 18) already appeared in Table~\ref{table: rank 49 dim arrays all divisible by 5}, and we handled them when we assumed that all simple objects in $\Cad$ were fixed.

    Hence, as we discarded all cases other than the desired existing one, this completes the proof.
\end{proof}

We summarize the results from this section in the following theorem.

\begin{theorem}
    Let $\C$ be an odd-dimensional MTC such that $33\leq\rank(\C)\leq 49$.
    \begin{enumerate}[label=(\alph*)]
        \item If $\rank(\C)\in \{37,39,45,47\}$, then $\C$ is either pointed or perfect.
        \item If $\rank(\C)=33$, then $\C$ is pointed, perfect, or has $3$ invertible objects. 
         \item If $\rank(\C)=35$, then $\C$ is pointed, perfect, or  the modular subcategory of $\mathcal Z(\operatorname{Vec}_{H_3}^{\omega})$ with 9 invertible objects and 26 simple objects of dimension 3, where $H_3$ denotes the Heisenberg group of order $3^3$.
          \item If $\rank(\C)=41$, then $\C$ is pointed, perfect, or has 3 or 5 invertible objects. Moreover, if the latter case exists, then $\C$ should be equivalent to $\mathcal D^{\mathbb Z_5}$, where $\mathcal D$ is a categorification of the ring $R_{5,H}$ as defined in~\cite[Definition 1.3]{JL}, and $H$ is a finite abelian group of order $3^4$.
          \item If $\rank(\C)=43$, then $\C$ is pointed, perfect, or has $9$ invertible objects.
        \item If $\rank(\C)=49$, then $\C$ is pointed, perfect, or has $3$ or $5$ invertible objects. Additionally, if $\C$ has $5$ invertible objects and $7\nmid\FPdim(\C)$, then $\C\cong\Rep(D^\omega(\Z_{11}\rtimes\Z_5))$ with $\Z_{11}\rtimes\Z_5$ non-abelian and for some $3$-cocycle $\omega$.
    \end{enumerate}
\end{theorem}

\section{Odd-dimensional MTCs of rank 51--73}\label{sec: ranks 51-73}

We perform the same procedures throughout this section as we did in Section~\ref{sec: ranks 33-49}. Specifically, we begin by finding the possible values of $\sizeGC$ using Lemma~\ref{lemma: rank linear combination}, and then find all possible multisets that contain the ranks of the components of the universal grading. We use this information to discard possibilities. Once again, many of the cases can be directly discarded using previous lemmas, which are also listed in Appendix~\ref{section: case list}.

\begin{enumerate}[label=(\alph*)]
    \item Lemma~\ref{lemma: rank more than 1}\label{enumitem: 51-73 rank more than 1},
    \item Lemma~\ref{lemma: odd number of components of rank}\label{enumitem: 51-73 odd number of components of rank},
    \item Lemma~\ref{lemma: rank Cad 9} in conjunction with Lemma~\ref{lemma: odd number of components of rank},
    \item \cite[Proposition 5.6]{CP}, Proposition~\ref{lemma: fixed simple}, and Proposition~\ref{proposition: invertible implies equal rank}, all of which are in conjunction with Lemma~\ref{lemma: odd number of components of rank}.
\end{enumerate}

We can use options~\ref{enumitem: 51-73 rank more than 1} and~\ref{enumitem: 51-73 odd number of components of rank} whenever $\sizeGC$ is larger than 1 (i.e., whenever the multiset representing $\C$ has length greater than $1$), since then $\Cpt$ is non-trivial and $\Cadpt$ is also non-trivial by Lemma~\ref{thm:trivial}.

All cases that we discard using the above lemmas are listed in Appendix~\ref{section: case list}. In the rest of this section, we present manual methods to discard many of the remaining cases for each rank.

\begin{proposition}
    Let $\C$ be an odd-dimensional MTC of rank $51$. Then $\C$ must be pointed, perfect, or have $3$ or $9$ invertible objects. Additionally, if $\sizeGC=3$, then the ranks of the components of the universal grading of $\C$ must be $[(1, 33), (2, 9)]$, and if $\sizeGC=9$, then those ranks must be $[(1, 27), (8, 3)]$ or  $[(3, 11), (6, 3)]$. 
\end{proposition}
\begin{proof}
    We can discard all values of $\sizeGC$ except for $51$, $43$, $35$, $27$, $19$, $17$, $11$, $9$, $7$, $5$, $3$, and $1$ using Lemma~\ref{lemma: rank linear combination}. Then, we discard most cases directly using previously shown lemmas, which are displayed in Appendix~\ref{section: case list}. The non-pointed and non-perfect cases that remain are $[(1, 27), (8, 3)]$, $[(3, 11), (6, 3)]$, and $[(1, 33), (2, 9)]$.
   
\end{proof}

\begin{proposition}
    Let $\C$ be an odd-dimensional MTC of rank $53$. Then $\C$ must be pointed or perfect.
\end{proposition}
\begin{proof}
    We can discard all values of $\sizeGC$ except for $53$, $45$, $37$, $29$, $21$, $15$, $13$, $9$, $7$, $5$, $3$, and $1$ using Lemma~\ref{lemma: rank linear combination}. Then, we discard most cases directly using previously shown lemmas, which are displayed in Appendix~\ref{section: case list}. The following non-pointed and non-perfect cases remain:
    \begin{itemize}
        \item {[(1, 23), (2, 15)]}. 
    \end{itemize}

    \textbf{Case [(1, 23), (2, 15)]}: By Lemma~\ref{thm:trivial}, $\sizeGCad=3$. We use Algorithm~\ref{alg: Cad dim generation} to compute the potential Frobenius-Perron dimensions of non-invertible simple objects in $\mathcal C_{\mathrm{ad}}$. After removing the cases that violate Remark~\ref{remark: fixed fpdim divisible by p}, this produces the cases listed in Table~\ref{table: rank 53 with 3 invertibles}, which exhibits these dimensions in columns 1 through 10, excluding duplicates due to duals.
    \begin{table}[h]
        \begin{center}
            \begin{tabular}{ c | c | c | c | c | c | c | c | c | c | c | c }
                \# & $\FPdim(\C)$ & 1 & 2 & 3 & 4 & 5 & 6 & 7 & 8 & 9 & 10 \\
                \hline
                1 & $3^8 \cdot 5^2 \cdot 13$ & 405 & 405 & 135 & 81 & 45 & 15 & 15 & 5 & 5 & 5 \\
                2 & $3^4 \cdot 11^2 \cdot 13$ & 99 & 99 & 33 & 11 & 11 & 11 & 9 & 9 & 3 & 3 \\
                3 & $3^4 \cdot 13^3$ & 117 & 117 & 39 & 13 & 13 & 13 & 9 & 9 & 9 & 3 \\
                4 & $3^4 \cdot 5^2 \cdot 13$ & 45 & 45 & 15 & 5 & 5 & 5 & 3 & 3 & 3 & 3 \\
                5 & $3^2 \cdot 5^2 \cdot 13^3$ & 195 & 195 & 39 & 39 & 39 & 39 & 15 & 3 & 3 & 3 \\
                6 & $3^2 \cdot 5^4 \cdot 13 \cdot 17^2$ & 1275 & 1275 & 255 & 255 & 255 & 255 & 75 & 51 & 51 & 3 \\
                7 & $3^2 \cdot 5^4 \cdot 13$ & 75 & 75 & 15 & 15 & 15 & 15 & 3 & 3 & 3 & 3 \\
                8 & $3^4 \cdot 5^2 \cdot 13$ & 45 & 45 & 9 & 9 & 9 & 5 & 5 & 5 & 3 & 3 \\
                9 & $3^2 \cdot 7^2 \cdot 13$ & 21 & 21 & 3 & 3 & 3 & 3 & 3 & 3 & 3 & 3 \\
                10 & $3^3 \cdot 5^2 \cdot 7$ & 15 & 15 & 15 & 5 & 5 & 5 & 3 & 3 & 3 & 3 \\
                11 & $3^2 \cdot 5^2 \cdot 37$ & 15 & 15 & 15 & 15 & 15 & 15 & 3 & 3 & 3 & 3 \\
                12 & $3^2 \cdot 5^3$ & 5 & 5 & 5 & 5 & 5 & 5 & 3 & 3 & 3 & 3 \\
                13 & $3^2 \cdot 61$ & 3 & 3 & 3 & 3 & 3 & 3 & 3 & 3 & 3 & 3 \\
            \end{tabular}
           \caption{Potential dimension arrays of $\Cad$ for $\C$ an MTC of rank $53$ with $3$ invertible objects, $\rank(\Cad)=23$, and $\sizeGCad=3$. All discarded.}
            \label{table: rank 53 with 3 invertibles}
        \end{center}
    \end{table}

    We will discard all potential arrays on a case-by-case basis. In case $1$, we know that the Frobenius-Perron dimension of the modular category $(\Cad)_{\Z_3}$ is $3^6\cdot 5^2\cdot 13$, so $(\Cad)_{\Z_3}$ is solvable by~\cite[Corollary 5.4]{N2}. Additionally, there does not exist a simple object of dimension $3$ in $\Cad$, so $(\Cad)_{\Z_3}$ is perfect. But this contradicts~\cite[Proposition 4.5 (iv)]{ENO2}. In case $2$, the modular category $(\Cad)_{\Z_3}$ has rank $45$ and $13$ invertible objects, a contradiction by Proposition~\ref{prop: rank 45}. 
    
    Cases $3$ and $6$ are impossible as in each case, $(\Cad)_{\Z_3}$ has $7$ invertible objects, which does not divide its dimension.

    In case $4$, we know that there must be either $7$ or $25$ invertible objects in the modular category $(\Cad)_{\Z_3}$. Since $7$ does not divide its Frobenius-Perron dimension of $3^2\cdot 5^2\cdot 13$, there must be $25$ invertible objects. Additionally, in this case, $(\Cad)_{\Z_3}$ must have rank $45$. But this contradicts Proposition~\ref{prop: rank 45}, discarding this case.

    In case $5$, as the Frobenius-Perron dimension of the modular category $(\Cad)_{\Z_3}$ is $5^2\cdot 13^3$, so $(\Cad)_{\Z_3}$ must be solvable by~\cite[Theorem 1.6]{ENO2}. As a result, by~\cite[Proposition 4.5 (iv)]{ENO2}, $(\Cad)_{\Z_3}$ cannot be perfect, so it must have $19$ invertible objects. However, it has rank either $29$ or $45$, a contradiction by Theorem~\ref{thm: rank 27 to 31} and Proposition~\ref{prop: rank 45}.

    In cases $7$, $9$, and $12$, the Frobenius-Perron dimensions of the modular category $(\Cad)_{\Z_3}$ are $5^4\cdot 13$, $7^2\cdot 13$, and $5^3$, respectively. Thus, in all of these cases $(\Cad)_{\Z_3}$ is pointed, by~\cite[Corollary 4.13]{DN},~\cite[Proposition 4.11]{NR}, and a generalization of~\cite[Lemma 4.11]{DN}, respectively. This means that each of their ranks must be equal to their dimensions. Additionally, in each case, the rank of $(\Cad)_{\Z_3}$ must be at most $3\cdot\rank(\Cad)=69$, i.e., when all simple objects in $\Cad$ are fixed under the action of $\GCad$ on them (note that this exact value is not even possible, as all invertible objects and simple objects with dimension not divisible by $3$ cannot be fixed). But this maximum possible rank is less than the actual rank in each case, hence discarding these cases.

    In case $8$, the modular category $(\Cad)_{\Z_3}$ has rank either $29$ or $45$ and $13$ invertible objects, a contradiction by Theorem~\ref{thm: rank 27 to 31} and Proposition~\ref{prop: rank 45}.

    In case $10$, the modular category $(\Cad)_{\Z_3}$ has rank $13$, $29$, or $45$ (as the simple objects of dimension $5$ in $\Cad$ cannot be fixed under the action of $\GCad$ on $\GC$) and at least $7$ invertible objects. Additionally, $(\Cad)_{\Z_3}$ is not pointed. But this is a contradiction, as there does not exist a non-pointed modular category of rank $13$, $29$, or $45$ with at least $7$ invertible objects by~\cite[Theorem 6.3 (a)]{CP}, Theorem~\ref{thm: rank 27 to 31}, and Proposition~\ref{prop: rank 45}, respectively.

    In case $11$, we know that the de-equivariantization $(\Cad)_{\Z_3}$ has Frobenius-Perron dimension $5^2\cdot 37$, so it must be pointed by~\cite[Theorem 4.11]{NR}. But $\Cad$ has simple objects of dimension $15$, so $(\Cad)_{\Z_3}$ cannot be pointed, a contradiction.

    Finally, in case $13$, $\C$ is pointed by~\cite[Proposition 4.11]{NR}, a contradiction (as there are simple objects of Frobenius-Perron dimension $3$).

    As a result, we have discarded all cases in Table~\ref{table: rank 53 with 3 invertibles}.
\end{proof}

\begin{proposition}
    Let $\C$ be an odd-dimensional MTC of rank $55$. Then $\C$ must be pointed or perfect.
\end{proposition}
\begin{proof}
    We can discard all values of $\sizeGC$ except for $55$, $47$, $39$, $31$, $23$, $15$, $13$, $11$, $7$, $5$, $3$, and $1$ using Lemma~\ref{lemma: rank linear combination}. Then, we discard most cases directly using previously shown lemmas, which are displayed in Appendix~\ref{section: case list}. The following non-pointed and non-perfect cases remain:
    \begin{itemize}
        \item {[(2, 21), (1, 13)]}.
    \end{itemize}

    \textbf{Case [(2, 21), (1, 13)]}: we use Algorithm~\ref{alg: Cad dim generation} to compute the potential Frobenius-Perron dimensions of non-invertible simple objects in $\mathcal C_{\mathrm{ad}}$. After removing the cases that violate Remark~\ref{remark: fixed fpdim divisible by p}, this produces the cases listed in Table~\ref{table: rank 55}, which exhibits these dimensions in columns 1 through 5, excluding duplicates due to duals.

\begin{table}[h]
    \begin{center}
        \begin{tabular}{ c | c | c | c | c | c | c }
            \# & $\FPdim(\C)$ & 1 & 2 & 3 & 4 & 5 \\
            \hline
            1 & $3^2 \cdot 5^2 \cdot 7$ & 15 & 3 & 3 & 3 & 3 \\
            2 & $3^2 \cdot 31$ & 3 & 3 & 3 & 3 & 3 \\
        \end{tabular}
        \caption{Potential dimension arrays of $\Cad$, for $\C$ MTC of rank $55$ with $3$ invertible objects, $\rank(\Cad)=13$, and $\sizeGCad=3$. All discarded.}
        \label{table: rank 55}
    \end{center}
\end{table}    

We will discard all potential arrays on a case-by-case basis. In case 1, the modular category $(\Cad)_{\mathbb Z_3}$ obtained by de-equivariantization has Frobenius-Perron dimension $5^2\cdot 7$, hence it is pointed. This contradicts that there are simple objects of dimension 15 in $\mathcal C_{\mathrm{ad}}$. In case 2, $\mathcal C$ should be pointed, see~\cite[Proposition 4.11]{NR}, which contradicts that the simple objects in $\mathcal C_{\mathrm{ad}}$ have dimension 3. Thus all cases in Table~\ref{table: rank 55} are discarded.
\end{proof}

\begin{proposition}
    Let $\C$ be an odd-dimensional MTC of rank $57$. Then $\C$ must be pointed, perfect, or have $3$ or $5$ invertible objects. Additionally, if $\sizeGC=3$, then the ranks of the components of the universal grading of $\C$ must be $[(1, 51), (2, 3)]$, and if $\sizeGC=5$, then those ranks must be $[(1, 37), (4, 5)]$.
\end{proposition}
\begin{proof}
    We can discard all values of $\sizeGC$ except for $57$, $49$, $41$, $33$, $25$, $19$, $17$, $11$, $9$, $7$, $5$, $3$, and $1$ using Lemma~\ref{lemma: rank linear combination}. Then, we discard most cases directly using previously shown lemmas, which are displayed in Appendix~\ref{section: case list}. The following non-pointed and non-perfect cases remain:
    \begin{itemize}
        \item {[(3, 17), (6, 1)]},
        \item {[(1, 15), (6, 7)]},
        \item {[(2, 27), (1, 3)]},
        \item {[(1, 37), (4, 5)]},
        \item {[(1, 51), (2, 3)]}.
    \end{itemize}

    \textbf{Case [(3, 17), (6, 1)]}: by Lemma~\ref{thm:trivial}, $\sizeGCad=3$ or $\sizeGCad=9$. The latter can be discarded by~\cite[Proposition 5.6]{CP}. For the former, we use Algorithm~\ref{alg: Cad dim generation} to compute the potential Frobenius-Perron dimensions of non-invertible simple objects in $\mathcal C_{\mathrm{ad}}$. After removing the cases that violate Remark~\ref{remark: fixed fpdim divisible by p}, this produces the cases listed in Table~\ref{table: rank 57 with 9 invertibles}, which exhibits these dimensions in columns 1 through 7, excluding duplicates due to duals.
    
    \begin{table}[h]
        \begin{center}
            \begin{tabular}{ c | c | c | c | c | c | c | c | c }
                \# & $\FPdim(\C)$ & 1 & 2 & 3 & 4 & 5 & 6 & 7 \\
                \hline
                1 & $3^8 \cdot 5^2$ & 81 & 45 & 15 & 15 & 5 & 5 & 5 \\
                2 & $3^3 \cdot 5^2 \cdot 19$ & 15 & 15 & 15 & 3 & 3 & 3 & 3 \\
                3 & $3^4 \cdot 11^2$ & 11 & 11 & 11 & 9 & 9 & 3 & 3 \\
                4 & $3^4 \cdot 13^2$ & 13 & 13 & 13 & 9 & 9 & 9 & 3 \\
                5 & $3^4 \cdot 5^2$ & 5 & 5 & 5 & 3 & 3 & 3 & 3 \\
                6 & $3^3 \cdot 43$ & 3 & 3 & 3 & 3 & 3 & 3 & 3 \\
            \end{tabular}
             \caption{Potential dimension arrays of $\Cad$, for $\C$ an MTC  of rank $57$ with $9$ invertible objects, $\rank(\Cad)=17$, and $\sizeGCad=3$. All discarded.}
            \label{table: rank 57 with 9 invertibles}
        \end{center}
    \end{table}

    Recall that there are components of the universal grading with a unique simple object. Since $\sizeGC=9$, this implies that $\FPdim(\mathcal C)$ should be a perfect square. This discards rows 2 and 6. On the other hand, for the arrays in rows 3-5, we get that $\FPdim((\Cad)_{\mathbb Z_3})=3\cdot p^2$, for $p=11, 13$ or $5$, respectively,  and thus $(\Cad)_{\mathbb Z_3}$ should be pointed, see~\cite[Proposition 4.11]{NR}. This contradicts that there are simple objects of dimension 11, 13 and 5 in $\Cad,$ respectively. 

Lastly, to discard row 1, note that non-invertible simple objects in $(\Cad)_{\mathbb Z_3}$ have Frobenius-Perron dimensions 27, 15 or 5. Let $X_1$ be one of the simple objects of Frobenius-Perron dimension $5$. Then consider the decomposition 
    \begin{align*}
    	X_1\otimes X_1^{*}=1\oplus \bigoplus\limits_{i=1}^k N_{X_1,X_1^*}^{X_i} (X_i+X_{i^*}),
    \end{align*}
    where $X_1, \dots, X_k, X_1^*, \dots, X_k^*$ denote all non-invertible simple objects in $(\mathcal C_{\mathrm{ad}})_{\mathbb Z_3}$, and we are using that  $N_{X_1,X_1^*}^{X_i}=N_{X_1,X_1^*}^{X_i^*}$ for all $i=1,\dots, k$. Taking Frobenius-Perron dimension on both sides of the equation above, we get
    \begin{align*}
    	25=1+2\sum\limits_{i=1}^k N_{X_1,X_1^*}^{X_i} \FPdim(X_i).
    \end{align*}
    But then we must have
    \begin{align*}
    	25=1+2n_1\cdot 5,
    \end{align*}
    for some $n_1\geq 0$, which is not possible. Thus all cases in Table~\ref{table: rank 57 with 9 invertibles} are discarded.

    \textbf{Case [(1, 15), (6, 7)]}: we use Algorithm~\ref{alg: Cad dim generation} to compute the potential Frobenius-Perron dimensions of non-invertible simple objects in $\mathcal C_{\mathrm{ad}}$. After removing the cases that violate Remark~\ref{remark: fixed fpdim divisible by p}, this produces the unique case listed in Table~\ref{table: rank 57 with 7 invertibles}, which exhibits these dimensions in columns 1 through 4, excluding duplicates due to duals.
        \begin{table}[h]
        \begin{center}
            \begin{tabular}{ c | c | c | c | c | c }
                \# & $\FPdim(\C)$ & 1 & 2 & 3 & 4 \\
                \hline
                1 & $3 \cdot 7^2 \cdot 19$ & 7 & 7 & 7 & 7 \\
            \end{tabular}
             \caption{Potential dimension arrays of $\Cad$, for $\C$ an MTC of rank $57$ with $7$ invertible objects, $\rank(\Cad)=15$, and $\sizeGCad=7$. All discarded.}
            \label{table: rank 57 with 7 invertibles}
        \end{center}
    \end{table}
Since $\FPdim(C)=7^2\cdot d,$ for $d$ square free, then $\mathcal C$ should be pointed, see~\cite[Theorem 4.7, Corollary 4.13]{DN}. Hence this case is discarded.

    \textbf{Case [(2, 27), (1, 3)]}: since $\rank(\Cad)=3$ by Lemma~\ref{lemma: odd number of components of rank}, we know $\sizeGCad=3$. But this means that $\FPdim(\Cad)=3$, which must be equal to the Frobenius-Perron dimension of the remaining two components of the universal grading of $\C$. This is not possible since each of those components contains $27$ simple objects, which discards this case. 
\end{proof}

\begin{proposition}
    Let $\C$ be an odd-dimensional MTC of rank $59$. Then $\C$ must be pointed, perfect, or have $3$ or $9$ invertible objects. Additionally, if $\sizeGC=3$, then the ranks of the components of the universal grading of $\C$ must be $[(1, 41), (2, 9)]$, and if $\sizeGC=9$, then those ranks must be $[(1, 35), (8, 3)]$. 
\end{proposition}
\begin{proof}
    We can discard all values of $\sizeGC$ except for $59$, $51$, $43$, $35$, $27$, $19$, $17$, $11$, $9$, $7$, $5$, $3$, and $1$ using Lemma~\ref{lemma: rank linear combination}. Then, we discard most cases directly using previously shown lemmas, which are displayed in Appendix~\ref{section: case list}. The following non-pointed and non-perfect cases remain:
    \begin{itemize}
        \item {[(2, 19), (7, 3)]},
        \item {[(4, 11), (5, 3)]},
        \item {[(1, 35), (8, 3)]},
        \item {[(1, 41), (2, 9)]}.
    \end{itemize}

    \textbf{Case [(2, 19), (7, 3)]}: since $\rank(\Cad)=3$ by Lemma~\ref{lemma: odd number of components of rank}, we know $\sizeGCad=3$. But this means that $\FPdim(\Cad)=3$, which must be equal to the Frobenius-Perron dimension of the remaining components of the universal grading of $\C$. This is not possible for the components of rank $19$, which discards this case. 

    \textbf{Case [(4, 11), (5, 3)]}: similarly to the previous case, $\rank(\Cad)=3$ by Lemma~\ref{lemma: odd number of components of rank}, and we know $\sizeGCad=3$. This means that $\FPdim(\Cad)=3$, which must be equal to the Frobenius-Perron dimension of the remaining components of the universal grading of $\C$. This is not possible for the components of rank $11$, which discards this case. 
\end{proof}

\begin{proposition}
    Let $\C$ be an odd-dimensional MTC of rank $61$. Then $\C$ must be pointed, perfect, or have $3$ invertible objects. Additionally, if $\sizeGC=3$, then the ranks of the components of the universal grading of $\C$ must be $[(1, 31), (2, 15)]$. 
\end{proposition}
\begin{proof}
    We can discard all values of $\sizeGC$ except for $61$, $53$, $45$, $37$, $29$, $21$, $15$, $13$, $9$, $7$, $5$, $3$, and $1$ using Lemma~\ref{lemma: rank linear combination}. Then, we discard most cases directly using previously shown lemmas, which are displayed in Appendix~\ref{section: case list}. The following non-pointed and non-perfect cases remain:
    \begin{itemize}
        \item {[(1, 19), (14, 3)]},
        \item {[(2, 11), (13, 3)]},
        \item {[(2, 13), (7, 5)]},
        \item {[(1, 31), (2, 15)]}.
    \end{itemize}

    \textbf{Case [(1, 19), (14, 3)]}: in this case we know $\rank(\Cad)=19$ by Lemma~\ref{lemma: odd number of components of rank}, and, given $\sizeGC=15$, we also know that $\sizeGCad=3$, $\sizeGCad=5$, or $\sizeGCad=15$. 

    The cases $\sizeGCad=3$ and $\sizeGCad=5$ can be discarded by Proposition~\ref{proposition: invertible implies equal rank} since there is only one component of rank $19$ in the universal grading of $\C$, which is less than $\frac{\sizeGC}{\sizeGCad}$. 

    The case $\sizeGCad=15$ can be discarded by Lemma~\ref{lemma: rank more than 1}, as there is only one component of rank at least $5$ (which divides $\sizeGCad=15$).

    \textbf{Case [(2, 11), (13, 3)]}: since $\rank(\Cad)=3$ by Lemma~\ref{lemma: odd number of components of rank}, we know $\sizeGCad=3$. But this means that $\FPdim(\Cad)=3$, which must be equal to the Frobenius-Perron dimension of the remaining components of the universal grading of $\C$. This is not possible for the components of rank $11$, which discards this case.

    \textbf{Case [(2, 13), (7, 5)]}: by Lemma~\ref{lemma: odd number of components of rank} we know that $\rank(\Cad)=5$. Since $\sizeGC=9$, we also know that $\sizeGCad=3$. 

    Consider the de-equivariantization $\deeqZ{3}$ of $\Cad$ by $\GCad\cong \Z_3$. The three invertible objects of $\Cad$ create one simple object in $\deeqZ{3}$. The two non-invertible simple objects in $\Cad$ must be fixed under the action of $\GCad$ on the simple objects of $\Cad$, and each of them creates three simple objects in $\deeqZ{3}$. As a result, $\rank(\deeqZ{3})=1+3\cdot 2=7$, and by~\cite[Theorem 4.5]{BR}, $\deeqZ{3}$ is pointed. This means that $\FPdim(\deeqZ{3})=7$ and $\FPdim(\Cad)=3\cdot 7=21$. This must be equal to the Frobenius-Perron dimension of any component of the universal grading of $\C$. By Proposition~\ref{proposition: invertible implies equal rank}, we know that the components with rank $13$ do not contain any invertible objects, and as such their Frobenius-Perron dimension must be at least $13\cdot 3^2=117$, which is larger than $21$, discarding this case. 
\end{proof}

\begin{proposition}
    Let $\C$ be an odd-dimensional MTC of rank $63$. Then $\C$ must be pointed or perfect.
\end{proposition}
\begin{proof}
    We can discard all values of $\sizeGC$ except for $63$, $55$, $47$, $39$, $31$, $23$, $21$, $15$, $13$, $11$, $9$, $7$, $5$, $3$, and $1$ using Lemma~\ref{lemma: rank linear combination}. Then, we discard most cases directly using previously shown lemmas, which are displayed in Appendix~\ref{section: case list}. The following non-pointed and non-perfect cases remain:
    \begin{itemize}
        \item {[(21, 3)]},
        \item {[(3, 17), (12, 1)]},
        \item {[(9, 7)]},
        \item {[(3, 21)]}.
    \end{itemize}

    \textbf{Case [(21, 3)]}: in this case we must have $\sizeGCad=3$ and $\FPdim(\Cad)=3$. This must be equal to the Frobenius-Perron dimension of the remaining components of the universal grading of $\C$, and since each of them has rank $3$, they must each contain $3$ invertible objects. This implies that $\C$ is pointed, a contradiction, which discards this case. 

    \textbf{Case [(3, 17), (12, 1)]}: by Lemma~\ref{thm:trivial}, $\sizeGCad=3, 5$ or 15. Note that $\sizeGCad=15$ can be discarded by Proposition~\ref{lemma: fixed simple}, and $\sizeGCad=3$ can be discarded by Proposition~\ref{proposition: invertible implies equal rank}.
    For $\sizeGCad=5$, we use Algorithm~\ref{alg: Cad dim generation} to compute the potential Frobenius-Perron dimensions of non-invertible simple objects in $\mathcal C_{\mathrm{ad}}$. After removing the cases that violate Remark~\ref{remark: fixed fpdim divisible by p}, this produces the cases listed in Table~\ref{table: rank 63 with 15 invertibles}, which exhibits these dimensions in columns 1 through 6, excluding duplicates due to duals.
    \begin{table}[h]
        \begin{center}
            \begin{tabular}{ c | c | c | c | c | c | c | c }
                \# & $\FPdim(\C)$ & 1 & 2 & 3 & 4 & 5 & 6 \\
                \hline
                1 & $3^2 \cdot 5^2 \cdot 47$ & 15 & 5 & 5 & 5 & 5 & 5 \\
                2 & $3^6 \cdot 5^2 \cdot 7 \cdot 11^2$ & 495 & 495 & 135 & 55 & 55 & 15 \\
                3 & $3^2 \cdot 5^2 \cdot 7 \cdot 11^2$ & 55 & 55 & 15 & 5 & 5 & 5 \\
                4 & $3^2 \cdot 5^2 \cdot 7^3 \cdot 13^2$ & 455 & 455 & 105 & 65 & 65 & 35 \\
                5 & $3^2 \cdot 5^2 \cdot 7 \cdot 17^2 \cdot 37^2$ & 3145 & 3145 & 555 & 555 & 555 & 255 \\
                6 & $3^2 \cdot 5^2 \cdot 7 \cdot 19^2$ & 95 & 95 & 15 & 15 & 15 & 15 \\
                7 & $3^2 \cdot 5^2 \cdot 127$ & 15 & 15 & 15 & 15 & 5 & 5 \\
                8 & $3 \cdot 5^2 \cdot 61$ & 5 & 5 & 5 & 5 & 5 & 5 \\
            \end{tabular}
              \caption{Potential dimension arrays of $\Cad$, for $\C$ an MTC  of rank $63$ with $15$ invertible objects, $\rank(\Cad)=17$, and $\sizeGCad=5$. All discarded.}
            \label{table: rank 63 with 15 invertibles}
        \end{center}
    \end{table}

    All the generated cases are easily discarded, since the existence of components of the universal grading with a unique simple object implies that $\FPdim(\Cad)=\FPdim(\C)/15$ should be a perfect square. 

    \textbf{Case [(9, 7)]}: we know that $\sizeGC=9$ and $\rank(\Cad)=7$.  

    Consider the de-equivariantization $\deeqZ{3}$ of $\Cad$ by $\GCad\cong \Z_3$. The three invertible objects of $\Cad$ create one simple object in $\deeqZ{3}$. The four non-invertible simple objects in $\Cad$ must be fixed under the action of $\GCad$ on the simple objects of $\Cad$, and each of them creates three simple objects in $\deeqZ{3}$. As a result, $\rank(\deeqZ{3})=1+3\cdot 4=13$, and by~\cite[Theorem 6.3 (a)]{CP}, $\deeqZ{3}$ is pointed. This means that $\FPdim(\deeqZ{3})=13$ and $\FPdim(\Cad)=3\cdot 13=39$. This must be equal to the Frobenius-Perron dimension of any component of the universal grading of $\C$. By Proposition~\ref{proposition: invertible implies equal rank}, we know that six of the components do not contain any invertible objects, and as such their Frobenius-Perron dimension must be at least $7\cdot 3^2=63$, which is larger than $39$, discarding this case. 

    \textbf{Case [(3, 21)]}: we use Algorithm~\ref{alg: Cad dim generation} to compute the potential Frobenius-Perron dimensions of non-invertible simple objects in $\mathcal C_{\mathrm{ad}}$. After removing the cases that violate Remark~\ref{remark: fixed fpdim divisible by p}, this produces the cases listed in Table~\ref{table: rank 63 with 3 invertibles}, which exhibits these dimensions in columns 1 through 9, excluding duplicates due to duals.
    \begin{table}[h]
        \begin{center}
            \begin{tabular}{ c | c | c | c | c | c | c | c | c | c | c }
                \# & $\FPdim(\C)$ & 1 & 2 & 3 & 4 & 5 & 6 & 7 & 8 & 9 \\
                \hline
                1 & $3^8 \cdot 5^2 \cdot 7$ & 405 & 135 & 81 & 45 & 15 & 15 & 5 & 5 & 5 \\
                2 & $3^4 \cdot 7 \cdot 11^2$ & 99 & 33 & 11 & 11 & 11 & 9 & 9 & 3 & 3 \\
                3 & $3^4 \cdot 7 \cdot 13^2$ & 117 & 39 & 13 & 13 & 13 & 9 & 9 & 9 & 3 \\
                4 & $3^4 \cdot 5^2 \cdot 7$ & 45 & 15 & 5 & 5 & 5 & 3 & 3 & 3 & 3 \\
                5 & $3^2 \cdot 5^2 \cdot 7 \cdot 13^2$ & 195 & 39 & 39 & 39 & 39 & 15 & 3 & 3 & 3 \\
                6 & $3^2 \cdot 5^4 \cdot 7 \cdot 17^2$ & 1275 & 255 & 255 & 255 & 255 & 75 & 51 & 51 & 3 \\
                7 & $3^2 \cdot 5^4 \cdot 7$ & 75 & 15 & 15 & 15 & 15 & 3 & 3 & 3 & 3 \\
                8 & $3^4 \cdot 5^2 \cdot 7$ & 45 & 9 & 9 & 9 & 5 & 5 & 5 & 3 & 3 \\
                9 & $3^2 \cdot 7^3$ & 21 & 3 & 3 & 3 & 3 & 3 & 3 & 3 & 3 \\
                10 & $3^3 \cdot 5^3$ & 15 & 15 & 5 & 5 & 5 & 3 & 3 & 3 & 3 \\
                11 & $3^2 \cdot 5^2 \cdot 31$ & 15 & 15 & 15 & 15 & 15 & 3 & 3 & 3 & 3 \\
                12 & $3^2 \cdot 5 \cdot 11$ & 3 & 3 & 3 & 3 & 3 & 3 & 3 & 3 & 3 \\
            \end{tabular}
           \caption{Potential dimension arrays of $\Cad$, for $\C$ an MTC of rank $63$ with $3$ invertible objects, $\rank(\Cad)=21$, and $\sizeGCad=3$. All discarded.}
            \label{table: rank 63 with 3 invertibles}
        \end{center}
    \end{table}    

In case 12, $\FPdim(\C)=3^2\cdot d,$ for $d$ square-free, and so it should be pointed, see~\cite[Theorem 4.7, Corollary 4.13]{DN}. In case 7, $\FPdim((\Cad)_{\mathbb Z_3})=5^4\cdot 7$, and thus by~\cite[Corollary 4.13]{DN} $(\Cad)_{\mathbb Z_3}$ should be pointed, which contradicts that there are simple objects in $\Cad$ of Frobenius-Perron dimension greater than $3$. Cases 9--11 are discarded in the same way. 

In cases 2--4 and 8,  $\rank((\Cad)_{\mathbb Z_3})$ is at most 39, $|\mathcal G((\Cad)_{\mathbb Z_3})|\geq 7$, and $|\mathcal G((\Cad)_{\mathbb Z_3})|\neq 9$. Thus by Theorem~\ref{thm: summary}, we should have that $(\Cad)_{\mathbb Z_3}$ is pointed, which contradicts that there are simple objects in $\Cad$ of Frobenius-Perron dimension $>3$.

In case 5, note that 3 does not divide $\FPdim((\Cad)_{\mathbb Z_3})$, so the simple objects of dimension 3 in $\Cad$ must be fixed by the action of $\mathbb Z_3$. But then $|\mathcal G((\Cad)_{\mathbb Z_3})|=19,$ which does not divide $\FPdim((\Cad)_{\mathbb Z_3})$. 

In case 6, $(\Cad)_{\mathbb Z_3}$ has 7 invertibles, and so $((\Cad)_{\mathbb Z_3})_{\mathrm{ad}}$  has Frobenius-Perron dimension $5^4\cdot 17^2$. Since 7 does not divide $5^4\cdot 17^2$, this implies that $(((\Cad)_{\mathbb Z_3})_{\mathrm{ad}})_{\mathrm{pt}}$ is trivial, which contradicts Lemma~\ref{thm:trivial}.

Lastly, in case 1, the simple objects in $(\Cad)_{\mathbb Z_3}$ have dimensions $135$, $45$, $27$, $15$, or $5$. Let $X_1$ be one of the simple objects of Frobenius-Perron dimension $5$. Then consider the decomposition 
    \begin{align*}
    	X_1\otimes X_1^{*}=1\oplus \bigoplus\limits_{i=1}^k N_{X_1,X_1^*}^{X_i} (X_i+X_{i^*}),
    \end{align*}
    where $X_1, \dots, X_k, X_1^*, \dots, X_k^*$ denote all non-invertible simple objects in $(\mathcal C_{\mathrm{ad}})_{\mathbb Z_3}$, and we are using that  $N_{X_1,X_1^*}^{X_i}=N_{X_1,X_1^*}^{X_i^*}$ for all $i=1,\dots, k$. Taking Frobenius-Perron dimension on both sides of the equation above, we get
    \begin{align*}
    	25=1+2\sum\limits_{i=1}^k N_{X_1,X_1^*}^{X_i} \FPdim(X_i).
    \end{align*}
    But then we must have
    \begin{align*}
    	25=1+2n_1\cdot 5,
    \end{align*}
for some $n_1\geq 0$, which is not possible. Thus all cases in Table~\ref{table: rank 63 with 3 invertibles} are discarded.
\end{proof}

\begin{proposition}\label{proposition: rank 65}
    Let $\C$ be an odd-dimensional MTC of rank $65$. Then $\C$ must be pointed, perfect, or have $3$ or $5$ invertible objects. Additionally, if $\sizeGC=3$, then the ranks of the components of the universal grading of $\C$ must be $[(1, 59), (2, 3)]$, and if $\sizeGC=5$, then those ranks must be $[(1, 45), (4, 5)]$. 
\end{proposition}
\begin{proof}
    We can discard all values of $\sizeGC$ except for $65$, $57$, $49$, $41$, $33$, $25$, $19$, $17$, $13$, $11$, $9$, $7$, $5$, $3$, and $1$ using Lemma~\ref{lemma: rank linear combination}. Then, we discard most cases directly using previously shown lemmas, which are displayed in Appendix~\ref{section: case list}. The following non-pointed and non-perfect cases remain:
    \begin{itemize}
       
        \item {[(1, 23), (6, 7)]},
        \item {[(2, 27), (1, 11)]},
        \item {[(1, 45), (4, 5)]},
        \item {[(1, 59), (2, 3)]},
    \end{itemize}

    \textbf{Case [(1, 23), (6, 7)]}: we use Algorithm~\ref{alg: Cad dim generation} to compute the potential Frobenius-Perron dimensions of non-invertible simple objects in $\mathcal C_{\mathrm{ad}}$. After removing the cases that violate Remark~\ref{remark: fixed fpdim divisible by p}, this produces the cases listed in Table~\ref{table: rank 65 with 7 invertibles}, which exhibits these dimensions in columns 1 through 8, excluding duplicates due to duals.
    \begin{table}[h]
        \begin{center}
            \begin{tabular}{ c | c | c | c | c | c | c | c | c | c }
                \# & $\FPdim(\C)$ & 1 & 2 & 3 & 4 & 5 & 6 & 7 & 8 \\
                \hline
                1 & $3^4 \cdot 5^4 \cdot 7^2$ & 315 & 175 & 175 & 105 & 63 & 35 & 21 & 7 \\
                2 & $3^6 \cdot 5^2 \cdot 7^2 \cdot 17^2$ & 3213 & 1785 & 1785 & 1071 & 595 & 357 & 315 & 105 \\
                3 & $3^2 \cdot 5^2 \cdot 7^2$ & 21 & 7 & 7 & 7 & 7 & 7 & 7 & 7 \\
                4 & $7^2 \cdot 113$ & 7 & 7 & 7 & 7 & 7 & 7 & 7 & 7 \\
            \end{tabular}
            \caption{Potential dimension arrays of $\Cad$, for $\C$ an MTC of rank $65$ with $7$ invertible objects, $\rank(\Cad)=23$, and $\sizeGCad=7$. All discarded.}
            \label{table: rank 65 with 7 invertibles}
        \end{center}
    \end{table}    

   Case 4 is discarded since $\C$ should be pointed by~\cite[Proposition 4.11]{NR}. In case 3, note that 7 does not divide $\FPdim((\Cad)_{\mathbb Z_7})$ and so all simple objects in $\Cad$ of Frobenius-Perron dimension 7 must be fixed by the action of $\mathbb Z_7$. Hence  $(\Cad)_{\mathbb Z_7}$ has rank 113 and 99 invertibles. This implies that at least one component of its universal grading will have a unique (invertible) simple, and thus all components would have
Frobenius-Perron dimension 1, which is a contradiction since $(\Cad)_{\mathbb Z_7}$ is not pointed. 

In case 1, $(\Cad)_{\mathbb Z_7}$ has rank 113 and 15 invertibles, so the ranks of the components of its universal grading should be  [(1, 15), (14, 7)]. By Corollary~\ref{corollary: three components rank Cad}, $|\mathcal G(((\Cad)_{\mathbb Z_7})_{\mathrm{ad}})|=15$. But then $\FPdim((\Cad)_{\mathbb Z_7})_{\mathrm{ad}}))=15,$ which contradicts that the remaining components have 7 simples, since the Frobenius-Perron dimension of a non-invertible simple is at least 3.

It remains to discard case 2. Simple objects in $(\Cad)_{\mathbb Z_7}$ have dimensions $15$, $45$, $51$, $85$, $153$, $255$, or $459$. Let $X_1$ be one of the simple objects of Frobenius-Perron dimension $15$. Then consider the decomposition 
    \begin{align*}
    	X_1\otimes X_1^{*}=1\oplus \bigoplus\limits_{i=1}^k N_{X_1,X_1^*}^{X_i} (X_i+X_{i^*}),
    \end{align*}
    where $X_1, \dots, X_k, X_1^*, \dots, X_k^*$ denote all non-invertible simple objects in $(\mathcal C_{\mathrm{ad}})_{\mathbb Z_7}$, and we are using that  $N_{X_1,X_1^*}^{X_i}=N_{X_1,X_1^*}^{X_i^*}$ for all $i=1,\dots, k$. Taking Frobenius-Perron dimension on both sides of the equation above, we get
    \begin{align*}
    	225=1+2\sum\limits_{i=1}^k N_{X_1,X_1^*}^{X_i} \FPdim(X_i).
    \end{align*}
    But then, noting that $153=3\cdot 51$, we must have
    \begin{align*}
    	225=1+2n_1\cdot 15+ 2n_2\cdot 45+ 2n_3\cdot 51+ 2n_4\cdot 85,
    \end{align*}
for some $n_1, \dots, n_4\geq 0$. Taking equivalence mod 5, we obtain that $n_3=2$. But then 
\begin{align*}
    	20=2n_1\cdot 15+ 2n_2\cdot 45 +2n_4\cdot 85,
    \end{align*}
which is not possible.
Thus all cases in Table~\ref{table: rank 65 with 7 invertibles} are discarded.

    \textbf{Case [(2, 27), (1, 11)]}: we use Algorithm~\ref{alg: Cad dim generation} to compute the potential Frobenius-Perron dimensions of non-invertible simple objects in $\mathcal C_{\mathrm{ad}}$. After removing the cases that violate Remark~\ref{remark: fixed fpdim divisible by p}, this produces the unique case listed in Table~\ref{table: rank 65 with 3 invertibles}, which exhibits these dimensions in columns 1 through 4, excluding duplicates due to duals. 
        \begin{table}[h]
        \begin{center}
            \begin{tabular}{ c | c | c | c | c | c }
                \# & $\FPdim(\C)$ & 1 & 2 & 3 & 4 \\
                \hline
                1 & $3^2 \cdot 5^2$ & 3 & 3 & 3 & 3 \\
            \end{tabular}
            \caption{Potential dimension arrays of $\Cad$, for $\C$ an MTC  of rank $65$ with $3$ invertible objects, $\rank(\Cad)=11$, and $\sizeGCad=3$. Discarded.}
            \label{table: rank 65 with 3 invertibles}
        \end{center}
    \end{table}

Suppose that $\mathcal C$ is a modular category with the dimension array of $\Cad$ of Table~\ref{table: rank 65 with 3 invertibles}. Then each component of the universal grading of $\C$ that has $27$ simple objects does not have any invertible objects, so their Frobenius-Perron dimensions are at least $3^2\cdot 27=243$. But $\FPdim(\C)=225$, a contradiction, so $\C$ cannot exist.
\end{proof}

\begin{remark}
    It is known that the category $\Rep(D^\omega(\Z_{13}\rtimes\Z_3))$ is an example of an odd-dimensional MTC of rank $65$ (see~\cite[Section 4.2]{DKP}) with $3$ invertible objects. 
\end{remark}

\begin{proposition}
    Let $\C$ be an odd-dimensional MTC of rank $67$. Then $\C$ must be pointed, perfect, or have $3$ or $9$ invertible objects. Additionally, if $\sizeGC=3$, then the ranks of the components of the universal grading of $\C$ must be $[(1, 49), (2, 9)]$, and if $\sizeGC=9$, then those ranks must be $[(1, 43), (8, 3)]$. 
\end{proposition}
\begin{proof}
    We can discard all values of $\sizeGC$ except for $67$, $59$, $51$, $43$, $35$, $27$, $19$, $17$, $11$, $9$, $7$, $5$, $3$, and $1$ using Lemma~\ref{lemma: rank linear combination}. Then, we discard most cases directly using previously shown lemmas, which are displayed in Appendix~\ref{section: case list}. The following non-pointed and non-perfect cases remain:
    \begin{itemize}
        
        \item {[(5, 11), (4, 3)]},
        \item {[(4, 15), (1, 7)]},
        \item {[(1, 43), (8, 3)]},
        \item {[(1, 49), (2, 9)]}.
    \end{itemize}

    \textbf{Case [(5, 11), (4, 3)]}: by Lemma~\ref{thm:trivial}, $\sizeGCad=3$ or 9. Note that $\sizeGCad=9$ can be discarded by~\cite[Proposition 5.6]{CP}. For $\sizeGCad=3$, we use Algorithm~\ref{alg: Cad dim generation} to compute the potential Frobenius-Perron dimensions of non-invertible simple objects in $\mathcal C_{\mathrm{ad}}$. After removing the cases that violate Remark~\ref{remark: fixed fpdim divisible by p}, this produces the cases listed in Table~\ref{table: rank 67 with 9 invertibles}, which exhibits these dimensions in columns 1 through 6, excluding duplicates due to duals.
    \begin{table}[h]
        \begin{center}
            \begin{tabular}{ c | c | c | c | c | c }
                \# & $\FPdim(\C)$ & 1 & 2 & 3 & 4 \\
                \hline
                1 & $3^3 \cdot 5^2$ & 3 & 3 & 3 & 3 \\
            \end{tabular}
           \caption{Potential dimension arrays of $\Cad$, for $\C$ an MTC of rank $67$ with $9$ invertible objects, $\rank(\Cad)=11$, and $\sizeGCad=3$. Discarded.}
            \label{table: rank 67 with 9 invertibles}
        \end{center}
    \end{table}    
    To discard this case, note that the Frobenius-Perron dimension of each component of the universal grading is $3\cdot 5^2=75$. By Proposition~\ref{proposition: invertible implies equal rank},  two of the components of rank 11 do not have any invertibles, and so their Frobenius-Perron dimension is at least $11 \cdot 3^2 = 99$, a contradiction.

    \textbf{Case [(4, 15), (1, 7)]}: we know that $\rank(\Cad)=7$ by Lemma~\ref{lemma: odd number of components of rank} and $\sizeGCad=5$. Then the de-equivariantization of $\Cad$ by $\GCad\cong\Z_5$ must have rank $1+2\cdot 5=11$, so it is pointed by~\cite[Theorem 4.5]{BR} and has Frobenius-Perron dimension $11$. As a result, the dimensions of all components of the universal grading of $\C$ must be $\FPdim(\Cad)=55$. But the dimension of one of the components of rank $15$ must be at least $15\cdot 3^2>55$, a contradiction.
\end{proof}

\begin{proposition}
    Let $\C$ be an odd-dimensional MTC of rank $69$. Then $\C$ must be pointed, perfect, or have $3$ invertible objects. Additionally, if $\sizeGC=3$, then the ranks of the components of the universal grading of $\C$ must be $[(1, 39), (2, 15)]$. 
\end{proposition}
\begin{proof}
    We can discard all values of $\sizeGC$ except for $69$, $61$, $53$, $45$, $37$, $29$, $23$, $21$, $15$, $13$, $9$, $7$, $5$, $3$, and $1$ using Lemma~\ref{lemma: rank linear combination}. Then, we discard most cases directly using previously shown lemmas, which are displayed in Appendix~\ref{section: case list}. The following non-pointed and non-perfect cases remain:
    \begin{itemize}
        \item {[(3, 17), (18, 1)]},
        \item {[(1, 27), (14, 3)]},
        \item {[(3, 11), (12, 3)]},
        \item {[(3, 13), (6, 5)]},
        \item {[(1, 39), (2, 15)]}.
    \end{itemize}

    \textbf{Case [(3, 17), (18, 1)]}: by Lemma~\ref{thm:trivial}, $\sizeGCad=3$ or 7. Note that $\sizeGCad=3$ can be discarded by Proposition~\ref{proposition: invertible implies equal rank}. 
    For $\sizeGCad=7$, we use Algorithm~\ref{alg: Cad dim generation} to compute the potential Frobenius-Perron dimensions of non-invertible simple objects in $\mathcal C_{\mathrm{ad}}$. After removing the cases that violate Remark~\ref{remark: fixed fpdim divisible by p}, this produces the cases listed in Table~\ref{table: rank 69 with 21 invertibles}, which exhibits these dimensions in columns 1 through 5, excluding duplicates due to duals.
    \begin{table}[h]
        \begin{center}
            \begin{tabular}{ c | c | c | c | c | c | c }
                \# & $\FPdim(\C)$ & 1 & 2 & 3 & 4 & 5 \\
                \hline
                1 & $3^2 \cdot 5 \cdot 7^2 \cdot 13^2$ & 91 & 21 & 7 & 7 & 7 \\
                2 & $3^2 \cdot 5^3 \cdot 7^2 \cdot 17^2$ & 595 & 119 & 105 & 7 & 7 \\
                3 & $3^2 \cdot 7^2 \cdot 61$ & 21 & 7 & 7 & 7 & 7 \\
                4 & $3^2 \cdot 7^2 \cdot 173$ & 21 & 21 & 21 & 21 & 7 \\
                5 & $3 \cdot 7^2 \cdot 71$ & 7 & 7 & 7 & 7 & 7 \\
            \end{tabular}
             \caption{Potential dimension arrays of $\Cad$, for $\C$ an MTC  of rank $69$ with $21$ invertible objects, $\rank(\Cad)=17$, and $\sizeGCad=7$. All discarded.}
            \label{table: rank 69 with 21 invertibles}
        \end{center}
    \end{table}    
In this case, there are components of the universal grading with a unique simple object. This implies that $\FPdim(\Cad)=\FPdim(\C)/21$ should be a perfect square, which discards all potential arrays in Table~\ref{table: rank 69 with 21 invertibles}.

    \textbf{Case [(1, 27), (14, 3)]}: in this case we know $\rank(\Cad)=27$ by Lemma~\ref{lemma: odd number of components of rank}, and, given $\sizeGC=15$, we also know that $\sizeGCad=3$, $\sizeGCad=5$, or $\sizeGCad=15$. 

    The cases $\sizeGCad=3$ and $\sizeGCad=5$ can be discarded by Proposition~\ref{proposition: invertible implies equal rank} since there is only one component of rank $27$ in the universal grading of $\C$, which is less than $\frac{\sizeGC}{\sizeGCad}$. 

    The case $\sizeGCad=15$ can be discarded by Lemma~\ref{lemma: rank more than 1}, since there is only one component of rank at least $5$ (which is a divisor of $\sizeGCad=15$). 

    \textbf{Case [(3, 11), (12, 3)]}: in this case we know that $\rank(\Cad)=11$ by Lemma~\ref{lemma: odd number of components of rank}, and $\sizeGCad=3$ or $\sizeGCad=5$. 

    The case $\sizeGCad=3$ can be discarded by Proposition~\ref{proposition: invertible implies equal rank} since there are only three components of rank $11$ in the universal grading of $\C$, which is less than $\frac{\sizeGC}{\sizeGCad}=5$. 

    For the case $\sizeGCad=5$, we use Algorithm~\ref{alg: Cad dim generation} to compute the potential Frobenius-Perron dimensions of non-invertible simple objects in $\mathcal C_{\mathrm{ad}}$. After removing the cases that violate Remark~\ref{remark: fixed fpdim divisible by p}, this produces the cases listed in Table~\ref{table: rank 69 with 15 invertibles}, which exhibits these dimensions in columns 1 through 3, excluding duplicates due to duals.
    \begin{table}[h]
        \begin{center}
            \begin{tabular}{ c | c | c | c | c }
                \# & $\FPdim(\C)$ & 1 & 2 & 3 \\
                \hline
                1 & $3^2\cdot 5^2\cdot 37$ & 15 & 5 & 5 \\
                2 & $3\cdot 5^2\cdot 31$ & 5 & 5 & 5
            \end{tabular}
             \caption{Potential dimension arrays of $\Cad$, for $\C$ an MTC of rank $69$ with $15$ invertible objects, $\rank(\Cad)=11$, and $\sizeGCad=5$. All discarded.}
            \label{table: rank 69 with 15 invertibles}
        \end{center}
    \end{table}   

    In case $1$, the modular category $(\Cad)_{\mathbb Z_5}$ has rank $31$ but is neither pointed nor perfect, a contradiction by Theorem~\ref{thm: rank 27 to 31}. In case $2$, the Frobenius-Perron dimension of $\C$ is of the form $p^2d$ with $d$ squarefree, so $\C$ should be pointed by a generalization of~\cite[Corollary 4.13]{DN}, a contradiction.

    \textbf{Case [(3, 13), (6, 5)]}: by Lemma~\ref{thm:trivial}, $\sizeGCad=3$ or 9. Note that $\sizeGCad=9$ can be discarded by~\cite[Proposition 5.6]{CP}.
    For $\sizeGCad=3$, we use Algorithm~\ref{alg: Cad dim generation} to compute the potential Frobenius-Perron dimensions of non-invertible simple objects in $\mathcal C_{\mathrm{ad}}$. After removing the cases that violate Remark~\ref{remark: fixed fpdim divisible by p}, this produces the cases listed in Table~\ref{table: rank 69 with 9 invertibles}, which exhibits these dimensions in columns 1 through 5, excluding duplicates due to duals.
    \begin{table}[h]
        \begin{center}
            \begin{tabular}{ c | c | c | c | c | c | c }
                \# & $\FPdim(\C)$ & 1 & 2 & 3 & 4 & 5 \\
                \hline
                1 & $3^3 \cdot 5^2 \cdot 7$ & 15 & 3 & 3 & 3 & 3 \\
                2 & $3^3 \cdot 31$ & 3 & 3 & 3 & 3 & 3 \\
            \end{tabular}
             \caption{Potential dimension arrays of $\Cad$, for $\C$ an MTC  of rank $69$ with $9$ invertible objects, $\rank(\Cad)=13$, and $\sizeGCad=3$. All discarded.}
            \label{table: rank 69 with 9 invertibles}
        \end{center}
    \end{table}

    Case 2 is not possible since $\mathcal C$ should be pointed, see~\cite[Proposition 4.12]{NR}. In case 1, the de-equivariantization  $(\Cad)_{\mathbb Z_3}$ has Frobenius-Perron dimension $5^2\cdot 7$, so it should be pointed by~\cite[Proposition 4.11]{NR}. This contradicts the fact that there are simple objects of Frobenius-Perron dimension $15$ in $\Cad$. 
\end{proof}

\begin{proposition}
    Let $\C$ be an odd-dimensional MTC of rank $71$. Then $\C$ must be pointed, perfect, or have $3$ invertible objects. Additionally, if $\sizeGC=3$, then the ranks of the components of the universal grading of $\C$ must be $[(1, 29), (2, 21)]$. 
\end{proposition}
\begin{proof}
    We can discard all values of $\sizeGC$ except for $71$, $63$, $55$, $47$, $39$, $31$, $23$, $21$, $15$, $13$, $11$, $9$, $7$, $5$, $3$, and $1$ using Lemma~\ref{lemma: rank linear combination}. Then, we discard most cases directly using previously shown lemmas, which are displayed in Appendix~\ref{section: case list}. The following non-pointed and non-perfect case remains:
    \begin{itemize}
       
        \item {[(1, 29), (2, 21)]}.
    \end{itemize}

\end{proof}

\begin{proposition}
    Let $\C$ be an odd-dimensional MTC of rank $73$. Then $\C$ must be pointed, perfect, or have $3$, $5$, or $7$ invertible objects. Additionally, if $\sizeGC=3$, then the ranks of the components of the universal grading of $\C$ must be $[(1, 67), (2, 3)]$, if $\sizeGC=5$, then those ranks must be $[(1, 53), (4, 5)]$, and if $\sizeGC=7$, then those ranks must be $[(1, 31), (6, 7)]$. 
\end{proposition}
\begin{proof}
    We can discard all values of $\sizeGC$ except for $73$, $65$, $57$, $49$, $41$, $33$, $25$, $19$, $17$, $13$, $11$, $9$, $7$, $5$, $3$, and $1$ using Lemma~\ref{lemma: rank linear combination}. Then, we discard most cases directly using previously shown lemmas, which are displayed in Appendix~\ref{section: case list}. The following non-pointed and non-perfect cases remain:
    \begin{itemize}
        \item {[(3, 17), (2, 9), (4, 1)]},
        \item {[(2, 27), (1, 19)]},
        \item {[(1, 31), (6, 7)]},
        \item {[(1, 53), (4, 5)]},
        \item {[(1, 67), (2, 3)]}.
    \end{itemize}

    \textbf{Case [(3, 17), (2, 9), (4, 1)]}: by Lemma~\ref{thm:trivial}, $\sizeGCad=3$ or 9. Note that $\sizeGCad=9$ can be discarded by~\cite[Proposition 5.6]{CP}.
    For $\sizeGCad=3$, we use Algorithm~\ref{alg: Cad dim generation} to compute the potential Frobenius-Perron dimensions of non-invertible simple objects in $\mathcal C_{\mathrm{ad}}$. After removing the cases that violate Remark~\ref{remark: fixed fpdim divisible by p}, this produces the cases listed in Table~\ref{table: rank 73 with 9 invertibles}, which exhibits these dimensions in columns 1 through 7, excluding duplicates due to duals.
    \begin{table}[h]
        \begin{center}
            \begin{tabular}{ c | c | c | c | c | c | c | c | c }
                \# & $\FPdim(\C)$ & 1 & 2 & 3 & 4 & 5 & 6 & 7 \\
                \hline
                1 & $3^8 \cdot 5^2$ & 81 & 45 & 15 & 15 & 5 & 5 & 5 \\
                2 & $3^3 \cdot 5^2 \cdot 19$ & 15 & 15 & 15 & 3 & 3 & 3 & 3 \\
                3 & $3^4 \cdot 11^2$ & 11 & 11 & 11 & 9 & 9 & 3 & 3 \\
                4 & $3^4 \cdot 13^2$ & 13 & 13 & 13 & 9 & 9 & 9 & 3 \\
                5 & $3^4 \cdot 5^2$ & 5 & 5 & 5 & 3 & 3 & 3 & 3 \\
                6 & $3^3 \cdot 43$ & 3 & 3 & 3 & 3 & 3 & 3 & 3 \\
            \end{tabular}
            \caption{Potential dimension arrays of $\Cad$, for $\C$ an MTC  of rank $73$ with $9$ invertible objects, $\rank(\Cad)=17$, and $\sizeGCad=3$. All discarded.}
            \label{table: rank 73 with 9 invertibles}
        \end{center}
    \end{table}

    In cases 2--5, the modular category $(\Cad)_{\mathbb Z_3}$ obtained by de-equivariantization has Frobenius-Perron dimension $\FPdim(\C)/3^3=p^2\cdot q,$ for $p$ and $q$ odd primes. Hence by~\cite[Proposition 4.11]{NR} $(\Cad)_{\mathbb Z_3}$ should be pointed, which contradicts that $\Cad$ has simple objects of dimension greater than 3. Case 6 is discarded by~\cite[Proposition 4.12]{NR}, since $\C$ should be pointed. It remains to discard case 1. Note that $(\Cad)_{\mathbb Z_3}$ is solvable by~\cite[Theorem 1.6]{ENO2}. It follows from~\cite[Proposition 4.5]{ENO2} that $(\Cad)_{\mathbb Z_3}$ has a non-trivial invertible object. On the other hand, $\Cad$ has no simple objects of dimension 3, thus $(\Cad)_{\mathbb Z_3}$ is perfect, so we get a contradiction.

    \textbf{Case [(2, 27), (1, 19)]}: by Lemma~\ref{thm:trivial}, $\sizeGCad=3$. We use Algorithm~\ref{alg: Cad dim generation} to compute the potential Frobenius-Perron dimensions of non-invertible simple objects in $\mathcal C_{\mathrm{ad}}$. After removing the cases that violate Remark~\ref{remark: fixed fpdim divisible by p}, this produces the cases listed in Table~\ref{table: rank 73 with 3 invertibles}, which exhibits these dimensions in columns 1 through 8, excluding duplicates due to duals.
    \begin{table}[H]
        \begin{center}
            \begin{tabular}{ c | c | c | c | c | c | c | c | c | c }
                \# & $\FPdim(\C)$ & 1 & 2 & 3 & 4 & 5 & 6 & 7 & 8 \\
                \hline
                1 & $3^8 \cdot 5^2$ & 135 & 81 & 45 & 15 & 15 & 5 & 5 & 5 \\
                2 & $3^4 \cdot 11^2$ & 33 & 11 & 11 & 11 & 9 & 9 & 3 & 3 \\
                3 & $3^4 \cdot 13^2$ & 39 & 13 & 13 & 13 & 9 & 9 & 9 & 3 \\
                4 & $3^4 \cdot 5^2$ & 15 & 5 & 5 & 5 & 3 & 3 & 3 & 3 \\
                5 & $3^2 \cdot 5^2 \cdot 13^2$ & 39 & 39 & 39 & 39 & 15 & 3 & 3 & 3 \\
                6 & $3^2 \cdot 5^4 \cdot 17^2$ & 255 & 255 & 255 & 255 & 75 & 51 & 51 & 3 \\
                7 & $3^2 \cdot 5^4$ & 15 & 15 & 15 & 15 & 3 & 3 & 3 & 3 \\
                8 & $3^4 \cdot 5^2$ & 9 & 9 & 9 & 5 & 5 & 5 & 3 & 3 \\
                9 & $3^2 \cdot 7^2$ & 3 & 3 & 3 & 3 & 3 & 3 & 3 & 3 \\
            \end{tabular}
            \caption{Potential dimension arrays of $\Cad$, for $\C$ an MTC  of rank $73$ with $3$ invertible objects, $\rank(\Cad)=19$, and $\sizeGCad=3$. All discarded.}
            \label{table: rank 73 with 3 invertibles}
        \end{center}
    \end{table}

    For case 1, consider the modular category $(\Cad)_{\mathbb Z_3}$, which has non-invertible simple objects of dimensions $5$, $15$, $27$, and $45$, and let $X_1$ be a simple object of dimension $5$. Then, as $(\Cad)_{\mathbb Z_3}$ has one invertible object, consider the decomposition 
    \begin{align*}
    	X_1\otimes X_1^{*}=1\oplus \bigoplus\limits_{i=1}^k N_{X_1,X_1^*}^{X_i} (X_i+X_{i^*}),
    \end{align*}
    where $X_1, \dots, X_k, X_1^*, \dots, X_k^*$ denote all non-invertible simple objects in $(\mathcal C_{\mathrm{ad}})_{\mathbb Z_3}$, and we are using that  $N_{X_1,X_1^*}^{X_i}=N_{X_1,X_1^*}^{X_i^*}$ for all $i=1,\dots, k$. Taking Frobenius-Perron dimension on both sides of the equation above, we get
    \begin{align*}
    	25=1+2\sum\limits_{i=1}^k N_{X_1,X_1^*}^{X_i} \FPdim(X_i).
    \end{align*}
    But then we must have
    \begin{align*}
    	25=1+2n_1\cdot 5+ 2n_2\cdot 15+ 2n_3\cdot 27+ 2n_4\cdot 45,
    \end{align*}
    for some $n_1, \dots, n_4\geq 0$. Clearly, $n_3=n_4=0$. But then taking equivalence mod 5 gives a contradiction, discarding this case.

    Now, for each of cases $2$, $3$, $4$, $6$, $7$, and $8$, we must have that $(\Cad)_{\mathbb Z_3}$ has $7$, $13$, or $25$ invertible objects. But we also know that the rank of the modular category $(\Cad)_{\mathbb Z_3}$ must be $17$, $33$, or $49$, which are impossible by Theorem~\ref{thm: pointed 17 to 23}, Proposition~\ref{prop: rank 33}, and Proposition~\ref{proposition: rank 49}, respectively, discarding these cases.

    Next, consider case $5$. We know that the Frobenius-Perron dimension of the deq-equivariantization $(\Cad)_{\mathbb Z_3}$ is $5^2\cdot 13^2$, which is solvable by~\cite[Theorem 1.6]{ENO2} and hence not perfect by~\cite[Proposition 4.5]{ENO2}. As a result, the $3$-dimensional simple objects in $\Cad$ must be fixed by the action of $\GCad$ on $\Cad$, so $(\Cad)_{\mathbb Z_3}$ has $19$ invertible objects. But $19$ does not divide its dimension $5^2\cdot 13^2$, discarding this case.

    The final remaining case is case $9$. We know that the Frobenius-Perron dimension of each component of the universal grading of $\C$ must be $3\cdot 7^2=147$. But the components with $27$ simple objects cannot have any invertible objects, so their dimensions are each at least $27\cdot 3^2=243$, a contradiction.

    Hence, all cases in Table~\ref{table: rank 73 with 3 invertibles} are discarded.
\end{proof}

We summarize the results from this section in the following theorem. 
\begin{theorem}
    Let $\C$ be an odd-dimensional MTC such that $51\leq\rank(\C)\leq 73$.
    \begin{enumerate}[label=(\alph*)]
        \item If $\rank(\C)\in \{53, 55, 63\}$, then $\C$ is either pointed or perfect.
        \item If $\rank(\C)\in \{51, 59, 67\}$, then $\C$ is either pointed, perfect or has 3 or 9 invertible objects. 
        \item If $\rank(\C)\in \{57, 65\}$, then $\C$ is either pointed, perfect or has 3 or 5 invertible objects. 
        \item If $\rank(\C)\in \{61, 69, 71\}$, then $\C$ is either pointed, perfect or has 3 invertible objects. 
        \item If $\rank(\C)=73$, then $\C$ is either pointed, perfect or has 3, 5 or 7 invertible objects. 
    \end{enumerate}
\end{theorem}

\clearpage
\appendix
\begin{appendices}
\section{Code for basic and adjoint algorithms}\label{section: algorithm code}

The following is the code for Algorithms~\ref{alg: dim generation} and~\ref{alg: Cad dim generation}, written in the Julia programming language. It utilizes the built-in rational data type in order to avoid arithmetic precision errors. Algorithm~\ref{alg: dim generation} is implemented in the function \texttt{basic\_algorithm} and Algorithm~\ref{alg: Cad dim generation} is implemented in the function \texttt{adjoint\_algorithm}.

\begin{verbatim}
using Primes

function validate_result(dk_squared)
    global GC, s, k, w, u_squared, solutions
    FPdimC_over_w = dk_squared * u_squared[k]
    d_squared = Vector{Rational{Int128}}(undef, k)
    d = Vector{Int128}(undef, k)
    ds_are_integers = true
    power_of_prime = false
    for i in eachindex(u_squared)
        d_squared[i] = FPdimC_over_w // u_squared[i]
        d[i] = round(Int128, sqrt(d_squared[i]))
        if !isinteger(d_squared[i]) || d[i] * d[i] != d_squared[i]
            ds_are_integers = false
            break
        end
        if length(factor(d[i])) == 1
            power_of_prime = true
        end
    end
    # for perfect MTCs, exclude cases with a simple object
    # whose Frobenius-Perron dimension is a power of a prime
    if ds_are_integers && (GC > 1 || !power_of_prime)
        FPdimC = FPdimC_over_w * w
        solution = Vector{Int128}(undef, k + 2)
        solution[1] = FPdimC
        solution[2] = s
        for i in eachindex(d)
            solution[i + 2] = d[i]
        end
        push!(solutions, solution)
    end
end

function solve(i, last_c, last_u)
    global s, k, t, u_squared 
    last_u_squared = last_u * last_u
    max_new_u_squared = 
        s * last_u_squared // (t * last_c) + 
        (2 * (k - i) * last_u_squared) // last_c
    min_new_u = round(Int128, sqrt(2 * last_u_squared // last_c))
    # defensive coding
    if min_new_u >= 1 && 
        (min_new_u - 1) * (min_new_u - 1) > 
        (2 * last_u_squared // last_c)
        throw(ErrorException("min_new_u rounded incorrectly, 
            which can lead to missing solutions."))
    end
    new_u = max(min_new_u, last_u)
    if new_u % 2 == 0
        new_u += 1
    end
    while new_u * new_u <= max_new_u_squared
        new_u_squared = new_u * new_u
        new_c = last_c * new_u_squared // last_u_squared - 2
        u_squared[i + 1] = new_u_squared
        if k == i + 1
            if new_c > 0 && isinteger(s // new_c)
                dk_squared = s // new_c
                dk = round(Int128, sqrt(dk_squared))
                if dk * dk == dk_squared
                    validate_result(Int128(dk_squared))
                end
            end
        elseif new_c > 0
            solve(i + 1, new_c, new_u)
        end
        new_u += 2
    end
end

function basic_algorithm(arg_rank, arg_GC, arg_min_m1)
    n = Int128(arg_rank)
    global GC = Int128(arg_GC)
    global s = GC
    global k = (n - s) ÷ 2
    global t
    global w
    global u_squared = Vector{Int128}(undef, k)
    global solutions = Vector{Vector{Int128}}()
    min_m1 = Int128(arg_min_m1)
    if s == 1
        t = 225
    else
        t = 9
    end
    m1 = n % 8
    while m1 <= 2 * k + s ÷ 9
        w = 1
        if m1 >= s && m1 >= min_m1
            c1 = m1 - 2
            if (c1 > 0)
                u1_squared = m1
                factorization = factor(m1)
                for i in factorization
                    if i.second % 2 == 1
                        u1_squared ÷= i.first
                        w *= i.first
                    end
                end
                u_squared[1] = u1_squared
                u1 = floor(Int128, sqrt(u1_squared))
                i = 1
                solve(i, c1, u1)
            end
        end
        m1 += 8
    end
    return solutions
end

function adjoint_algorithm(arg_rank, arg_n, arg_GC, arg_s, arg_min_m1)
    rank = Int128(arg_rank)
    n = Int128(arg_n)
    global GC = Int128(arg_GC)
    global s = Int128(arg_s)
    global k = (n - s) ÷ 2
    global t
    global w
    global u_squared = Vector{Int128}(undef, k)
    global solutions = Vector{Vector{Int128}}()
    min_m1 = Int128(arg_min_m1)
    if GC == 1
        t = 225
    else
        t = 9
    end
    m1 = rank % 8
    while m1 <= 2 * k * GC + (s * GC) ÷ 9
        w = 1
        if m1 >= GC && m1 >= min_m1
            c1 = m1 // GC - 2
            if (c1 > 0)
                u1_squared = m1
                factorization = factor(m1)
                for i in factorization
                    if i.second % 2 == 1
                        u1_squared ÷= i.first
                        w *= i.first
                    end
                end
                u_squared[1] = u1_squared
                u1 = floor(Int128, sqrt(u1_squared))
                i = 1
                if k > 1
                    solve(i, c1, u1)
                elseif isinteger(s // c1)
                    dk_squared = s // c1
                    dk = round(Int128, sqrt(dk_squared))
                    if dk * dk == dk_squared
                        validate_result(Int128(dk_squared))
                    end
                end
            end
        end
        m1 += 8
    end
    return solutions
end
\end{verbatim}

\clearpage
\section{All non-pointed and non-perfect cases for ranks 33--73}\label{section: case list}

Let $\C$ be an odd-dimensional MTC. For each value of $\rank(\C)$ between $33$ and $73$, we present all possible multisets of ranks of the universal grading components of $\C$ that satisfy Lemma~\ref{lemma: rank linear combination}.

In each rank, we split the cases based on whether they are handled or not, and, if they are handled, the method we use to handle them. These methods are:
\begin{enumerate}[label=(\alph*)]
    \item Discarded by Lemma~\ref{lemma: rank more than 1},
    \item Discarded by Lemma~\ref{lemma: odd number of components of rank},
    \item Discarded by Lemma~\ref{lemma: rank Cad 9} in conjunction with Lemma~\ref{lemma: odd number of components of rank},
    \item Discarded by~\cite[Proposition 5.6]{CP}, Proposition~\ref{lemma: fixed simple}, and Proposition~\ref{proposition: invertible implies equal rank}, all of which are in conjunction with Lemma~\ref{lemma: odd number of components of rank},
    \item Specific methods listed in the respective proposition for each rank, found in Sections~\ref{sec: ranks 33-49} and~\ref{sec: ranks 51-73}.
\end{enumerate}

The cases that remain after using these methods are listed in the last row of each table, which is labelled ``Open.''

\begin{longtable}[h]{ p{0.06\textwidth} | >{\raggedright\arraybackslash}p{0.84\textwidth} }
Rule & Cases \\
\hline
(a) & 
$[(1, 9), (24, 1)]$, $[(1, 17), (16, 1)]$, $[(2, 9), (15, 1)]$, $[(1, 25), (8, 1)]$, $[(1, 17), (1, 9), (7, 1)]$\\
\hline
(b) & 
$[(1, 19), (1, 11), (1, 3)]$\\
\hline
(c) & 
$[(3, 9), (6, 1)]$\\
\hline
(d) & 
$[(11, 3)]$, $[(3, 11)]$\\
\hline
(e) & 
$[(1, 13), (4, 5)]$\\
\hline
Open & 
$[(1, 27), (2, 3)]$\\
\hline
\caption{Non-pointed and non-perfect cases for rank 33. Each case is denoted as an array of ordered pairs $(n, r)$, indicating that there are $n$ components of rank $r$ in the universal grading of $\C$.}
\label{table: rank 33 cases}
\end{longtable}
\begin{longtable}[h]{ p{0.06\textwidth} | >{\raggedright\arraybackslash}p{0.84\textwidth} }
Rule & Cases \\
\hline
(a) & 
$[(1, 9), (26, 1)]$, $[(1, 17), (18, 1)]$, $[(2, 9), (17, 1)]$, $[(1, 25), (10, 1)]$, $[(1, 17), (1, 9), (9, 1)]$, $[(1, 33), (2, 1)]$, $[(1, 25), (1, 9), (1, 1)]$, $[(2, 17), (1, 1)]$\\
\hline
(b) & 
None\\
\hline
(c) & 
$[(3, 9), (8, 1)]$\\
\hline
(d) & 
$[(7, 5)]$, $[(5, 7)]$\\
\hline
(e) & 
$[(1, 11), (8, 3)]$, $[(1, 17), (2, 9)]$\\
\hline
Open & 
None\\
\hline
\caption{Non-pointed and non-perfect cases for rank 35. Each case is denoted as an array of ordered pairs $(n, r)$, indicating that there are $n$ components of rank $r$ in the universal grading of $\C$.}
\label{table: rank 35 cases}
\end{longtable}
\begin{longtable}[h]{ p{0.06\textwidth} | >{\raggedright\arraybackslash}p{0.84\textwidth} }
Rule & Cases \\
\hline
(a) & 
$[(1, 9), (28, 1)]$, $[(1, 17), (20, 1)]$, $[(2, 9), (19, 1)]$, $[(1, 25), (12, 1)]$, $[(1, 17), (1, 9), (11, 1)]$, $[(1, 33), (4, 1)]$, $[(1, 25), (1, 9), (3, 1)]$, $[(2, 17), (3, 1)]$\\
\hline
(b) & 
$[(4, 9), (1, 1)]$\\
\hline
(c) & 
$[(3, 9), (10, 1)]$\\
\hline
(d) & 
$[(1, 19), (6, 3)]$, $[(2, 11), (5, 3)]$, $[(1, 17), (2, 9), (2, 1)]$, $[(1, 23), (2, 7)]$\\
\hline
(e) & 
$[(2, 15), (1, 7)]$\\
\hline
Open & 
None\\
\hline
\caption{Non-pointed and non-perfect cases for rank 37. Each case is denoted as an array of ordered pairs $(n, r)$, indicating that there are $n$ components of rank $r$ in the universal grading of $\C$.}
\label{table: rank 37 cases}
\end{longtable}
\begin{longtable}[h]{ p{0.06\textwidth} | >{\raggedright\arraybackslash}p{0.84\textwidth} }
Rule & Cases \\
\hline
(a) & 
$[(1, 9), (30, 1)]$, $[(1, 17), (22, 1)]$, $[(2, 9), (21, 1)]$, $[(1, 25), (14, 1)]$, $[(1, 17), (1, 9), (13, 1)]$, $[(1, 33), (6, 1)]$, $[(1, 25), (1, 9), (5, 1)]$, $[(2, 17), (5, 1)]$\\
\hline
(b) & 
$[(4, 9), (3, 1)]$, $[(1, 19), (1, 11), (3, 3)]$, $[(1, 21), (1, 13), (1, 5)]$\\
\hline
(c) & 
$[(3, 9), (12, 1)]$\\
\hline
(d) & 
$[(13, 3)]$, $[(1, 17), (2, 9), (4, 1)]$, $[(1, 27), (4, 3)]$, $[(3, 11), (2, 3)]$, $[(1, 29), (2, 5)]$, $[(3, 13)]$\\
\hline
(e) & 
None\\
\hline
Open & 
None\\
\hline
\caption{Non-pointed and non-perfect cases for rank 39. Each case is denoted as an array of ordered pairs $(n, r)$, indicating that there are $n$ components of rank $r$ in the universal grading of $\C$.}
\label{table: rank 39 cases}
\end{longtable}
\begin{longtable}[h]{ p{0.06\textwidth} | >{\raggedright\arraybackslash}p{0.84\textwidth} }
Rule & Cases \\
\hline
(a) & 
$[(1, 9), (32, 1)]$, $[(1, 17), (24, 1)]$, $[(2, 9), (23, 1)]$, $[(1, 25), (16, 1)]$, $[(1, 17), (1, 9), (15, 1)]$, $[(1, 33), (8, 1)]$, $[(1, 25), (1, 9), (7, 1)]$, $[(2, 17), (7, 1)]$\\
\hline
(b) & 
$[(4, 9), (5, 1)]$, $[(1, 27), (1, 11), (1, 3)]$\\
\hline
(c) & 
$[(3, 9), (14, 1)]$\\
\hline
(d) & 
$[(1, 11), (10, 3)]$, $[(1, 17), (2, 9), (6, 1)]$, $[(2, 13), (3, 5)]$, $[(2, 19), (1, 3)]$, $[(1, 19), (2, 11)]$\\
\hline
(e) & 
None\\
\hline
Open & 
$[(1, 21), (4, 5)]$, $[(1, 35), (2, 3)]$\\
\hline
\caption{Non-pointed and non-perfect cases for rank 41. Each case is denoted as an array of ordered pairs $(n, r)$, indicating that there are $n$ components of rank $r$ in the universal grading of $\C$.}
\label{table: rank 41 cases}
\end{longtable}
\begin{longtable}[h]{ p{0.06\textwidth} | >{\raggedright\arraybackslash}p{0.84\textwidth} }
Rule & Cases \\
\hline
(a) & 
$[(1, 9), (34, 1)]$, $[(1, 17), (26, 1)]$, $[(2, 9), (25, 1)]$, $[(1, 25), (18, 1)]$, $[(1, 17), (1, 9), (17, 1)]$, $[(1, 33), (10, 1)]$, $[(1, 25), (1, 9), (9, 1)]$, $[(2, 17), (9, 1)]$, $[(1, 41), (2, 1)]$, $[(1, 33), (1, 9), (1, 1)]$, $[(1, 25), (1, 17), (1, 1)]$\\
\hline
(b) & 
$[(4, 9), (7, 1)]$\\
\hline
(c) & 
$[(3, 9), (16, 1)]$\\
\hline
(d) & 
$[(1, 17), (2, 9), (8, 1)]$, $[(1, 13), (6, 5)]$, $[(1, 15), (4, 7)]$, $[(2, 17), (1, 9)]$\\
\hline
(e) & 
$[(2, 11), (7, 3)]$, $[(1, 25), (2, 9)]$\\
\hline
Open & 
$[(1, 19), (8, 3)]$\\
\hline
\caption{Non-pointed and non-perfect cases for rank 43. Each case is denoted as an array of ordered pairs $(n, r)$, indicating that there are $n$ components of rank $r$ in the universal grading of $\C$.}
\label{table: rank 43 cases}
\end{longtable}
\begin{longtable}[h]{ p{0.06\textwidth} | >{\raggedright\arraybackslash}p{0.84\textwidth} }
Rule & Cases \\
\hline
(a) & 
$[(1, 9), (36, 1)]$, $[(1, 17), (28, 1)]$, $[(2, 9), (27, 1)]$, $[(1, 25), (20, 1)]$, $[(1, 17), (1, 9), (19, 1)]$, $[(1, 33), (12, 1)]$, $[(1, 25), (1, 9), (11, 1)]$, $[(2, 17), (11, 1)]$, $[(1, 41), (4, 1)]$, $[(1, 33), (1, 9), (3, 1)]$, $[(1, 25), (1, 17), (3, 1)]$\\
\hline
(b) & 
$[(4, 9), (9, 1)]$, $[(1, 19), (1, 11), (5, 3)]$, $[(1, 17), (3, 9), (1, 1)]$, $[(1, 23), (1, 15), (1, 7)]$\\
\hline
(c) & 
$[(3, 9), (18, 1)]$, $[(2, 17), (1, 9), (2, 1)]$\\
\hline
(d) & 
$[(1, 17), (2, 9), (10, 1)]$, $[(1, 27), (6, 3)]$, $[(3, 11), (4, 3)]$, $[(1, 25), (2, 9), (2, 1)]$, $[(5, 9)]$, $[(1, 31), (2, 7)]$\\
\hline
(e) & 
$[(15, 3)]$, $[(9, 5)]$, $[(3, 15)]$\\
\hline
Open & 
None\\
\hline
\caption{Non-pointed and non-perfect cases for rank 45. Each case is denoted as an array of ordered pairs $(n, r)$, indicating that there are $n$ components of rank $r$ in the universal grading of $\C$.}
\label{table: rank 45 cases}
\end{longtable}
\begin{longtable}[h]{ p{0.06\textwidth} | >{\raggedright\arraybackslash}p{0.84\textwidth} }
Rule & Cases \\
\hline
(a) & 
$[(1, 9), (38, 1)]$, $[(1, 17), (30, 1)]$, $[(2, 9), (29, 1)]$, $[(1, 25), (22, 1)]$, $[(1, 17), (1, 9), (21, 1)]$, $[(1, 33), (14, 1)]$, $[(1, 25), (1, 9), (13, 1)]$, $[(2, 17), (13, 1)]$, $[(1, 41), (6, 1)]$, $[(1, 33), (1, 9), (5, 1)]$, $[(1, 25), (1, 17), (5, 1)]$\\
\hline
(b) & 
$[(4, 9), (11, 1)]$, $[(1, 17), (3, 9), (3, 1)]$, $[(1, 27), (1, 11), (3, 3)]$, $[(1, 29), (1, 13), (1, 5)]$\\
\hline
(c) & 
$[(3, 9), (20, 1)]$, $[(2, 17), (1, 9), (4, 1)]$, $[(5, 9), (2, 1)]$\\
\hline
(d) & 
$[(1, 17), (2, 9), (12, 1)]$, $[(1, 11), (12, 3)]$, $[(1, 25), (2, 9), (4, 1)]$, $[(1, 35), (4, 3)]$, $[(2, 19), (3, 3)]$, $[(1, 19), (2, 11), (2, 3)]$, $[(4, 11), (1, 3)]$, $[(1, 37), (2, 5)]$, $[(1, 21), (2, 13)]$\\
\hline
(e) & 
$[(2, 21), (1, 5)]$\\
\hline
Open & 
None\\
\hline
\caption{Non-pointed and non-perfect cases for rank 47. Each case is denoted as an array of ordered pairs $(n, r)$, indicating that there are $n$ components of rank $r$ in the universal grading of $\C$.}
\label{table: rank 47 cases}
\end{longtable}
\begin{longtable}[h]{ p{0.06\textwidth} | >{\raggedright\arraybackslash}p{0.84\textwidth} }
Rule & Cases \\
\hline
(a) & 
$[(1, 9), (40, 1)]$, $[(1, 17), (32, 1)]$, $[(2, 9), (31, 1)]$, $[(1, 25), (24, 1)]$, $[(1, 17), (1, 9), (23, 1)]$, $[(1, 33), (16, 1)]$, $[(1, 25), (1, 9), (15, 1)]$, $[(2, 17), (15, 1)]$, $[(1, 41), (8, 1)]$, $[(1, 33), (1, 9), (7, 1)]$, $[(1, 25), (1, 17), (7, 1)]$\\
\hline
(b) & 
$[(4, 9), (13, 1)]$, $[(1, 17), (3, 9), (5, 1)]$, $[(1, 21), (1, 13), (3, 5)]$, $[(1, 35), (1, 11), (1, 3)]$, $[(1, 27), (1, 19), (1, 3)]$\\
\hline
(c) & 
$[(3, 9), (22, 1)]$, $[(2, 17), (1, 9), (6, 1)]$, $[(5, 9), (4, 1)]$\\
\hline
(d) & 
$[(1, 17), (2, 9), (14, 1)]$, $[(1, 19), (10, 3)]$, $[(2, 11), (9, 3)]$, $[(1, 25), (2, 9), (6, 1)]$, $[(3, 13), (2, 5)]$, $[(1, 27), (2, 11)]$, $[(2, 19), (1, 11)]$\\
\hline
(e) & 
$[(7, 7)]$\\
\hline
Open & 
$[(1, 29), (4, 5)]$, $[(1, 43), (2, 3)]$\\
\hline
\caption{Non-pointed and non-perfect cases for rank 49. Each case is denoted as an array of ordered pairs $(n, r)$, indicating that there are $n$ components of rank $r$ in the universal grading of $\C$.}
\label{table: rank 49 cases}
\end{longtable}
\begin{longtable}[h]{ p{0.06\textwidth} | >{\raggedright\arraybackslash}p{0.84\textwidth} }
Rule & Cases \\
\hline
(a) & 
$[(1, 9), (42, 1)]$, $[(1, 17), (34, 1)]$, $[(2, 9), (33, 1)]$, $[(1, 25), (26, 1)]$, $[(1, 17), (1, 9), (25, 1)]$, $[(1, 33), (18, 1)]$, $[(1, 25), (1, 9), (17, 1)]$, $[(2, 17), (17, 1)]$, $[(1, 41), (10, 1)]$, $[(1, 33), (1, 9), (9, 1)]$, $[(1, 25), (1, 17), (9, 1)]$, $[(1, 49), (2, 1)]$, $[(1, 41), (1, 9), (1, 1)]$, $[(1, 33), (1, 17), (1, 1)]$, $[(2, 25), (1, 1)]$\\
\hline
(b) & 
$[(4, 9), (15, 1)]$, $[(1, 17), (3, 9), (7, 1)]$, $[(1, 19), (1, 11), (7, 3)]$, $[(1, 25), (1, 17), (1, 9)]$\\
\hline
(c) & 
$[(3, 9), (24, 1)]$, $[(2, 17), (1, 9), (8, 1)]$, $[(5, 9), (6, 1)]$\\
\hline
(d) & 
$[(1, 17), (2, 9), (16, 1)]$, $[(17, 3)]$, $[(1, 25), (2, 9), (8, 1)]$, $[(1, 21), (6, 5)]$, $[(2, 13), (5, 5)]$, $[(1, 23), (4, 7)]$, $[(2, 15), (3, 7)]$, $[(3, 17)]$\\
\hline
(e) & 
$[(1, 27), (8, 3)]$\\
\hline
Open & 
$[(3, 11), (6, 3)]$, $[(1, 33), (2, 9)]$\\
\hline
\caption{Non-pointed and non-perfect cases for rank 51. Each case is denoted as an array of ordered pairs $(n, r)$, indicating that there are $n$ components of rank $r$ in the universal grading of $\C$.}
\label{table: rank 51 cases}
\end{longtable}
\begin{longtable}[h]{ p{0.06\textwidth} | >{\raggedright\arraybackslash}p{0.84\textwidth} }
Rule & Cases \\
\hline
(a) & 
$[(1, 9), (44, 1)]$, $[(1, 17), (36, 1)]$, $[(2, 9), (35, 1)]$, $[(1, 25), (28, 1)]$, $[(1, 17), (1, 9), (27, 1)]$, $[(1, 33), (20, 1)]$, $[(1, 25), (1, 9), (19, 1)]$, $[(2, 17), (19, 1)]$, $[(1, 41), (12, 1)]$, $[(1, 33), (1, 9), (11, 1)]$, $[(1, 25), (1, 17), (11, 1)]$, $[(1, 49), (4, 1)]$, $[(1, 41), (1, 9), (3, 1)]$, $[(1, 33), (1, 17), (3, 1)]$, $[(2, 25), (3, 1)]$\\
\hline
(b) & 
$[(4, 9), (17, 1)]$, $[(1, 17), (3, 9), (9, 1)]$, $[(1, 27), (1, 11), (5, 3)]$, $[(1, 25), (1, 17), (1, 9), (2, 1)]$, $[(1, 25), (3, 9), (1, 1)]$, $[(2, 17), (2, 9), (1, 1)]$, $[(1, 31), (1, 15), (1, 7)]$\\
\hline
(c) & 
$[(3, 9), (26, 1)]$, $[(2, 17), (1, 9), (10, 1)]$, $[(5, 9), (8, 1)]$\\
\hline
(d) & 
$[(1, 17), (2, 9), (18, 1)]$, $[(1, 11), (14, 3)]$, $[(1, 25), (2, 9), (10, 1)]$, $[(1, 13), (8, 5)]$, $[(1, 35), (6, 3)]$, $[(2, 19), (5, 3)]$, $[(1, 19), (2, 11), (4, 3)]$, $[(4, 11), (3, 3)]$, $[(1, 33), (2, 9), (2, 1)]$, $[(3, 17), (2, 1)]$, $[(1, 17), (4, 9)]$, $[(1, 39), (2, 7)]$, $[(2, 23), (1, 7)]$\\
\hline
(e) & 
$[(1, 23), (2, 15)]$\\
\hline
Open & 
None\\
\hline
\caption{Non-pointed and non-perfect cases for rank 53. Each case is denoted as an array of ordered pairs $(n, r)$, indicating that there are $n$ components of rank $r$ in the universal grading of $\C$.}
\label{table: rank 53 cases}
\end{longtable}
\begin{longtable}[h]{ p{0.06\textwidth} | >{\raggedright\arraybackslash}p{0.84\textwidth} }
Rule & Cases \\
\hline
(a) & 
$[(1, 9), (46, 1)]$, $[(1, 17), (38, 1)]$, $[(2, 9), (37, 1)]$, $[(1, 25), (30, 1)]$, $[(1, 17), (1, 9), (29, 1)]$, $[(1, 33), (22, 1)]$, $[(1, 25), (1, 9), (21, 1)]$, $[(2, 17), (21, 1)]$, $[(1, 41), (14, 1)]$, $[(1, 33), (1, 9), (13, 1)]$, $[(1, 25), (1, 17), (13, 1)]$, $[(1, 49), (6, 1)]$, $[(1, 41), (1, 9), (5, 1)]$, $[(1, 33), (1, 17), (5, 1)]$, $[(2, 25), (5, 1)]$\\
\hline
(b) & 
$[(4, 9), (19, 1)]$, $[(1, 17), (3, 9), (11, 1)]$, $[(1, 25), (1, 17), (1, 9), (4, 1)]$, $[(1, 25), (3, 9), (3, 1)]$, $[(2, 17), (2, 9), (3, 1)]$, $[(6, 9), (1, 1)]$, $[(1, 35), (1, 11), (3, 3)]$, $[(1, 27), (1, 19), (3, 3)]$, $[(1, 19), (3, 11), (1, 3)]$, $[(1, 37), (1, 13), (1, 5)]$, $[(1, 29), (1, 21), (1, 5)]$\\
\hline
(c) & 
$[(3, 9), (28, 1)]$, $[(2, 17), (1, 9), (12, 1)]$, $[(5, 9), (10, 1)]$\\
\hline
(d) & 
$[(1, 17), (2, 9), (20, 1)]$, $[(1, 25), (2, 9), (12, 1)]$, $[(1, 19), (12, 3)]$, $[(2, 11), (11, 3)]$, $[(11, 5)]$, $[(1, 33), (2, 9), (4, 1)]$, $[(3, 17), (4, 1)]$, $[(1, 17), (4, 9), (2, 1)]$, $[(1, 43), (4, 3)]$, $[(1, 27), (2, 11), (2, 3)]$, $[(2, 19), (1, 11), (2, 3)]$, $[(5, 11)]$, $[(1, 45), (2, 5)]$, $[(1, 29), (2, 13)]$\\
\hline
(e) & 
$[(2, 21), (1, 13)]$\\
\hline
Open & 
None\\
\hline
\caption{Non-pointed and non-perfect cases for rank 55. Each case is denoted as an array of ordered pairs $(n, r)$, indicating that there are $n$ components of rank $r$ in the universal grading of $\C$.}
\label{table: rank 55 cases}
\end{longtable}
\begin{longtable}[h]{ p{0.06\textwidth} | >{\raggedright\arraybackslash}p{0.84\textwidth} }
Rule & Cases \\
\hline
(a) & 
$[(1, 9), (48, 1)]$, $[(1, 17), (40, 1)]$, $[(2, 9), (39, 1)]$, $[(1, 25), (32, 1)]$, $[(1, 17), (1, 9), (31, 1)]$, $[(1, 33), (24, 1)]$, $[(1, 25), (1, 9), (23, 1)]$, $[(2, 17), (23, 1)]$, $[(1, 41), (16, 1)]$, $[(1, 33), (1, 9), (15, 1)]$, $[(1, 25), (1, 17), (15, 1)]$, $[(1, 49), (8, 1)]$, $[(1, 41), (1, 9), (7, 1)]$, $[(1, 33), (1, 17), (7, 1)]$, $[(2, 25), (7, 1)]$\\
\hline
(b) & 
$[(4, 9), (21, 1)]$, $[(1, 17), (3, 9), (13, 1)]$, $[(1, 19), (1, 11), (9, 3)]$, $[(1, 25), (1, 17), (1, 9), (6, 1)]$, $[(1, 25), (3, 9), (5, 1)]$, $[(2, 17), (2, 9), (5, 1)]$, $[(6, 9), (3, 1)]$, $[(1, 29), (1, 13), (3, 5)]$, $[(1, 43), (1, 11), (1, 3)]$, $[(1, 35), (1, 19), (1, 3)]$, $[(1, 27), (1, 19), (1, 11)]$\\
\hline
(c) & 
$[(3, 9), (30, 1)]$, $[(2, 17), (1, 9), (14, 1)]$, $[(5, 9), (12, 1)]$\\
\hline
(d) & 
$[(1, 17), (2, 9), (22, 1)]$, $[(19, 3)]$, $[(1, 25), (2, 9), (14, 1)]$, $[(1, 27), (10, 3)]$, $[(3, 11), (8, 3)]$, $[(1, 33), (2, 9), (6, 1)]$, $[(1, 17), (4, 9), (4, 1)]$, $[(2, 21), (3, 5)]$, $[(1, 21), (2, 13), (2, 5)]$, $[(4, 13), (1, 5)]$, $[(1, 35), (2, 11)]$, $[(3, 19)]$\\
\hline
(e) & 
$[(3, 17), (6, 1)]$, $[(1, 15), (6, 7)]$, $[(2, 27), (1, 3)]$\\
\hline
Open & 
$[(1, 37), (4, 5)]$, $[(1, 51), (2, 3)]$\\
\hline
\caption{Non-pointed and non-perfect cases for rank 57. Each case is denoted as an array of ordered pairs $(n, r)$, indicating that there are $n$ components of rank $r$ in the universal grading of $\C$.}
\label{table: rank 57 cases}
\end{longtable}
\begin{longtable}[h]{ p{0.06\textwidth} | >{\raggedright\arraybackslash}p{0.84\textwidth} }
Rule & Cases \\
\hline
(a) & 
$[(1, 9), (50, 1)]$, $[(1, 17), (42, 1)]$, $[(2, 9), (41, 1)]$, $[(1, 25), (34, 1)]$, $[(1, 17), (1, 9), (33, 1)]$, $[(1, 33), (26, 1)]$, $[(1, 25), (1, 9), (25, 1)]$, $[(2, 17), (25, 1)]$, $[(1, 41), (18, 1)]$, $[(1, 33), (1, 9), (17, 1)]$, $[(1, 25), (1, 17), (17, 1)]$, $[(1, 49), (10, 1)]$, $[(1, 41), (1, 9), (9, 1)]$, $[(1, 33), (1, 17), (9, 1)]$, $[(2, 25), (9, 1)]$, $[(1, 57), (2, 1)]$, $[(1, 49), (1, 9), (1, 1)]$, $[(1, 41), (1, 17), (1, 1)]$, $[(1, 33), (1, 25), (1, 1)]$\\
\hline
(b) & 
$[(4, 9), (23, 1)]$, $[(1, 17), (3, 9), (15, 1)]$, $[(1, 25), (1, 17), (1, 9), (8, 1)]$, $[(1, 25), (3, 9), (7, 1)]$, $[(2, 17), (2, 9), (7, 1)]$, $[(6, 9), (5, 1)]$, $[(1, 27), (1, 11), (7, 3)]$, $[(1, 21), (1, 13), (5, 5)]$, $[(1, 23), (1, 15), (3, 7)]$, $[(1, 33), (1, 17), (1, 9)]$\\
\hline
(c) & 
$[(3, 9), (32, 1)]$, $[(2, 17), (1, 9), (16, 1)]$, $[(5, 9), (14, 1)]$\\
\hline
(d) & 
$[(1, 17), (2, 9), (24, 1)]$, $[(1, 25), (2, 9), (16, 1)]$, $[(1, 11), (16, 3)]$, $[(1, 33), (2, 9), (8, 1)]$, $[(3, 17), (8, 1)]$, $[(1, 17), (4, 9), (6, 1)]$, $[(1, 19), (2, 11), (6, 3)]$, $[(1, 29), (6, 5)]$, $[(3, 13), (4, 5)]$, $[(1, 31), (4, 7)]$, $[(3, 15), (2, 7)]$, $[(2, 25), (1, 9)]$, $[(1, 25), (2, 17)]$\\
\hline
(e) & 
$[(2, 19), (7, 3)]$, $[(4, 11), (5, 3)]$\\
\hline
Open & 
$[(1, 35), (8, 3)]$, $[(1, 41), (2, 9)]$\\
\hline
\caption{Non-pointed and non-perfect cases for rank 59. Each case is denoted as an array of ordered pairs $(n, r)$, indicating that there are $n$ components of rank $r$ in the universal grading of $\C$.}
\label{table: rank 59 cases}
\end{longtable}
\begin{longtable}[h]{ p{0.06\textwidth} | >{\raggedright\arraybackslash}p{0.84\textwidth} }
Rule & Cases \\
\hline
(a) & 
$[(1, 9), (52, 1)]$, $[(1, 17), (44, 1)]$, $[(2, 9), (43, 1)]$, $[(1, 25), (36, 1)]$, $[(1, 17), (1, 9), (35, 1)]$, $[(1, 33), (28, 1)]$, $[(1, 25), (1, 9), (27, 1)]$, $[(2, 17), (27, 1)]$, $[(1, 41), (20, 1)]$, $[(1, 33), (1, 9), (19, 1)]$, $[(1, 25), (1, 17), (19, 1)]$, $[(1, 49), (12, 1)]$, $[(1, 41), (1, 9), (11, 1)]$, $[(1, 33), (1, 17), (11, 1)]$, $[(2, 25), (11, 1)]$, $[(1, 57), (4, 1)]$, $[(1, 49), (1, 9), (3, 1)]$, $[(1, 41), (1, 17), (3, 1)]$, $[(1, 33), (1, 25), (3, 1)]$\\
\hline
(b) & 
$[(4, 9), (25, 1)]$, $[(1, 17), (3, 9), (17, 1)]$, $[(1, 25), (1, 17), (1, 9), (10, 1)]$, $[(1, 25), (3, 9), (9, 1)]$, $[(2, 17), (2, 9), (9, 1)]$, $[(6, 9), (7, 1)]$, $[(1, 35), (1, 11), (5, 3)]$, $[(1, 27), (1, 19), (5, 3)]$, $[(1, 19), (3, 11), (3, 3)]$, $[(1, 33), (1, 17), (1, 9), (2, 1)]$, $[(1, 33), (3, 9), (1, 1)]$, $[(1, 25), (1, 17), (2, 9), (1, 1)]$, $[(3, 17), (1, 9), (1, 1)]$, $[(1, 39), (1, 15), (1, 7)]$, $[(1, 31), (1, 23), (1, 7)]$\\
\hline
(c) & 
$[(3, 9), (34, 1)]$, $[(2, 17), (1, 9), (18, 1)]$, $[(5, 9), (16, 1)]$, $[(2, 25), (1, 9), (2, 1)]$\\
\hline
(d) & 
$[(1, 17), (2, 9), (26, 1)]$, $[(1, 25), (2, 9), (18, 1)]$, $[(1, 33), (2, 9), (10, 1)]$, $[(3, 17), (10, 1)]$, $[(1, 17), (4, 9), (8, 1)]$, $[(1, 21), (8, 5)]$, $[(1, 43), (6, 3)]$, $[(1, 27), (2, 11), (4, 3)]$, $[(2, 19), (1, 11), (4, 3)]$, $[(5, 11), (2, 3)]$, $[(1, 41), (2, 9), (2, 1)]$, $[(1, 25), (2, 17), (2, 1)]$, $[(1, 25), (4, 9)]$, $[(2, 17), (3, 9)]$, $[(1, 47), (2, 7)]$, $[(2, 23), (1, 15)]$\\
\hline
(e) & 
$[(1, 19), (14, 3)]$, $[(2, 11), (13, 3)]$, $[(2, 13), (7, 5)]$\\
\hline
Open & 
$[(1, 31), (2, 15)]$\\
\hline
\caption{Non-pointed and non-perfect cases for rank 61. Each case is denoted as an array of ordered pairs $(n, r)$, indicating that there are $n$ components of rank $r$ in the universal grading of $\C$.}
\label{table: rank 61 cases}
\end{longtable}
\begin{longtable}[h]{ p{0.06\textwidth} | >{\raggedright\arraybackslash}p{0.84\textwidth} }
Rule & Cases \\
\hline
(a) & 
$[(1, 9), (54, 1)]$, $[(1, 17), (46, 1)]$, $[(2, 9), (45, 1)]$, $[(1, 25), (38, 1)]$, $[(1, 17), (1, 9), (37, 1)]$, $[(1, 33), (30, 1)]$, $[(1, 25), (1, 9), (29, 1)]$, $[(2, 17), (29, 1)]$, $[(1, 41), (22, 1)]$, $[(1, 33), (1, 9), (21, 1)]$, $[(1, 25), (1, 17), (21, 1)]$, $[(1, 49), (14, 1)]$, $[(1, 41), (1, 9), (13, 1)]$, $[(1, 33), (1, 17), (13, 1)]$, $[(2, 25), (13, 1)]$, $[(1, 57), (6, 1)]$, $[(1, 49), (1, 9), (5, 1)]$, $[(1, 41), (1, 17), (5, 1)]$, $[(1, 33), (1, 25), (5, 1)]$\\
\hline
(b) & 
$[(4, 9), (27, 1)]$, $[(1, 17), (3, 9), (19, 1)]$, $[(1, 25), (1, 17), (1, 9), (12, 1)]$, $[(1, 25), (3, 9), (11, 1)]$, $[(2, 17), (2, 9), (11, 1)]$, $[(6, 9), (9, 1)]$, $[(1, 19), (1, 11), (11, 3)]$, $[(1, 33), (1, 17), (1, 9), (4, 1)]$, $[(1, 33), (3, 9), (3, 1)]$, $[(1, 25), (1, 17), (2, 9), (3, 1)]$, $[(3, 17), (1, 9), (3, 1)]$, $[(1, 17), (5, 9), (1, 1)]$, $[(1, 43), (1, 11), (3, 3)]$, $[(1, 35), (1, 19), (3, 3)]$, $[(1, 27), (1, 19), (1, 11), (2, 3)]$, $[(1, 27), (3, 11), (1, 3)]$, $[(1, 45), (1, 13), (1, 5)]$, $[(1, 37), (1, 21), (1, 5)]$, $[(1, 29), (1, 21), (1, 13)]$\\
\hline
(c) & 
$[(3, 9), (36, 1)]$, $[(2, 17), (1, 9), (20, 1)]$, $[(5, 9), (18, 1)]$, $[(2, 25), (1, 9), (4, 1)]$, $[(2, 17), (3, 9), (2, 1)]$\\
\hline
(d) & 
$[(1, 17), (2, 9), (28, 1)]$, $[(1, 25), (2, 9), (20, 1)]$, $[(1, 33), (2, 9), (12, 1)]$, $[(1, 17), (4, 9), (10, 1)]$, $[(1, 27), (12, 3)]$, $[(3, 11), (10, 3)]$, $[(1, 13), (10, 5)]$, $[(1, 41), (2, 9), (4, 1)]$, $[(1, 25), (2, 17), (4, 1)]$, $[(1, 25), (4, 9), (2, 1)]$, $[(7, 9)]$, $[(1, 51), (4, 3)]$, $[(2, 27), (3, 3)]$, $[(1, 35), (2, 11), (2, 3)]$, $[(3, 19), (2, 3)]$, $[(2, 19), (2, 11), (1, 3)]$, $[(1, 19), (4, 11)]$, $[(1, 53), (2, 5)]$, $[(2, 29), (1, 5)]$, $[(1, 37), (2, 13)]$\\
\hline
(e) & 
$[(21, 3)]$, $[(3, 17), (12, 1)]$, $[(9, 7)]$, $[(3, 21)]$\\
\hline
Open & 
None\\
\hline
\caption{Non-pointed and non-perfect cases for rank 63. Each case is denoted as an array of ordered pairs $(n, r)$, indicating that there are $n$ components of rank $r$ in the universal grading of $\C$.}
\label{table: rank 63 cases}
\end{longtable}
\begin{longtable}[h]{ p{0.06\textwidth} | >{\raggedright\arraybackslash}p{0.84\textwidth} }
Rule & Cases \\
\hline
(a) & 
$[(1, 9), (56, 1)]$, $[(1, 17), (48, 1)]$, $[(2, 9), (47, 1)]$, $[(1, 25), (40, 1)]$, $[(1, 17), (1, 9), (39, 1)]$, $[(1, 33), (32, 1)]$, $[(1, 25), (1, 9), (31, 1)]$, $[(2, 17), (31, 1)]$, $[(1, 41), (24, 1)]$, $[(1, 33), (1, 9), (23, 1)]$, $[(1, 25), (1, 17), (23, 1)]$, $[(1, 49), (16, 1)]$, $[(1, 41), (1, 9), (15, 1)]$, $[(1, 33), (1, 17), (15, 1)]$, $[(2, 25), (15, 1)]$, $[(1, 57), (8, 1)]$, $[(1, 49), (1, 9), (7, 1)]$, $[(1, 41), (1, 17), (7, 1)]$, $[(1, 33), (1, 25), (7, 1)]$\\
\hline
(b) & 
$[(4, 9), (29, 1)]$, $[(1, 17), (3, 9), (21, 1)]$, $[(1, 25), (1, 17), (1, 9), (14, 1)]$, $[(1, 25), (3, 9), (13, 1)]$, $[(2, 17), (2, 9), (13, 1)]$, $[(6, 9), (11, 1)]$, $[(1, 27), (1, 11), (9, 3)]$, $[(1, 33), (1, 17), (1, 9), (6, 1)]$, $[(1, 33), (3, 9), (5, 1)]$, $[(1, 25), (1, 17), (2, 9), (5, 1)]$, $[(3, 17), (1, 9), (5, 1)]$, $[(1, 17), (5, 9), (3, 1)]$, $[(1, 37), (1, 13), (3, 5)]$, $[(1, 29), (1, 21), (3, 5)]$, $[(1, 21), (3, 13), (1, 5)]$, $[(1, 51), (1, 11), (1, 3)]$, $[(1, 43), (1, 19), (1, 3)]$, $[(1, 35), (1, 27), (1, 3)]$, $[(1, 35), (1, 19), (1, 11)]$\\
\hline
(c) & 
$[(3, 9), (38, 1)]$, $[(2, 17), (1, 9), (22, 1)]$, $[(5, 9), (20, 1)]$, $[(2, 25), (1, 9), (6, 1)]$, $[(2, 17), (3, 9), (4, 1)]$, $[(7, 9), (2, 1)]$\\
\hline
(d) & 
$[(1, 17), (2, 9), (30, 1)]$, $[(1, 25), (2, 9), (22, 1)]$, $[(1, 11), (18, 3)]$, $[(1, 33), (2, 9), (14, 1)]$, $[(3, 17), (14, 1)]$, $[(1, 17), (4, 9), (12, 1)]$, $[(13, 5)]$, $[(1, 35), (10, 3)]$, $[(2, 19), (9, 3)]$, $[(1, 19), (2, 11), (8, 3)]$, $[(4, 11), (7, 3)]$, $[(1, 41), (2, 9), (6, 1)]$, $[(1, 25), (2, 17), (6, 1)]$, $[(1, 25), (4, 9), (4, 1)]$, $[(2, 15), (5, 7)]$, $[(1, 29), (2, 13), (2, 5)]$, $[(2, 21), (1, 13), (2, 5)]$, $[(5, 13)]$, $[(1, 43), (2, 11)]$, $[(1, 27), (2, 19)]$\\
\hline
(e) & 
$[(1, 23), (6, 7)]$, $[(2, 27), (1, 11)]$\\
\hline
Open & 
$[(1, 45), (4, 5)]$, $[(1, 59), (2, 3)]$\\
\hline
\caption{Non-pointed and non-perfect cases for rank 65. Each case is denoted as an array of ordered pairs $(n, r)$, indicating that there are $n$ components of rank $r$ in the universal grading of $\C$.}
\label{table: rank 65 cases}
\end{longtable}
\begin{longtable}[h]{ p{0.06\textwidth} | >{\raggedright\arraybackslash}p{0.84\textwidth} }
Rule & Cases \\
\hline
(a) & 
$[(1, 9), (58, 1)]$, $[(1, 17), (50, 1)]$, $[(2, 9), (49, 1)]$, $[(1, 25), (42, 1)]$, $[(1, 17), (1, 9), (41, 1)]$, $[(1, 33), (34, 1)]$, $[(1, 25), (1, 9), (33, 1)]$, $[(2, 17), (33, 1)]$, $[(1, 41), (26, 1)]$, $[(1, 33), (1, 9), (25, 1)]$, $[(1, 25), (1, 17), (25, 1)]$, $[(1, 49), (18, 1)]$, $[(1, 41), (1, 9), (17, 1)]$, $[(1, 33), (1, 17), (17, 1)]$, $[(2, 25), (17, 1)]$, $[(1, 57), (10, 1)]$, $[(1, 49), (1, 9), (9, 1)]$, $[(1, 41), (1, 17), (9, 1)]$, $[(1, 33), (1, 25), (9, 1)]$, $[(1, 65), (2, 1)]$, $[(1, 57), (1, 9), (1, 1)]$, $[(1, 49), (1, 17), (1, 1)]$, $[(1, 41), (1, 25), (1, 1)]$, $[(2, 33), (1, 1)]$\\
\hline
(b) & 
$[(4, 9), (31, 1)]$, $[(1, 17), (3, 9), (23, 1)]$, $[(1, 25), (1, 17), (1, 9), (16, 1)]$, $[(1, 25), (3, 9), (15, 1)]$, $[(2, 17), (2, 9), (15, 1)]$, $[(6, 9), (13, 1)]$, $[(1, 33), (1, 17), (1, 9), (8, 1)]$, $[(1, 33), (3, 9), (7, 1)]$, $[(1, 25), (1, 17), (2, 9), (7, 1)]$, $[(3, 17), (1, 9), (7, 1)]$, $[(1, 17), (5, 9), (5, 1)]$, $[(1, 35), (1, 11), (7, 3)]$, $[(1, 27), (1, 19), (7, 3)]$, $[(1, 19), (3, 11), (5, 3)]$, $[(1, 29), (1, 13), (5, 5)]$, $[(1, 31), (1, 15), (3, 7)]$, $[(1, 41), (1, 17), (1, 9)]$, $[(1, 33), (1, 25), (1, 9)]$\\
\hline
(c) & 
$[(3, 9), (40, 1)]$, $[(2, 17), (1, 9), (24, 1)]$, $[(5, 9), (22, 1)]$, $[(2, 25), (1, 9), (8, 1)]$, $[(2, 17), (3, 9), (6, 1)]$, $[(7, 9), (4, 1)]$\\
\hline
(d) & 
$[(1, 17), (2, 9), (32, 1)]$, $[(1, 25), (2, 9), (24, 1)]$, $[(1, 33), (2, 9), (16, 1)]$, $[(3, 17), (16, 1)]$, $[(1, 17), (4, 9), (14, 1)]$, $[(1, 19), (16, 3)]$, $[(2, 11), (15, 3)]$, $[(1, 41), (2, 9), (8, 1)]$, $[(1, 25), (2, 17), (8, 1)]$, $[(1, 25), (4, 9), (6, 1)]$, $[(1, 27), (2, 11), (6, 3)]$, $[(2, 19), (1, 11), (6, 3)]$, $[(1, 37), (6, 5)]$, $[(2, 21), (5, 5)]$, $[(1, 21), (2, 13), (4, 5)]$, $[(4, 13), (3, 5)]$, $[(1, 39), (4, 7)]$, $[(2, 23), (3, 7)]$, $[(1, 23), (2, 15), (2, 7)]$, $[(1, 33), (2, 17)]$, $[(2, 25), (1, 17)]$\\
\hline
(e) & 
$[(5, 11), (4, 3)]$, $[(4, 15), (1, 7)]$\\
\hline
Open & 
$[(1, 43), (8, 3)]$, $[(1, 49), (2, 9)]$\\
\hline
\caption{Non-pointed and non-perfect cases for rank 67. Each case is denoted as an array of ordered pairs $(n, r)$, indicating that there are $n$ components of rank $r$ in the universal grading of $\C$.}
\label{table: rank 67 cases}
\end{longtable}
\begin{longtable}[h]{ p{0.06\textwidth} | >{\raggedright\arraybackslash}p{0.84\textwidth} }
Rule & Cases \\
\hline
(a) & 
$[(1, 9), (60, 1)]$, $[(1, 17), (52, 1)]$, $[(2, 9), (51, 1)]$, $[(1, 25), (44, 1)]$, $[(1, 17), (1, 9), (43, 1)]$, $[(1, 33), (36, 1)]$, $[(1, 25), (1, 9), (35, 1)]$, $[(2, 17), (35, 1)]$, $[(1, 41), (28, 1)]$, $[(1, 33), (1, 9), (27, 1)]$, $[(1, 25), (1, 17), (27, 1)]$, $[(1, 49), (20, 1)]$, $[(1, 41), (1, 9), (19, 1)]$, $[(1, 33), (1, 17), (19, 1)]$, $[(2, 25), (19, 1)]$, $[(1, 57), (12, 1)]$, $[(1, 49), (1, 9), (11, 1)]$, $[(1, 41), (1, 17), (11, 1)]$, $[(1, 33), (1, 25), (11, 1)]$, $[(1, 65), (4, 1)]$, $[(1, 57), (1, 9), (3, 1)]$, $[(1, 49), (1, 17), (3, 1)]$, $[(1, 41), (1, 25), (3, 1)]$, $[(2, 33), (3, 1)]$\\
\hline
(b) & 
$[(4, 9), (33, 1)]$, $[(1, 17), (3, 9), (25, 1)]$, $[(1, 25), (1, 17), (1, 9), (18, 1)]$, $[(1, 25), (3, 9), (17, 1)]$, $[(2, 17), (2, 9), (17, 1)]$, $[(6, 9), (15, 1)]$, $[(1, 19), (1, 11), (13, 3)]$, $[(1, 33), (1, 17), (1, 9), (10, 1)]$, $[(1, 33), (3, 9), (9, 1)]$, $[(1, 25), (1, 17), (2, 9), (9, 1)]$, $[(3, 17), (1, 9), (9, 1)]$, $[(1, 17), (5, 9), (7, 1)]$, $[(1, 21), (1, 13), (7, 5)]$, $[(1, 43), (1, 11), (5, 3)]$, $[(1, 35), (1, 19), (5, 3)]$, $[(1, 27), (1, 19), (1, 11), (4, 3)]$, $[(1, 27), (3, 11), (3, 3)]$, $[(1, 41), (1, 17), (1, 9), (2, 1)]$, $[(1, 33), (1, 25), (1, 9), (2, 1)]$, $[(1, 41), (3, 9), (1, 1)]$, $[(1, 33), (1, 17), (2, 9), (1, 1)]$, $[(2, 25), (2, 9), (1, 1)]$, $[(1, 25), (2, 17), (1, 9), (1, 1)]$, $[(4, 17), (1, 1)]$, $[(1, 25), (1, 17), (3, 9)]$, $[(1, 47), (1, 15), (1, 7)]$, $[(1, 39), (1, 23), (1, 7)]$, $[(1, 31), (1, 23), (1, 15)]$\\
\hline
(c) & 
$[(3, 9), (42, 1)]$, $[(2, 17), (1, 9), (26, 1)]$, $[(5, 9), (24, 1)]$, $[(2, 25), (1, 9), (10, 1)]$, $[(2, 17), (3, 9), (8, 1)]$, $[(7, 9), (6, 1)]$\\
\hline
(d) & 
$[(1, 17), (2, 9), (34, 1)]$, $[(1, 25), (2, 9), (26, 1)]$, $[(23, 3)]$, $[(1, 33), (2, 9), (18, 1)]$, $[(1, 17), (4, 9), (16, 1)]$, $[(1, 41), (2, 9), (10, 1)]$, $[(1, 25), (2, 17), (10, 1)]$, $[(1, 25), (4, 9), (8, 1)]$, $[(1, 29), (8, 5)]$, $[(1, 51), (6, 3)]$, $[(2, 27), (5, 3)]$, $[(1, 35), (2, 11), (4, 3)]$, $[(3, 19), (4, 3)]$, $[(2, 19), (2, 11), (3, 3)]$, $[(1, 19), (4, 11), (2, 3)]$, $[(6, 11), (1, 3)]$, $[(1, 49), (2, 9), (2, 1)]$, $[(1, 33), (2, 17), (2, 1)]$, $[(2, 25), (1, 17), (2, 1)]$, $[(1, 33), (4, 9)]$, $[(3, 17), (2, 9)]$, $[(1, 55), (2, 7)]$, $[(2, 31), (1, 7)]$, $[(3, 23)]$\\
\hline
(e) & 
$[(3, 17), (18, 1)]$, $[(1, 27), (14, 3)]$, $[(3, 11), (12, 3)]$, $[(3, 13), (6, 5)]$\\
\hline
Open & 
$[(1, 39), (2, 15)]$\\
\hline
\caption{Non-pointed and non-perfect cases for rank 69. Each case is denoted as an array of ordered pairs $(n, r)$, indicating that there are $n$ components of rank $r$ in the universal grading of $\C$.}
\label{table: rank 69 cases}
\end{longtable}
\begin{longtable}[h]{ p{0.06\textwidth} | >{\raggedright\arraybackslash}p{0.84\textwidth} }
Rule & Cases \\
\hline
(a) & 
$[(1, 9), (62, 1)]$, $[(1, 17), (54, 1)]$, $[(2, 9), (53, 1)]$, $[(1, 25), (46, 1)]$, $[(1, 17), (1, 9), (45, 1)]$, $[(1, 33), (38, 1)]$, $[(1, 25), (1, 9), (37, 1)]$, $[(2, 17), (37, 1)]$, $[(1, 41), (30, 1)]$, $[(1, 33), (1, 9), (29, 1)]$, $[(1, 25), (1, 17), (29, 1)]$, $[(1, 49), (22, 1)]$, $[(1, 41), (1, 9), (21, 1)]$, $[(1, 33), (1, 17), (21, 1)]$, $[(2, 25), (21, 1)]$, $[(1, 57), (14, 1)]$, $[(1, 49), (1, 9), (13, 1)]$, $[(1, 41), (1, 17), (13, 1)]$, $[(1, 33), (1, 25), (13, 1)]$, $[(1, 65), (6, 1)]$, $[(1, 57), (1, 9), (5, 1)]$, $[(1, 49), (1, 17), (5, 1)]$, $[(1, 41), (1, 25), (5, 1)]$, $[(2, 33), (5, 1)]$\\
\hline
(b) & 
$[(4, 9), (35, 1)]$, $[(1, 17), (3, 9), (27, 1)]$, $[(1, 25), (1, 17), (1, 9), (20, 1)]$, $[(1, 25), (3, 9), (19, 1)]$, $[(2, 17), (2, 9), (19, 1)]$, $[(6, 9), (17, 1)]$, $[(1, 33), (1, 17), (1, 9), (12, 1)]$, $[(1, 33), (3, 9), (11, 1)]$, $[(1, 25), (1, 17), (2, 9), (11, 1)]$, $[(3, 17), (1, 9), (11, 1)]$, $[(1, 17), (5, 9), (9, 1)]$, $[(1, 27), (1, 11), (11, 3)]$, $[(1, 41), (1, 17), (1, 9), (4, 1)]$, $[(1, 33), (1, 25), (1, 9), (4, 1)]$, $[(1, 41), (3, 9), (3, 1)]$, $[(1, 33), (1, 17), (2, 9), (3, 1)]$, $[(2, 25), (2, 9), (3, 1)]$, $[(1, 25), (2, 17), (1, 9), (3, 1)]$, $[(4, 17), (3, 1)]$, $[(1, 25), (1, 17), (3, 9), (2, 1)]$, $[(1, 25), (5, 9), (1, 1)]$, $[(2, 17), (4, 9), (1, 1)]$, $[(1, 51), (1, 11), (3, 3)]$, $[(1, 43), (1, 19), (3, 3)]$, $[(1, 35), (1, 27), (3, 3)]$, $[(1, 35), (1, 19), (1, 11), (2, 3)]$, $[(1, 35), (3, 11), (1, 3)]$, $[(1, 27), (1, 19), (2, 11), (1, 3)]$, $[(3, 19), (1, 11), (1, 3)]$, $[(1, 53), (1, 13), (1, 5)]$, $[(1, 45), (1, 21), (1, 5)]$, $[(1, 37), (1, 29), (1, 5)]$, $[(1, 37), (1, 21), (1, 13)]$\\
\hline
(c) & 
$[(3, 9), (44, 1)]$, $[(2, 17), (1, 9), (28, 1)]$, $[(5, 9), (26, 1)]$, $[(2, 25), (1, 9), (12, 1)]$, $[(2, 17), (3, 9), (10, 1)]$, $[(7, 9), (8, 1)]$\\
\hline
(d) & 
$[(1, 17), (2, 9), (36, 1)]$, $[(1, 25), (2, 9), (28, 1)]$, $[(1, 33), (2, 9), (20, 1)]$, $[(3, 17), (20, 1)]$, $[(1, 17), (4, 9), (18, 1)]$, $[(1, 11), (20, 3)]$, $[(1, 41), (2, 9), (12, 1)]$, $[(1, 25), (2, 17), (12, 1)]$, $[(1, 25), (4, 9), (10, 1)]$, $[(1, 35), (12, 3)]$, $[(2, 19), (11, 3)]$, $[(1, 19), (2, 11), (10, 3)]$, $[(4, 11), (9, 3)]$, $[(1, 21), (10, 5)]$, $[(2, 13), (9, 5)]$, $[(1, 15), (8, 7)]$, $[(1, 49), (2, 9), (4, 1)]$, $[(1, 33), (2, 17), (4, 1)]$, $[(2, 25), (1, 17), (4, 1)]$, $[(1, 33), (4, 9), (2, 1)]$, $[(3, 17), (2, 9), (2, 1)]$, $[(1, 17), (6, 9)]$, $[(1, 59), (4, 3)]$, $[(1, 43), (2, 11), (2, 3)]$, $[(2, 27), (1, 11), (2, 3)]$, $[(1, 27), (2, 19), (2, 3)]$, $[(1, 27), (4, 11)]$, $[(2, 19), (3, 11)]$, $[(1, 61), (2, 5)]$, $[(1, 45), (2, 13)]$, $[(2, 29), (1, 13)]$\\
\hline
(e) & 
None\\
\hline
Open & 
$[(1, 29), (2, 21)]$\\
\hline
\caption{Non-pointed and non-perfect cases for rank 71. Each case is denoted as an array of ordered pairs $(n, r)$, indicating that there are $n$ components of rank $r$ in the universal grading of $\C$.}
\label{table: rank 71 cases}
\end{longtable}
\begin{longtable}[h]{ p{0.06\textwidth} | >{\raggedright\arraybackslash}p{0.84\textwidth} }
Rule & Cases \\
\hline
(a) & 
$[(1, 9), (64, 1)]$, $[(1, 17), (56, 1)]$, $[(2, 9), (55, 1)]$, $[(1, 25), (48, 1)]$, $[(1, 17), (1, 9), (47, 1)]$, $[(1, 33), (40, 1)]$, $[(1, 25), (1, 9), (39, 1)]$, $[(2, 17), (39, 1)]$, $[(1, 41), (32, 1)]$, $[(1, 33), (1, 9), (31, 1)]$, $[(1, 25), (1, 17), (31, 1)]$, $[(1, 49), (24, 1)]$, $[(1, 41), (1, 9), (23, 1)]$, $[(1, 33), (1, 17), (23, 1)]$, $[(2, 25), (23, 1)]$, $[(1, 57), (16, 1)]$, $[(1, 49), (1, 9), (15, 1)]$, $[(1, 41), (1, 17), (15, 1)]$, $[(1, 33), (1, 25), (15, 1)]$, $[(1, 65), (8, 1)]$, $[(1, 57), (1, 9), (7, 1)]$, $[(1, 49), (1, 17), (7, 1)]$, $[(1, 41), (1, 25), (7, 1)]$, $[(2, 33), (7, 1)]$\\
\hline
(b) & 
$[(4, 9), (37, 1)]$, $[(1, 17), (3, 9), (29, 1)]$, $[(1, 25), (1, 17), (1, 9), (22, 1)]$, $[(1, 25), (3, 9), (21, 1)]$, $[(2, 17), (2, 9), (21, 1)]$, $[(6, 9), (19, 1)]$, $[(1, 33), (1, 17), (1, 9), (14, 1)]$, $[(1, 33), (3, 9), (13, 1)]$, $[(1, 25), (1, 17), (2, 9), (13, 1)]$, $[(3, 17), (1, 9), (13, 1)]$, $[(1, 17), (5, 9), (11, 1)]$, $[(1, 35), (1, 11), (9, 3)]$, $[(1, 27), (1, 19), (9, 3)]$, $[(1, 19), (3, 11), (7, 3)]$, $[(1, 41), (1, 17), (1, 9), (6, 1)]$, $[(1, 33), (1, 25), (1, 9), (6, 1)]$, $[(1, 41), (3, 9), (5, 1)]$, $[(1, 33), (1, 17), (2, 9), (5, 1)]$, $[(2, 25), (2, 9), (5, 1)]$, $[(1, 25), (2, 17), (1, 9), (5, 1)]$, $[(4, 17), (5, 1)]$, $[(1, 25), (1, 17), (3, 9), (4, 1)]$, $[(1, 25), (5, 9), (3, 1)]$, $[(2, 17), (4, 9), (3, 1)]$, $[(8, 9), (1, 1)]$, $[(1, 23), (1, 15), (5, 7)]$, $[(1, 45), (1, 13), (3, 5)]$, $[(1, 37), (1, 21), (3, 5)]$, $[(1, 29), (1, 21), (1, 13), (2, 5)]$, $[(1, 29), (3, 13), (1, 5)]$, $[(1, 59), (1, 11), (1, 3)]$, $[(1, 51), (1, 19), (1, 3)]$, $[(1, 43), (1, 27), (1, 3)]$, $[(1, 43), (1, 19), (1, 11)]$, $[(1, 35), (1, 27), (1, 11)]$\\
\hline
(c) & 
$[(3, 9), (46, 1)]$, $[(2, 17), (1, 9), (30, 1)]$, $[(5, 9), (28, 1)]$, $[(2, 25), (1, 9), (14, 1)]$, $[(2, 17), (3, 9), (12, 1)]$, $[(7, 9), (10, 1)]$\\
\hline
(d) & 
$[(1, 17), (2, 9), (38, 1)]$, $[(1, 25), (2, 9), (30, 1)]$, $[(1, 33), (2, 9), (22, 1)]$, $[(3, 17), (22, 1)]$, $[(1, 17), (4, 9), (20, 1)]$, $[(1, 19), (18, 3)]$, $[(2, 11), (17, 3)]$, $[(1, 41), (2, 9), (14, 1)]$, $[(1, 25), (2, 17), (14, 1)]$, $[(1, 25), (4, 9), (12, 1)]$, $[(1, 13), (12, 5)]$, $[(1, 43), (10, 3)]$, $[(1, 27), (2, 11), (8, 3)]$, $[(2, 19), (1, 11), (8, 3)]$, $[(5, 11), (6, 3)]$, $[(1, 49), (2, 9), (6, 1)]$, $[(1, 33), (2, 17), (6, 1)]$, $[(2, 25), (1, 17), (6, 1)]$, $[(1, 33), (4, 9), (4, 1)]$, $[(1, 17), (6, 9), (2, 1)]$, $[(3, 15), (4, 7)]$, $[(2, 29), (3, 5)]$, $[(1, 37), (2, 13), (2, 5)]$, $[(3, 21), (2, 5)]$, $[(2, 21), (2, 13), (1, 5)]$, $[(1, 21), (4, 13)]$, $[(2, 35), (1, 3)]$, $[(1, 51), (2, 11)]$, $[(1, 35), (2, 19)]$\\
\hline
(e) & 
$[(3, 17), (2, 9), (4, 1)]$, $[(2, 27), (1, 19)]$\\
\hline
Open & 
$[(1, 31), (6, 7)]$, $[(1, 53), (4, 5)]$, $[(1, 67), (2, 3)]$\\
\hline
\caption{Non-pointed and non-perfect cases for rank 73. Each case is denoted as an array of ordered pairs $(n, r)$, indicating that there are $n$ components of rank $r$ in the universal grading of $\C$.}
\label{table: rank 73 cases}
\end{longtable}

\end{appendices}
\end{document}